\documentclass[10pt, a4paper, reqno, captions=tableheading]{scrartcl}

\usepackage[T1]{fontenc}
\usepackage{lmodern}
\usepackage[english]{babel}
\usepackage{amsfonts}
\usepackage{amsmath} 
\usepackage{amssymb}
\usepackage{amsthm}
\usepackage{mathrsfs}
\usepackage[dvipsnames,svgnames]{xcolor}
\usepackage{booktabs}
\usepackage{float}
\usepackage[backgroundcolor=blue!10]{todonotes}
\usepackage{enumitem}
\usepackage{graphicx}
\usepackage{tikz,pgf}
\usetikzlibrary{calc}
\usetikzlibrary{decorations.markings}
\usetikzlibrary{shapes.geometric}
\usetikzlibrary{patterns}
\usetikzlibrary{intersections}
\usepackage{multirow}
\usepackage{array}
\usepackage{adjustbox}
\usepackage{calc}
\usepackage{hyperref}
\hypersetup{
    colorlinks,
    linkcolor={red!50!black},
    citecolor={blue!50!black},
    urlcolor={blue!80!black}
}

\newcommand*\phantomas[3][c]{%
\ifmmode
\makebox[\widthof{$#2$}][#1]{$#3$}%
\else
\makebox[\widthof{#2}][#1]{#3}%
\fi
}

\theoremstyle{plain}
\newtheorem{lemma}{Lemma}[section]
\newtheorem{proposition}[lemma]{\textbf{Proposition}}

\newtheorem{fact}[lemma]{\textbf{Fact}}
\newtheorem{assumption}[lemma]{\textbf{Assumption}}
\newtheorem*{reduction}{\textbf{Reduction}}
\theoremstyle{definition}
\newtheorem{definition}[lemma]{\textbf{Definition}}

\newtheorem*{notation}{\textbf{Notation}}
\newtheorem{remark}[lemma]{Remark}

\newcommand{\R}{\mathbb{R}}
\newcommand{\C}{\mathbb{C}}

\newcommand{\p}{\mathbb{P}}

\newcommand{\Bond}{\mathbf{B}}
\newcommand{\ObondA}{X}
\newcommand{\ObondB}{Y}
\newcommand{\ObondC}{Z}
\newcommand{\ObondAc}{\bar\ObondA}
\newcommand{\ObondBc}{\bar\ObondB}
\newcommand{\ObondCc}{\bar\ObondC}

\newcommand{\oct}{\mathrm{oct}}
\newcommand{\GOct}{\mathbb{G}_{\oct}}
\newcommand{\GOctDir}{\vec{\mathbb{G}}_{\oct}}
\newcommand{\GEdg}{\mathbb{G}_{\mathrm{edg}}}

\newcommand{\M}{\overline{\mscr{M}}}

\newcommand{\mscr}{\mathscr}
\newcommand{\mcal}{\mathcal}

\renewcommand{\hat}{\widehat}

\newcommand{\cuts}[1]{\mathrm{#1}}
\newcommand{\divou}[1]{D_\cuts{ou}^{#1}}
\newcommand{\diveu}[1]{D_\cuts{eu}^{#1}}
\newcommand{\divom}[1]{D_\cuts{om}^{#1}}
\newcommand{\divem}[1]{D_\cuts{em}^{#1}}

\newcommand{\casesQ}[1]{$\mathfrak{#1}$}
\newcommand{\caseG}{\casesQ{g}}
\newcommand{\caseO}{\casesQ{o}}
\newcommand{\caseE}{\casesQ{e}}
\newcommand{\caseR}{\casesQ{r}}
\newcommand{\caseL}{\casesQ{l}}

\colorlet{ecol}{black!50!white}
\definecolor{colR}{rgb}{.932,.172,.172} 
\definecolor{colB}{rgb}{.255,.41,.884} 

\colorlet{col1}{DarkGoldenrod}
\colorlet{col2}{DarkKhaki}
\colorlet{col3}{MediumPurple}
\colorlet{col4}{IndianRed}
\colorlet{col5}{SkyBlue!90!black}
\colorlet{col6}{MediumSeaGreen}

\colorlet{cole1}{CornflowerBlue}
\colorlet{cole2}{CornflowerBlue!60!black}
\colorlet{cole3}{SeaGreen} 
\colorlet{cole4}{Coral}
\colorlet{cole5}{DarkSeaGreen} 
\colorlet{cole6}{Coral!70!black} 

\colorlet{colds}{MediumSeaGreen} 
\colorlet{coldo}{IndianRed} 
\colorlet{colpb}{CornflowerBlue} 
\colorlet{colpbfixed}{CornflowerBlue} 

\tikzstyle{vertex}=[circle, draw, fill=black, inner sep=0pt, minimum size=4pt]
\tikzstyle{svertex}=[circle, draw, fill=black, inner sep=0pt, minimum size=2pt] 
\tikzstyle{edge}=[line width=1.5pt,ecol]
\tikzstyle{redge}=[edge,colR]
\tikzstyle{bedge}=[edge,colB]

\tikzstyle{labelsty}=[font=\scriptsize]

\tikzstyle{dedge}=[
	edge,%
	postaction={%
		decoration={markings,
			mark=at position 0.5
			with {
				\node[dart,fill,inner sep=1pt,minimum size=6pt,transform shape] (0,0) {};	
			},
		},%
	decorate}%
]

\tikzstyle{sdedge}=[
	edge,%
	postaction={%
		decoration={markings,
			mark=at position 0.5
			with {
				\node[dart,fill,inner sep=1pt,minimum size=4pt,transform shape] (0,0) {};
			},
		},%
	decorate}%
]

\newcommand{\quadrefA}{\tikz[baseline={(0,-0.1)}]{\node[rectangle,thick,draw=black,minimum size=0.4cm] at (0,0) {\textnormal{1}};}}
\newcommand{\quadrefB}{\tikz[baseline={(0,-0.1)}]{\node[rectangle,thick,draw=black,minimum size=0.4cm] at (0,0) {\textnormal{2}};}}
\newcommand{\quadrefC}{\tikz[baseline={(0,-0.1)}]{\node[rectangle,thick,draw=black,minimum size=0.4cm] at (0,0) {\textnormal{3}};}}
\newcommand{\quadrefD}{\tikz[baseline={(0,-0.1)}]{\node[rectangle,thick,draw=black,minimum size=0.4cm] at (0,0) {\textnormal{4}};}}
\newcommand{\quadrefE}{\tikz[baseline={(0,-0.1)}]{\node[rectangle,thick,draw=black,minimum size=0.4cm] at (0,0) {\textnormal{5}};}}
\newcommand{\quadrefF}{\tikz[baseline={(0,-0.1)}]{\node[rectangle,thick,draw=black,minimum size=0.4cm] at (0,0) {\textnormal{6}};}}
\newcommand{\quadrefV}{\tikz[baseline={(0,-0.1)}]{\node[rectangle,thick,draw=black,minimum size=0.4cm] at (0,0) {$v$};}}
\newcommand{\quadrefW}{\tikz[baseline={(0,-0.1)}]{\node[rectangle,thick,draw=black,minimum size=0.4cm] at (0,0) {$w$};}}

\tikzstyle{pvertex}=[circle, draw, fill=black, inner sep=0pt, minimum size=1.75pt]
\tikzstyle{pedge}=[line width=1pt,ecol]
\tikzstyle{pdedge}=[
	pedge,%
	postaction={%
		decoration={markings,
			mark=at position 0.5
			with {
				\node[dart,fill,inner sep=1pt,minimum size=4pt,transform shape] (0,0) {};
			},
		},%
	decorate}%
]
\newcommand{\pbondar}{\tikz[scale=0.75]{ 
	\node[pvertex] (a) at (0,0) {};
	\node[pvertex] (b) at (1,0) {};
	\node[pvertex] (c) at (1,1) {};
	\node[pvertex] (d) at (0,1) {};
	\node[pvertex] (e) at (0.5,0.5) {};
	\draw[pedge] (a)edge(b) (b)edge(c) (c)edge(d) (d)edge(a);
	\draw[pedge] (b)--(e) (d)--(e);
	\draw[pdedge,colpb] (a)--(e);
	\draw[pdedge,colpb] (e)--(c);
}}
\newcommand{\pbondal}{\tikz[scale=0.75]{
	\node[pvertex] (a) at (0,0) {};
	\node[pvertex] (b) at (1,0) {};
	\node[pvertex] (c) at (1,1) {};
	\node[pvertex] (d) at (0,1) {};
	\node[pvertex] (e) at (0.5,0.5) {};
	\draw[pedge] (a)edge(b) (b)edge(c) (c)edge(d) (d)edge(a);
	\draw[pedge] (b)--(e) (d)--(e);
	\draw[pdedge,colpb] (e)--(a);
	\draw[pdedge,colpb] (c)--(e);
}}
\newcommand{\pbonddr}{\tikz[scale=0.75]{
	\node[pvertex] (a) at (0,0) {};
	\node[pvertex] (b) at (1,0) {};
	\node[pvertex] (c) at (1,1) {};
	\node[pvertex] (d) at (0,1) {};
	\node[pvertex] (e) at (0.5,0.5) {};
	\draw[pedge] (a)edge(b) (b)edge(c) (c)edge(d) (d)edge(a);
	\draw[pedge] (a)--(e) (c)--(e);
	\draw[pdedge,colpb] (d)--(e);
	\draw[pdedge,colpb] (e)--(b);
}}
\newcommand{\pbonddl}{\tikz[scale=0.75]{
	\node[pvertex] (a) at (0,0) {};
	\node[pvertex] (b) at (1,0) {};
	\node[pvertex] (c) at (1,1) {};
	\node[pvertex] (d) at (0,1) {};
	\node[pvertex] (e) at (0.5,0.5) {};
	\draw[pedge] (a)edge(b) (b)edge(c) (c)edge(d) (d)edge(a);
	\draw[pedge] (a)--(e) (c)--(e);
	\draw[pdedge,colpb] (e)--(d);
	\draw[pdedge,colpb] (b)--(e);
}}
\newcommand{\pbondac}{\tikz[scale=0.75]{
	\node[pvertex] (a) at (0,0) {};
	\node[pvertex] (b) at (1,0) {};
	\node[pvertex] (c) at (1,1) {};
	\node[pvertex] (d) at (0,1) {};
	\node[pvertex] (e) at (0.5,0.5) {};
	\draw[pedge] (a)edge(b) (b)edge(c) (c)edge(d) (d)edge(a);
	\draw[pedge] (b)--(e) (d)--(e);
	\draw[pdedge,colpb] (a)--(e);
	\draw[pdedge,colpb] (c)--(e);
}}
\newcommand{\pbondao}{\tikz[scale=0.75]{
	\node[pvertex] (a) at (0,0) {};
	\node[pvertex] (b) at (1,0) {};
	\node[pvertex] (c) at (1,1) {};
	\node[pvertex] (d) at (0,1) {};
	\node[pvertex] (e) at (0.5,0.5) {};
	\draw[pedge] (a)edge(b) (b)edge(c) (c)edge(d) (d)edge(a);
	\draw[pedge] (b)--(e) (d)--(e);
	\draw[pdedge,colpb] (e)--(a);
	\draw[pdedge,colpb] (e)--(c);
}}
\newcommand{\pbonddc}{\tikz[scale=0.75]{
	\node[pvertex] (a) at (0,0) {};
	\node[pvertex] (b) at (1,0) {};
	\node[pvertex] (c) at (1,1) {};
	\node[pvertex] (d) at (0,1) {};
	\node[pvertex] (e) at (0.5,0.5) {};
	\draw[pedge] (a)edge(b) (b)edge(c) (c)edge(d) (d)edge(a);
	\draw[pedge] (a)--(e) (c)--(e);
	\draw[pdedge,colpb] (d)--(e);
	\draw[pdedge,colpb] (b)--(e);
}}
\newcommand{\pbonddo}{\tikz[scale=0.75]{
	\node[pvertex] (a) at (0,0) {};
	\node[pvertex] (b) at (1,0) {};
	\node[pvertex] (c) at (1,1) {};
	\node[pvertex] (d) at (0,1) {};
	\node[pvertex] (e) at (0.5,0.5) {};
	\draw[pedge] (a)edge(b) (b)edge(c) (c)edge(d) (d)edge(a);
	\draw[pedge] (a)--(e) (c)--(e);
	\draw[pdedge,colpb] (e)--(d);
	\draw[pdedge,colpb] (e)--(b);
}}
\tikzstyle{ppedge}=[line width=0.75pt,ecol]
\tikzstyle{ppdedge}=[
	ppedge,%
	postaction={%
		decoration={markings,
			mark=at position 0.75
			with {
				\node[dart,fill,inner sep=0.6pt,minimum size=2pt,transform shape] (0,0) {};
			},
		},%
	decorate}%
]
\newcommand{\Pbondar}{\!\!\tikz[scale=0.25]{ 
	\coordinate (a) at (0,0);
	\coordinate (b) at (1,0);
	\coordinate (c) at (1,1);
	\coordinate (d) at (0,1);
	\coordinate (e) at (0.5,0.5);
	\draw[ppdedge,black] (a)--(e);
	\draw[ppdedge,black] (e)--(c);
}}
\newcommand{\Pbondal}{\!\!\tikz[scale=0.25]{ 
	\coordinate (a) at (0,0);
	\coordinate (b) at (1,0);
	\coordinate (c) at (1,1);
	\coordinate (d) at (0,1);
	\coordinate (e) at (0.5,0.5);
	\draw[ppdedge,black] (e)--(a);
	\draw[ppdedge,black] (c)--(e);
}}
\newcommand{\Pbonddr}{\tikz[scale=0.25]{ 
	\coordinate (a) at (0,0);
	\coordinate (b) at (1,0);
	\coordinate (c) at (1,1);
	\coordinate (d) at (0,1);
	\coordinate (e) at (0.5,0.5);
	\draw[ppdedge,black] (d)--(e);
	\draw[ppdedge,black] (e)--(b);
}}
\newcommand{\Pbonddl}{\tikz[scale=0.25]{ 
	\coordinate (a) at (0,0);
	\coordinate (b) at (1,0);
	\coordinate (c) at (1,1);
	\coordinate (d) at (0,1);
	\coordinate (e) at (0.5,0.5);
	\draw[ppdedge,black] (e)--(d);
	\draw[ppdedge,black] (b)--(e);
}}
\newcommand{\Pbondac}{\!\!\tikz[scale=0.25]{ 
	\coordinate (a) at (0,0);
	\coordinate (b) at (1,0);
	\coordinate (c) at (1,1);
	\coordinate (d) at (0,1);
	\coordinate (e) at (0.5,0.5);
	\draw[ppdedge,black] (a)--(e);
	\draw[ppdedge,black] (c)--(e);
}}
\newcommand{\Pbondao}{\!\!\tikz[scale=0.25]{ 
	\coordinate (a) at (0,0);
	\coordinate (b) at (1,0);
	\coordinate (c) at (1,1);
	\coordinate (d) at (0,1);
	\coordinate (e) at (0.5,0.5);
	\draw[ppdedge,black] (e)--(a);
	\draw[ppdedge,black] (e)--(c);
}}
\newcommand{\Pbonddc}{\tikz[scale=0.25]{ 
	\coordinate (a) at (0,0);
	\coordinate (b) at (1,0);
	\coordinate (c) at (1,1);
	\coordinate (d) at (0,1);
	\coordinate (e) at (0.5,0.5);
	\draw[ppdedge,black] (b)--(e);
	\draw[ppdedge,black] (d)--(e);
}}
\newcommand{\Pbonddo}{\tikz[scale=0.25]{ 
	\coordinate (a) at (0,0);
	\coordinate (b) at (1,0);
	\coordinate (c) at (1,1);
	\coordinate (d) at (0,1);
	\coordinate (e) at (0.5,0.5);
	\draw[ppdedge,black] (e)--(b);
	\draw[ppdedge,black] (e)--(d);
}}
\newcommand{\Pyr}{\mathbf{P}}

\title{Combinatorics of Bricard's octahedra}
\date{}

\author{%
Matteo Gallet$^{\ast,\diamond}$%
\and
Georg Grasegger$^{\ast,\triangleright}$%
\and
Jan Legersk\'y$^{\circ}$%
\and
Josef Schicho$^{\ast, \circ}$%
}
\renewcommand{\thefootnote}{\fnsymbol{footnote}}

\begin{document}
\maketitle
\footnotetext{\hspace{0.15cm}$^\ast$ Supported by the Austrian Science Fund (FWF): W1214-N15, project DK9.\\%
$^\circ$ Supported by the Austrian Science Fund (FWF): P31061.\\%
$^\diamond$ Supported by the Austrian Science Fund (FWF): Erwin Schr\"odinger Fellowship J4253.\\%
$^\triangleright$ Supported by the Austrian Science Fund (FWF): P31888.%
}
\begin{abstract}
 We re-prove the classification of flexible octahedra, 
 obtained by Bricard at the beginning of the XX century, 
 by means of combinatorial objects satisfying some elementary rules. 
 The explanations of these rules rely on the use of a well-known creation of modern algebraic geometry, 
 the moduli space of stable rational curves with marked points, for the description of configurations of graphs on the sphere.
 Once one accepts the objects and the rules, the classification becomes elementary (though not trivial) 
 and can be enjoyed without the need of a very deep background on the topic. 
\end{abstract}
\renewcommand{\thefootnote}{\arabic{footnote}}

\section{Introduction}
\label{introduction}
 
Cauchy proved \cite{Cauchy1813} that every convex polyhedron is rigid, in the sense that it cannot move keeping the shape of its faces. 
Hence flexible polyhedra must be concave, and indeed Bricard discovered 
\cite{Bricard1897, Bricard1926, Bricard1927} three families of concave flexible octahedra. 
Lebesgue lectured about Bricard's construction in 1938/39 \cite{Lebesgue1967}, 
and Bennett discussed flexible octahedra in his work~\cite{Bennett1912}. 
In recent years, there has been renewed interest in the topic; 
see the works of Baker \cite{Baker1980, Baker1996, Baker2009}, Stachel \cite{Stachel1987, Stachel2014, Stachel2015}, Nawratil \cite{Nawratil2010, Nawratil2018}, 
and others \cite{Connelly1978, Bushmelev1990, Mikhalev2002, Alexandrov2010, Alexandrov2011, Nelson2010, Nelson2012, Chen2012}.
The goal of this paper is to re-prove Bricard's result by employing modern techniques in algebraic geometry that hopefully may be applied to more general situations.

The three families of flexible octahedra are the following (see Figures~\ref{figure:bricard_typesIandII} and~\ref{figure:bricard_typeIII}):
\begin{description}
 \item[Type I.] Octahedra whose vertices form three pairs of points symmetric with respect to a line.
 \item[Type II.] Octahedra whose vertices are given by two pairs of points symmetric with respect to a plane passing through the last two vertices.
 \item[Type III.] Octahedra all of whose pyramids\footnote{Here by ``pyramid'' we mean a $4$-tuple of edges sharing a vertex. See Definition~\ref{definition:pyramid} for formal specification and notation.} have the following property: the two pairs of opposite angles%
 \footnote{Here by ``angle'' of a pyramid we mean the angle formed by two concurrent edges belonging to the same face.} %
 are constituted of angles that are either both equal or both supplementary; 
 moreover, we ask the lengths~$\ell_{ij}$ of the edges%
 \footnote{Here we label the vertices of the octahedron by the numbers $\{1, \dotsc, 6\}$.} %
 to satisfy three linear equations of the form:
\begin{align*}
 \eta_{35} \, \ell_{35} + \eta_{45} \, \ell_{45} + \eta_{46} \, \ell_{46} + \eta_{36} \, \ell_{36} &= 0 \,, \\
 \eta_{14} \, \ell_{14} + \eta_{24} \, \ell_{24} + \eta_{23} \, \ell_{23} + \eta_{13} \, \ell_{13} &= 0 \,, \\
 \eta_{15} \, \ell_{15} + \eta_{25} \, \ell_{25} + \eta_{26} \, \ell_{26} + \eta_{16} \, \ell_{16} &= 0 \,,
\end{align*}
where $\eta_{ij} \in \{1,-1\}$ and in each equation we have exactly two positive~$\eta_{ij}$ and two negative ones.
\end{description}
\begin{figure}
\centering
 \tikz[baseline={(0,0)}]{\node at (0,0) {\includegraphics[width=.4\textwidth]{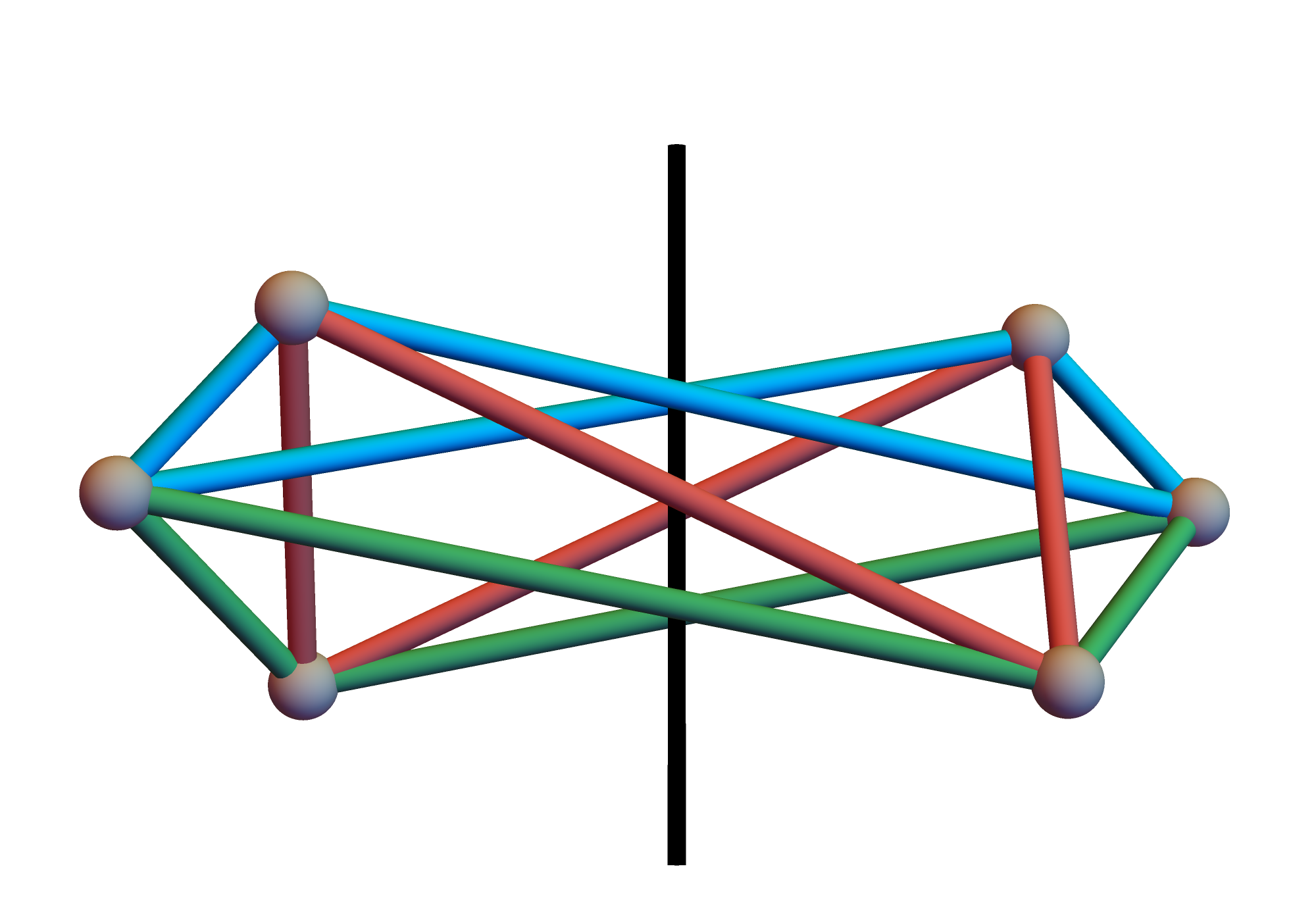}}}
 \tikz[baseline={(0,0)}]{\node at (0,0) {\includegraphics[width=.4\textwidth]{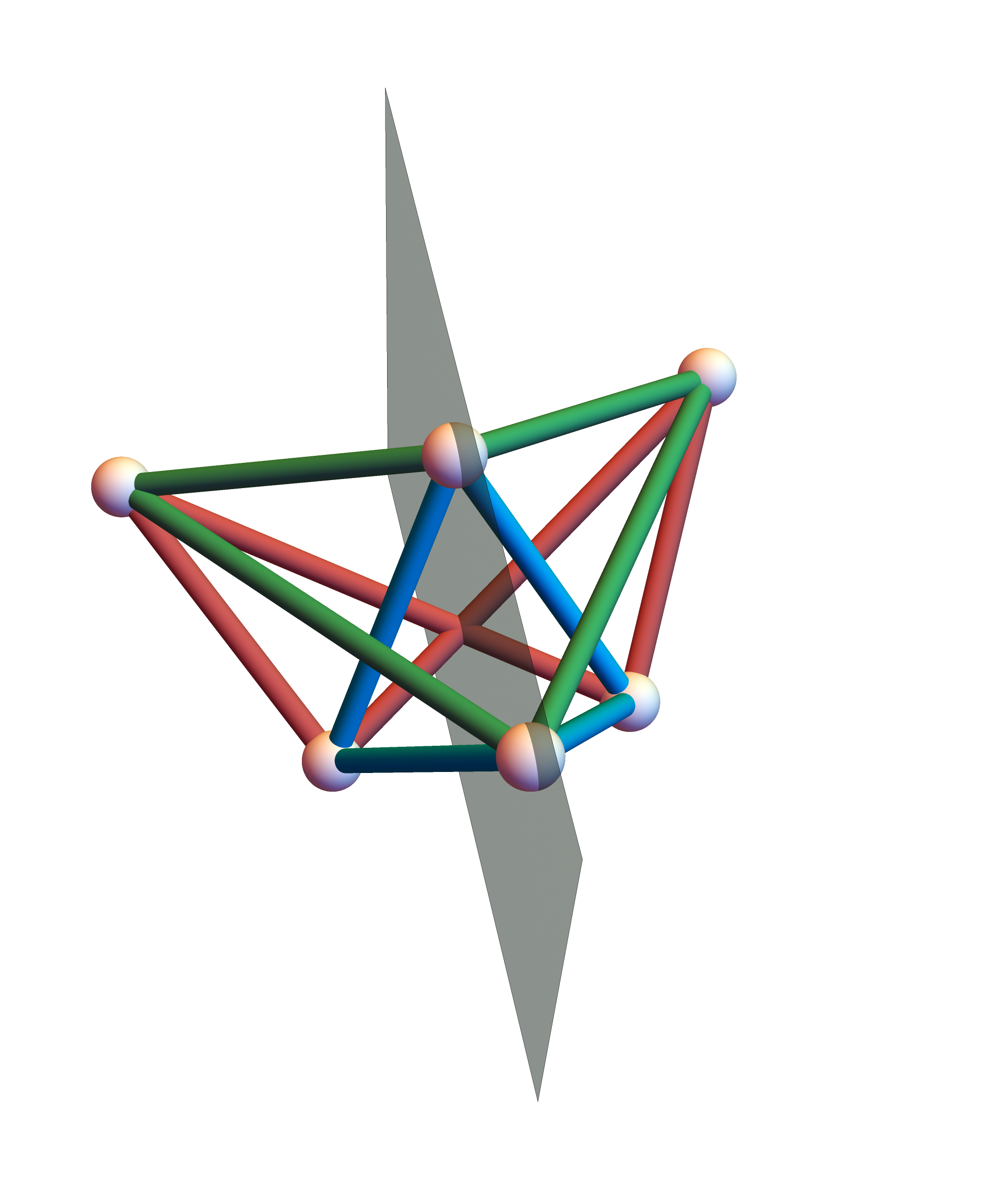}}}
 \caption{Flexible octahedra of Type I and Type II found by Bricard.}
 \label{figure:bricard_typesIandII}
\end{figure}

\begin{figure}
\centering
 \rotatebox{90}{\includegraphics[width=.4\textwidth]{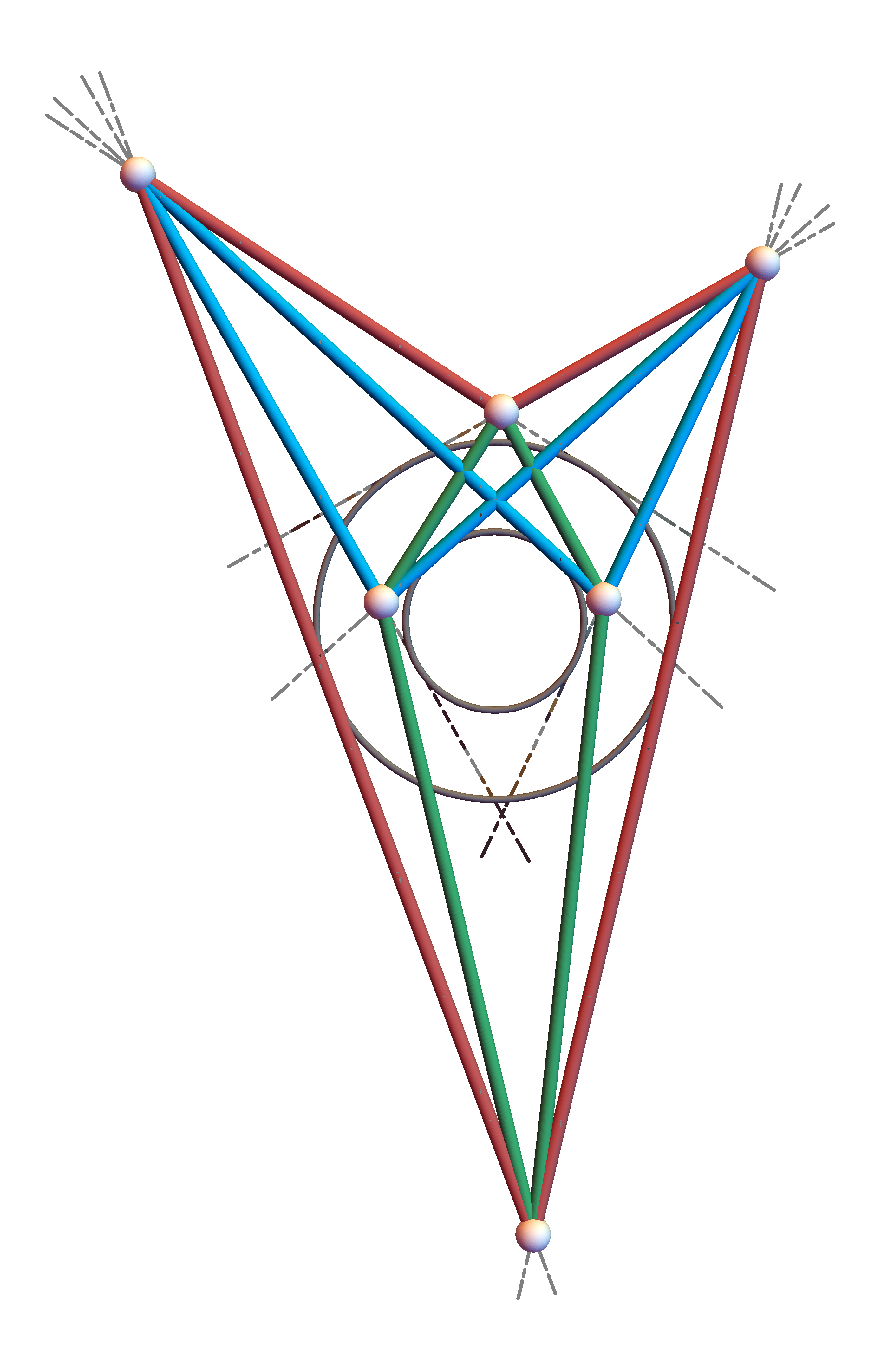}}
 \caption{Bricard found that certain flexible octahedra of Type III admit the following construction: in the plane, pick two points~$A_1$ and $A_2$ and two circles, and draw the tangent lines to the circles passing through the points~$A_i$. These lines determine four other points $B_1$, $B_2$, $C_1$ and $C_2$, which together with $A_1$ and $A_2$ define a flat realization of a Type III octahedron.}
 \label{figure:bricard_typeIII}
\end{figure}

Animations of the motions of each of the three families can be found at \\
\makebox[\textwidth]{\url{https://jan.legersky.cz/project/bricard_octahedra/}.}

The fact that an octahedron with line symmetry is flexible is well-known 
(see, for example, \cite[Section~5]{Schulze2010}) 
and follows from a count of the free parameters versus the number of equations imposed by the edges.
A similar argument also shows that plane-symmetric octahedra are flexible.
Proving that Type III octahedra are flexible is more complicated, and for this we refer to the proof given by Lebesgue (see~\cite{Lebesgue1967}).

The technique we adopt to analyze flexible octahedra is to reduce to the case of flexible spherical linkages,
and to use the tools developed in our previous work~\cite{Gallet2019} to derive the classification. 
More precisely, our work consists of two parts: in the first part, we prove some elementary facts about flexible octahedra
and we provide their classification by using combinatorial objects called \emph{octahedral} and \emph{pyramidal bonds}, and rules that relate them; 
in the second part, we explain the rules via the theory developed in~\cite{Gallet2019} on flexible graphs on the sphere.
The first part is rather nontechnical and aimed at a general public; 
the second part involves more technicalities and requires some acquaintance with the material from~\cite{Gallet2019} for the detailed justification of the arguments.
By splitting the text in such a way we hope to widen the possible readership to those readers 
who may not be extremely interested in the specific details of the algebro-geometric part of the proof but are fascinated by this old topic;
at the same time, we hope to convince them that the techniques we introduce by employing objects from modern algebraic geometry
may be well-suited for these classical questions, and may have the chance to shed light on related topics that have not been fully investigated yet.

The paper is structured as follows. 
Section~\ref{elementary} reports elementary results on flexible octahedra.
Section~\ref{classification} provides the classification of flexible octahedra by introducing octahedral and pyramidal bonds,
and by setting up, in an axiomatic way, the rules that guide their behavior.
The constraints imposed by these rules are then used to classify flexible octahedra.
Section~\ref{reduction} describes how realizations of octahedra in the space determine realizations on the unit sphere
of the graph whose vertices are the edges of the octahedron, and whose edges encode the fact that edges of the octahedron lie on the same face.
This opens the way to the use of the methods developed by the authors in~\cite{Gallet2019}, namely to the study of flexible graphs on the sphere.
Section~\ref{justifications} provides the precise background for the notion of bonds,
and justifies the rules in Section~\ref{classification} via the techniques from~\cite{Gallet2019}.

\section{Elementary results on flexible octahedra}
\label{elementary}

In this section we collect some results about the motions of octahedra, which we use in later sections.
The results are known and elementary; we report them here mainly for self-containedness.

We represent the combinatorial structure of an octahedron by the graph~$\GOct$ with vertices $\{1, \dotsc, 6\}$ and edge set~$E_{\oct}$
given by all unordered pairs~$\{ i, j \}$ where $i, j \in \{1, \dotsc, 6\}$ except for $\{1,2\}$, $\{3,4\}$, and $\{5,6\}$ (see Figure~\ref{figure:Goct}).
\begin{figure}[ht]
	\centering
		\begin{tikzpicture}[scale=1.5]
			\coordinate (o) at (0,0);
				\node[vertex,label={[labelsty]right:$1$}] (1) at (1,0) {};
				\node[vertex,label={[labelsty,label distance=-2pt]right:$4$},rotate around={60:(o)}] (4) at (1) {};
				\node[vertex,label={[labelsty]right:$5$},rotate around={120:(o)}] (5) at (1) {};
				\node[vertex,label={[labelsty]left:$2$}] (2) at (-1,0) {};
				\node[vertex,label={[labelsty]left:$3$},rotate around={60:(o)}] (3) at (2) {};	
				\node[vertex,label={[labelsty,label distance=-2pt]left:$6$},rotate around={120:(o)}] (6) at (2) {};
				\draw[edge] (1) -- (6);
				\draw[edge] (1) -- (3);
				\draw[edge] (4) -- (1);
				\draw[edge] (5) -- (1); 
				\draw[edge] (6) -- (3);
				\draw[edge] (4) -- (6);
				\draw[edge] (6) -- (2);
				\draw[edge] (3) -- (5);
				\draw[edge] (3) -- (2);
				\draw[edge] (5) -- (4);
				\draw[edge] (2) -- (4);
				\draw[edge] (2) -- (5);
		\end{tikzpicture}
 \caption{The octahedral graph~$\GOct$.}
 \label{figure:Goct}
\end{figure}
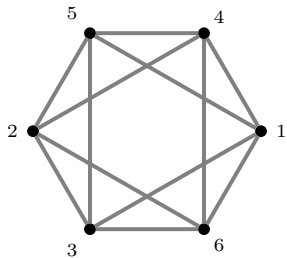

\begin{definition}
A \emph{realization} of an octahedron is a map $\rho \colon \{1, \dotsc, 6\} \longrightarrow \R^3$.
A \emph{labeling} is a map $\lambda \colon E_{\oct} \longrightarrow \R_{>0}$; we use the notation $\lambda_{\{i,j\}}$ for~$\lambda(\{i,j\})$. 
A realization~$\rho$ is \emph{compatible} with a labeling~$\lambda$ if $\left\| \rho(i) - \rho(j) \right\| = \lambda_{\{i,j\}}$ for all $\{i,j\} \in E_{\oct}$. 
A labeling~$\lambda$ is called \emph{flexible} if there exist infinitely many non-congruent realizations compatible with~$\lambda$. 
Two realizations~$\rho_1$ and~$\rho_2$ are called \emph{congruent} if there exists an isometry~$\sigma$ of~$\R^3$ such that $\rho_1 = \sigma \circ \rho_2$.
\end{definition}

We formalize the intuitive notion of \emph{motion} of an octahedron by exploiting the fact 
that being compatible with a labeling imposes polynomial constraints on realizations.
Notice that, since we can always apply rotations or translations to a given realization to obtain another compatible realization, 
once we have a compatible realization, we actually have a $6$-dimensional set of compatible ones.
Therefore motions are asked to be $7$-dimensional objects, in order to encode octahedra
that move with \emph{one degree of freedom}, up to rotations and translations. 

\begin{definition}
 Given a flexible labeling of an octahedron, the set of realizations compatible with it is an algebraic variety.
 Indeed, all realizations are a real vector space of dimension $3 \times 6 = 18$,
 and being compatible with a given labeling imposes polynomial conditions.
 A \emph{motion} of the octahedron is a $7$-dimensional irreducible component of this algebraic variety.
\end{definition}

The goal of this paper is to classify flexible octahedra satisfying the following genericity assumption.

\newcounter{tmp}
\begingroup
\setcounter{tmp}{\value{lemma}}
\setcounter{lemma}{6}
\renewcommand\thelemma{\Alph{lemma}}
\begin{assumption}
\label{assumption}
 No two faces of the octahedron are coplanar for a general realization in a motion.
\end{assumption}
\endgroup
\setcounter{lemma}{\thetmp}

\begin{remark}
\label{remark:degenerate}
 An elementary, but tedious, inspection of all possibilities shows that if a triangle in an octahedron degenerates, 
 then the octahedron does not have flexible labelings.
 Moreover, as a corollary of the assumption we have that no two vertices of the octahedron coincide for a general realization in a motion.
\end{remark}

A key object in our proof of the classification of octahedra are \emph{pyramids}.

\begin{definition}
\label{definition:pyramid}
 Pyramids are subgraphs of~$\GOct$ induced by a vertex and its four neighbors.
 The pyramid determined by~$v$ is denoted by~\quadrefV.
 Realizations of pyramids, their congruence, and flexibility are defined analogously as for octahedra.
 The same happens for motions.
 Here we make a similar request as in Assumption~\ref{assumption}, namely we do not allow motions
 for which two triangular faces of a pyramid are coplanar in a general realization of that motion.
\end{definition}

Flexible pyramids come in four families (here by an \emph{angle} of a pyramid~\quadrefV\ 
we mean an angle between edges of the form~$\{u,v\}$ and~$\{w,v\}$, where $u$ and~$w$ are neighbors):
\begin{description}
 \item[deltoids:] here two disjoint pairs of adjacent angles are constituted of angles that are either both equal or both supplementary;
 \item[rhomboids:] here the two pairs of opposite angles are constituted of angles that are either both equal or both supplementary;
 \item[lozenges:] here none of the angles equals $\pi/2$ and either all angles are equal, or two are equal and the other two are each supplementary to the first two\footnote{When all the angles are $\pi/2$, any motion of such a pyramid is degenerate, namely two triangular faces stay coplanar during the motion.};
  \item[general:] here are flexible pyramids such that not all angles equal~$\pi/2$ and not falling in the one of the previous families.
\end{description}

\begin{figure}[ht]
 \centering
 \includegraphics[width=4cm]{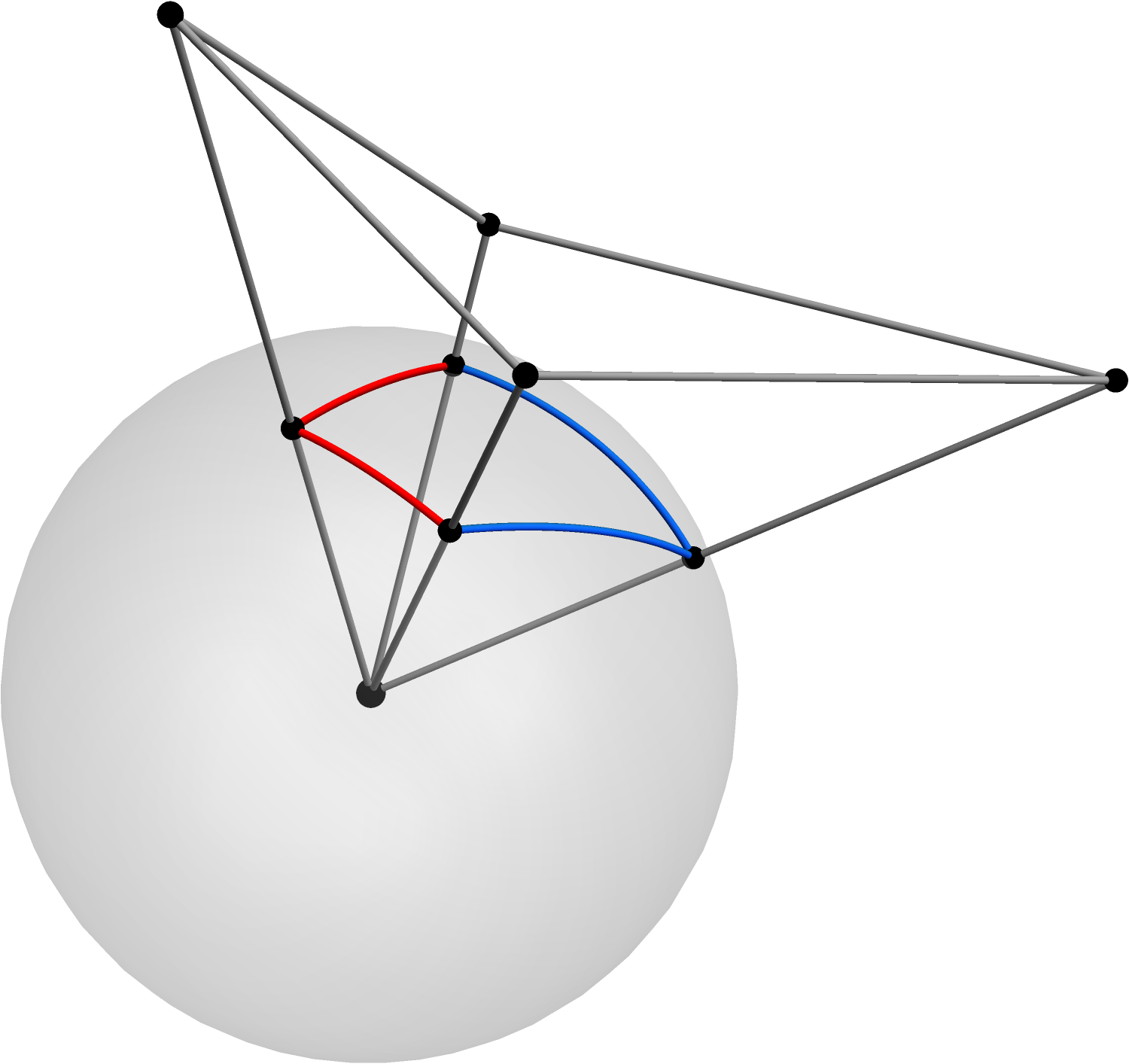}
 \quad
 \includegraphics[width=4cm]{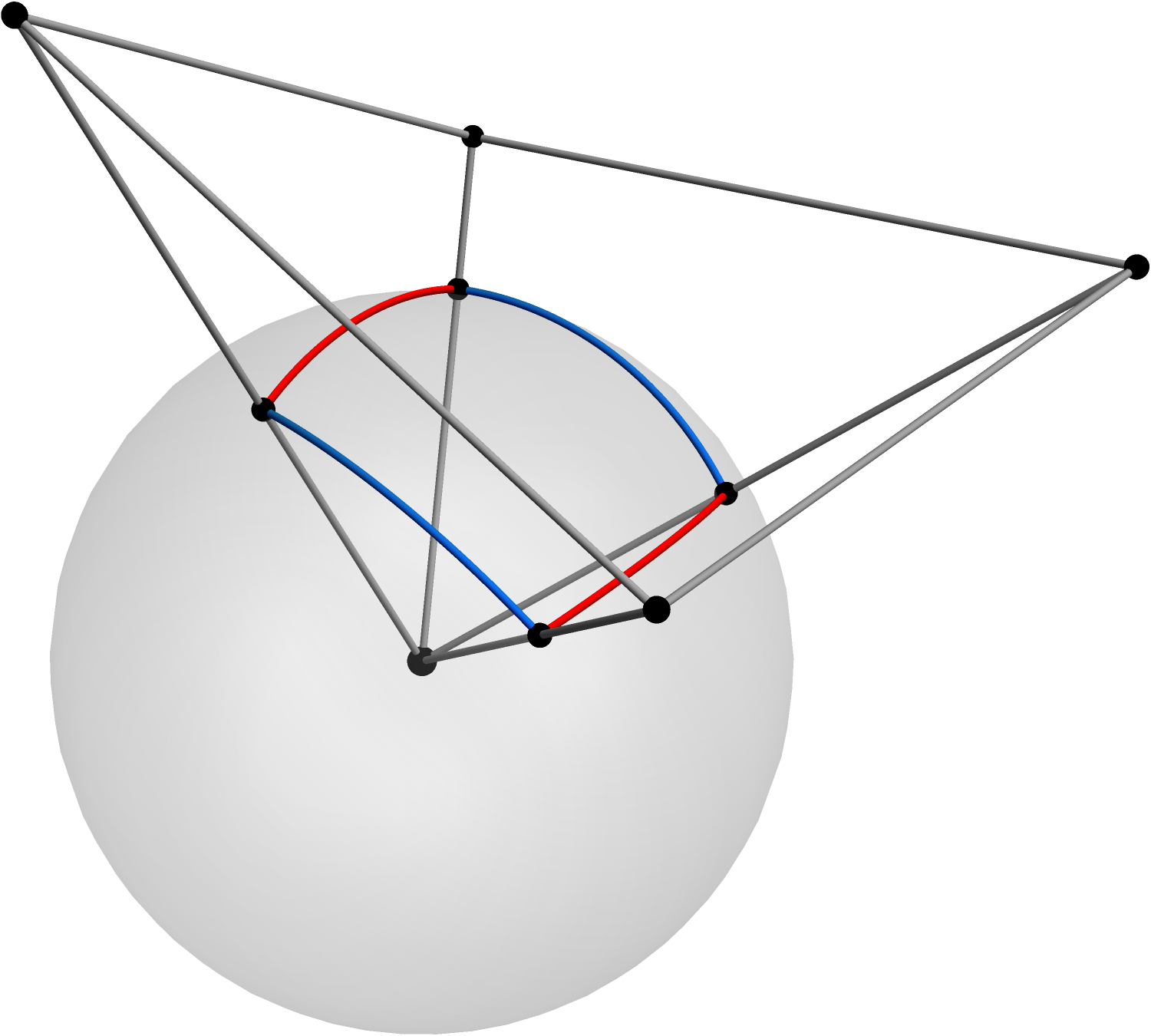}
 \caption{Examples of a deltoid (on the left) and of a rhomboid (on the right). 
 The intersection of the pyramid with the sphere highlights which pairs of angles are equal.}
 \label{figure:pyramid_families}
\end{figure}

\begin{remark}
 Given these definitions, we can say that a Type III octahedron is an octahedron where all pyramids are rhomboids or lozenges,
 and for which the equations on the edge lengths from the introduction hold.
\end{remark}

\subsection{Planar realizations of pyramids}
\label{elementary:pyramids}

Hereafter we list some elementary properties of pyramids, in particular concerning their planar (also called \emph{flat}) realizations.

\begin{fact}
 Deltoids and rhomboids have two flat realizations.
 Lozenges have three flat realizations.
 In the case of deltoids and lozenges, in these positions three vertices of the pyramid are collinear.
\end{fact}

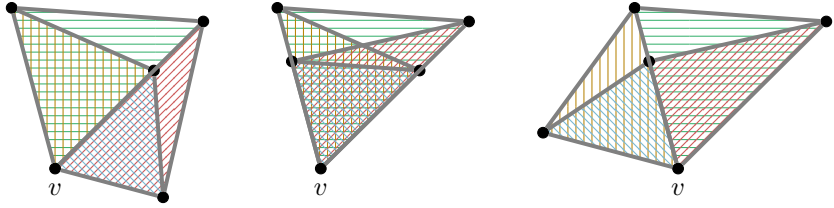
\begin{figure}[ht]
 \centering
  \begin{tikzpicture}[scale=1.3]
   \coordinate (A) at (0,0);
   \coordinate[rotate around={-30:(A)}] (B) at (-1.2,1.2);
   \coordinate[rotate around={30:(A)}] (C) at (0.8,-0.8);
   \coordinate (D) at (1.5,1.5);
   \coordinate (E) at (1,1);
   \fill[pattern=horizontal lines,pattern color=col6] (A)--(B)--(D)--cycle;
   \fill[pattern=vertical lines,pattern color=col1] (A)--(B)--(E)--cycle;
   \fill[pattern=north east lines,pattern color=col4] (A)--(C)--(D)--cycle;
   \fill[pattern=north west lines,pattern color=col5] (A)--(C)--(E)--cycle;
   \node[vertex, label={below:$v$}] (1) at (A) {};
   \node[vertex] (2) at (B) {};
   \node[vertex] (3) at (C) {};
   \node[vertex] (4) at (D) {};
   \node[vertex] (5) at (E) {};
   \draw[edge] (1)--(2);
   \draw[edge] (1)--(3);
   \draw[edge] (1)--(4);
   \draw[edge] (1)--(5);
   \draw[edge] (2)--(4);
   \draw[edge] (2)--(5);
   \draw[edge] (3)--(4);
   \draw[edge] (3)--(5);
  \end{tikzpicture}
  \qquad
  \begin{tikzpicture}[scale=1.3]
   \coordinate (A) at (0,0);
   \coordinate[rotate around={-30:(A)}] (B) at (-1.2,1.2);
   \coordinate[rotate around={-30:(A)}] (C) at (-0.8,0.8);
   \coordinate (D) at (1.5,1.5);
   \coordinate (E) at (1,1);
   \fill[pattern=horizontal lines,pattern color=col6] (A)--(B)--(D)--cycle;
   \fill[pattern=vertical lines,pattern color=col1] (A)--(B)--(E)--cycle;
   \fill[pattern=north east lines,pattern color=col4] (A)--(C)--(D)--cycle;
   \fill[pattern=north west lines,pattern color=col5] (A)--(C)--(E)--cycle;
   \node[vertex, label={below:$v$}] (1) at (A) {};
   \node[vertex] (2) at (B) {};
   \node[vertex] (3) at (C) {};
   \node[vertex] (4) at (D) {};
   \node[vertex] (5) at (E) {};
   \draw[edge] (1)--(2);
   \draw[edge] (1)--(3);
   \draw[edge] (1)--(4);
   \draw[edge] (1)--(5);
   \draw[edge] (2)--(4);
   \draw[edge] (2)--(5);
   \draw[edge] (3)--(4);
   \draw[edge] (3)--(5);
  \end{tikzpicture}
  \qquad
  \begin{tikzpicture}[scale=1.3]
   \coordinate (A) at (0,0);
   \coordinate[rotate around={-30:(A)}] (B) at (-1.2,1.2);
   \coordinate[rotate around={-30:(A)}] (C) at (-0.8,0.8);
   \coordinate (D) at (1.5,1.5);
   \coordinate[rotate around={-60:(A)}] (E) at (-1,-1);
   \fill[pattern=horizontal lines,pattern color=col6] (A)--(B)--(D)--cycle;
   \fill[pattern=vertical lines,pattern color=col1] (A)--(B)--(E)--cycle;
   \fill[pattern=north east lines,pattern color=col4] (A)--(C)--(D)--cycle;
   \fill[pattern=north west lines,pattern color=col5] (A)--(C)--(E)--cycle;
   \node[vertex, label={below:$v$}] (1) at (A) {};
   \node[vertex] (2) at (B) {};
   \node[vertex] (3) at (C) {};
   \node[vertex] (4) at (D) {};
   \node[vertex] (5) at (E) {};
   \draw[edge] (1)--(2);
   \draw[edge] (1)--(3);
   \draw[edge] (1)--(4);
   \draw[edge] (1)--(5);
   \draw[edge] (2)--(4);
   \draw[edge] (2)--(5);
   \draw[edge] (3)--(4);
   \draw[edge] (3)--(5);
  \end{tikzpicture}
 \caption{The three flat realizations of a lozenge.}
\end{figure}

\begin{fact}
\label{fact:flat_rhomboid}
 If a realization of a rhomboid has one dihedral angle between its triangular faces which is $0$ or~$\pi$, 
 then the realization is flat; 
 vice versa, if a realization of such a pyramid is flat, 
 then all dihedral angles are~$0$ or~$\pi$.
\end{fact}

\begin{fact}
\label{fact:flat_lozenge}
 Consider a non-degenerate motion of a lozenge. 
 The flat realizations in such a motion are precisely the ones where all dihedral angles are~$0$ or~$\pi$.
\end{fact}

\begin{remark}
 For deltoids and lozenges, there are non-flat realizations where one dihedral angle between its triangular faces is~$0$ or~$\pi$.
 Notice that these non-flat realizations appear, for example (though not only, in the case of deltoids), in \emph{degenerate} motions, 
 namely when two pairs of faces stay always coplanar during the motion. 
\end{remark}

\begin{definition}
A dihedral angle between two triangular faces of a flexible pyramid is \emph{simple} 
with respect to a motion of the pyramid if, once we fix a general value for that angle, 
there exists a unique (up to isometries) realization in that motion for which the angle has the given value.
\end{definition}

Notice that a lozenge has four simple dihedral angles, while a deltoid has two simple dihedral angles.\footnote{Recall that we always consider non-degenerate motions of pyramids.}

\begin{fact}
\label{fact:flat_deltoid}
 A deltoid is in a flat realization if and only if one of its simple dihedral angles between triangular faces are $0$ or~$\pi$.
\end{fact}

\begin{proposition}
\label{proposition:flat_realizations}
 If a flexible octahedron has two neighbor rhomboids or lozenges among its $6$ pyramids, 
 then it admits $2$ flat realizations.
 Here, two pyramids~\quadrefV\ and~\quadrefW\ are neighbors
 if the vertices~$v$ and~$w$ are connected by an edge.
\end{proposition}
\begin{proof}
 Suppose that the two neighbor pyramids are~\quadrefA\ and~\quadrefC. 
 Suppose we are in a realization that is flat for~\quadrefA. 
 Then the dihedral angle between the planes~$135$ and~$136$ is~$0$ or~$\pi$. 
 Hence, by Facts~\ref{fact:flat_rhomboid} and~\ref{fact:flat_lozenge} this realization is also flat for~\quadrefC.
\end{proof}

\begin{remark}
 By Proposition~\ref{proposition:flat_realizations}, Type III octahedra admit two flat realizations.
\end{remark}

\section{Classification of flexible octahedra}
\label{classification}

In this section we provide the classification of flexible octahedra, reproving the known results by Bricard.
We do this by attaching combinatorial objects to flexible octahedra and prescribing rules for these objects.
Eventually, the rules determine constraints on edge lengths and angles, which can be grouped in four cases. 
By analyzing each of these cases, we classify flexible octahedra into the three families introduced by Bricard and described in the introduction. 

The justification for the rules is provided in Section~\ref{justifications}, 
and requires the algebro-geometric notion of \emph{moduli space of rational stable curves with marked points}, 
together with the theory developed by the authors in~\cite{Gallet2019} about flexible graphs on the sphere. 
Once the rules are established, however, the derivation of the classification is combinatorial in nature, 
and uses the elementary facts about octahedra reported in Section~\ref{elementary}. 
We believe that inserting this combinatorial ``extra-layer'' in the proof has two advantages: 
it helps highlighting the structure of the proof and separating logically independent units, 
and facilitates readers that may not be interested in the algebro-geometric technicalities to follow the proof of the classification.

\subsection{Objects}
\label{classification:objects}

We are going to introduce two combinatorial objects that will guide the classification, called \emph{octahedral} and \emph{pyramidal bonds}.
These are graphical representations of ``points at infinity'' of the space of realizations of an octahedron.
Key ingredients in our constructions are \emph{quadrilaterals} in~$\GOct$;
they are induced subgraphs isomorphic to the cycle~$C_4$ on four vertices.

\begin{notation}
	There exist exactly three quadrilaterals of~$\GOct$: they are those induced by the vertices $\{1,2,3,4\}$, $\{1,2,5,6\}$, and $\{3,4,5,6\}$. 
	Each quadrilateral is completely specified by the pair of vertices not appearing in it, which form a non-edge in~$\GOct$.
	Therefore we can label the quadrilaterals by~$12$, $34$, and~$56$.
\end{notation}

\minisec{Octahedral bonds}
Octahedral bonds are quadrilaterals of $\GOct$ together with a choice of orientation for all edges in the quadrilateral.
There are, therefore, $16$ octahedral bonds supported on a given quadrilateral in~$\GOct$.
Octahedral bonds come with some multiplicity, which we call the \emph{$\mu$-number}.

\minisec{Pyramidal bonds}
A pyramidal bond is a pyramid~\quadrefV, together with a direction on two edges incident to~$v$ 
that are not in the same triangle subgraph (see Figure~\ref{fig:pyramid_bond}).
There are hence $8$ pyramidal bonds supported on a given pyramid~\quadrefV\ in~$\GOct$.
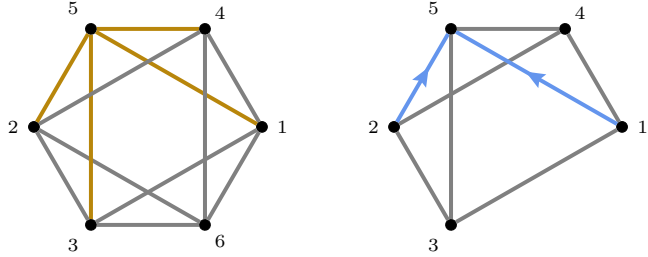
\begin{figure}[H]
	\centering
	\begin{tikzpicture}[scale=1.5]
		\coordinate (o) at (0,0);
		\node[vertex,label={[labelsty]right:$1$}] (1) at (1,0) {};
		\node[vertex,label={[labelsty,label distance=-2pt]right:$4$},rotate around={60:(o)}] (4) at (1) {};
		\node[vertex,label={[labelsty]right:$5$},rotate around={120:(o)}] (5) at (1) {};
		\node[vertex,label={[labelsty]left:$2$}] (2) at (-1,0) {};
		\node[vertex,label={[labelsty]left:$3$},rotate around={60:(o)}] (3) at (2) {};
		\node[vertex,label={[labelsty,label distance=-2pt]left:$6$},rotate around={120:(o)}] (6) at (2) {};
		\draw[edge] (1) -- (6);
		\draw[edge] (1) -- (3);
		\draw[edge] (4) -- (1);
		\draw[edge,col1] (5) -- (1);
		\draw[edge] (6) -- (3);
		\draw[edge] (4) -- (6);
		\draw[edge] (6) -- (2);
		\draw[edge,col1] (3) -- (5);
		\draw[edge] (3) -- (2);
		\draw[edge,col1] (5) -- (4);
		\draw[edge] (2) -- (4);
		\draw[edge,col1] (2) -- (5);
	\end{tikzpicture}
	\qquad
	\begin{tikzpicture}[scale=1.5]
		\coordinate (o) at (0,0);
		\node[vertex,label={[labelsty]right:$1$}] (1) at (1,0) {};
		\node[vertex,label={[labelsty,label distance=-2pt]right:$4$},rotate around={60:(o)}] (4) at (1) {};
		\node[vertex,label={[labelsty]left:$2$}] (2) at (-1,0) {};
		\node[vertex,label={[labelsty]left:$3$},rotate around={60:(o)}] (3) at (2) {};
		\node[vertex,label={[labelsty]right:$5$},rotate around={120:(o)}] (5) at (1) {};
		\draw[edge] (1) -- (3) (1) -- (4) (2) -- (3) (2) -- (4);
		\draw[dedge,colpb] (1)--(5);
		\draw[dedge,colpb] (2)--(5);
		\draw[edge] (3)edge(5) (4)edge(5);
	\end{tikzpicture}
 \caption{A pyramidal bond defined by vertex~$5$.}
 \label{fig:pyramid_bond}
\end{figure}
We fix a standard representation of a pyramid~\quadrefV\ by specifying which vertex is drawn where.
More precisely, we draw the pyramid as a square with the vertex~$v$ in the middle.
Then we take the clockwise neighbor of $v$ in the drawing of Figure~\ref{figure:Goct} to be on the bottom right corner of the square.
The other vertices are drawn accordingly to the clockwise order (see Figure~\ref{fig:pyramid_representation}).
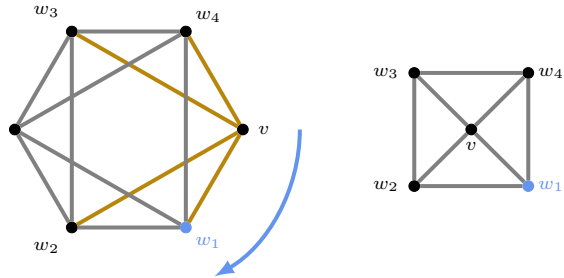
\begin{figure}[H]
	\centering
	\begin{tikzpicture}[scale=1.5,baseline=(o)]
		\coordinate (o) at (0,0);
		\node[vertex,label={[labelsty]right:$v$}] (1) at (1,0) {};
		\node[vertex,label={[labelsty,label distance=-2pt]right:$w_4$},rotate around={60:(o)}] (4) at (1) {};
		\node[vertex,label={[labelsty]right:$w_3$},rotate around={120:(o)}] (5) at (1) {};
		\node[vertex] (2) at (-1,0) {};
		\node[vertex,label={[labelsty]left:$w_2$},rotate around={60:(o)}] (3) at (2) {};
		\node[vertex,label={[labelsty,label distance=-2pt,colpbfixed]left:$w_1$},rotate around={120:(o)},colpbfixed] (6) at (2) {};
		\draw[edge,col1] (1) -- (6);
		\draw[edge,col1] (1) -- (3);
		\draw[edge,col1] (4) -- (1);
		\draw[edge,col1] (5) -- (1);
		\draw[edge] (6) -- (3);
		\draw[edge] (4) -- (6);
		\draw[edge] (6) -- (2);
		\draw[edge] (3) -- (5);
		\draw[edge] (3) -- (2);
		\draw[edge] (5) -- (4);
		\draw[edge] (2) -- (4);
		\draw[edge] (2) -- (5);
		\draw[edge,-latex,colpbfixed] ($(1)+(0.5,0)$) arc [radius=1.5,start angle=0,delta angle=-60];
	\end{tikzpicture}
	\qquad
	\begin{tikzpicture}[scale=1.5,baseline=(o)]
		\coordinate (o) at (0,0);
		\coordinate (1) at (1,0) {};
		\node[vertex,label={[labelsty,label distance=-2pt]right:$w_4$}] (w4) at (0.5,0.5) {};
		\coordinate (2) at (-1,0);
		\node[vertex,label={[labelsty]left:$w_2$}] (w2) at (-0.5,-0.5) {};
		\node[vertex,label={[labelsty]left:$w_3$}] (w3) at (-0.5,0.5) {};
		\node[vertex,label={[labelsty,label distance=-2pt,colpbfixed]right:$w_1$},colpbfixed] (w1) at (0.5,-0.5) {};
		\node[vertex,label={[labelsty,label distance=-1pt]below:$v$}] (v) at (o) {};
		\draw[edge] (w1) -- (w2) (w2) -- (w3) (w3) -- (w4) (w4) -- (w1);
		\draw[edge] (v)edge(w1) (v)edge(w2) (v)edge(w3) (v)edge(w4);
	\end{tikzpicture}
 \caption{A pyramid in standard representation.}
 \label{fig:pyramid_representation}
\end{figure}

The standard representation of pyramids provides a standard way to represent pyramidal bonds:
\begin{table}[H]
	\centering
  \begin{tabular}{*{8}{c}}
    \pbonddl & \pbonddc & \pbondar & \pbondac & \pbonddr & \pbonddo & \pbondal & \pbondao\\
    $\Pyr_{\Pbonddl}$ & $\Pyr_{\Pbonddc}$ & $\Pyr_{\Pbondar}$ & $\Pyr_{\Pbondac}$ & $\Pyr_{\Pbonddr}$ & $\Pyr_{\Pbonddo}$ & $\Pyr_{\Pbondal}$ & $\Pyr_{\Pbondao}$ \\
  \end{tabular}
\end{table}
If we want to specify that a bond is supported on a pyramid~\quadrefV, we put the symbol~$v$ as superscript, as for example in~$\Pyr_{\Pbonddl}^v$.
 
\subsection{Rules}
\label{classification:rules}

We introduce the rules that are satisfied by octahedral and pyramidal bonds;
their formal justification will be provided in Section~\ref{justifications}.
First of all, we fix a motion of an octahedron. 
This induces motions also on all the pyramids of the octahedron.
The motion of the octahedron carries a certain number of octahedral bonds,
and similarly the motions of the pyramids carry a certain number of pyramidal bonds.
``To carry'' a bond means that its $\mu$-number is positive for a given motion.
If the $\mu$-number of a bond is zero, we say that the motion ``does not have'' that bond.
Pyramidal and octahedral bonds associated to a motion must satisfy the following rules.

\begin{enumerate}\renewcommand{\theenumi}{\textbf{R\arabic{enumi}}}\renewcommand{\labelenumi}{\theenumi:}
  \item\label{rule:table}
		Depending on their bonds, flexible pyramids can be distinguished into five families: 
		general (\caseG), even deltoids (\caseE, with two subfamilies), odd deltoids 
		(\caseO, with two subfamilies), rhomboids (\caseR, with four subfamilies), and 
		lozenges (\caseL, with four subfamilies). The family/subfamily is completely 
		determined by the kind of bonds arising (with only one ambiguity: 
		rhomboids/lozenges), as specified by Table~\ref{table:bonds}.
		A deltoid~\quadrefV\ is even if the dihedral angles at its even edges are simple,
		where even edges are determined by Figure~\ref{figure:even_odd_pyramids}.
		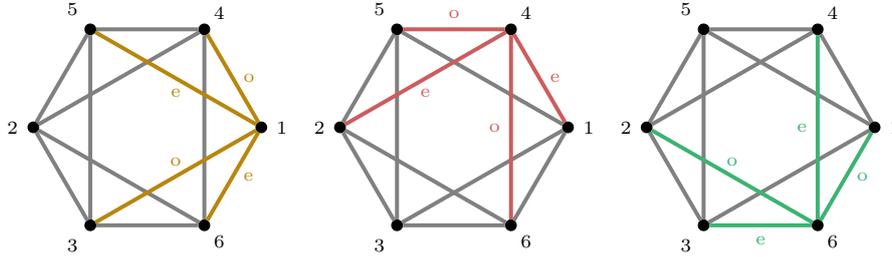
\begin{figure}[ht]
			\centering
			\begin{tikzpicture}[scale=1.5]
				\coordinate (o) at (0,0);
				\node[vertex,label={[labelsty]right:$1$}] (1) at (1,0) {};
				\node[vertex,label={[labelsty,label distance=-2pt]right:$4$},rotate around={60:(o)}] (4) at (1) {};
				\node[vertex,label={[labelsty]right:$5$},rotate around={120:(o)}] (5) at (1) {};
				\node[vertex,label={[labelsty]left:$2$}] (2) at (-1,0) {};
				\node[vertex,label={[labelsty]left:$3$},rotate around={60:(o)}] (3) at (2) {};
				\node[vertex,label={[labelsty,label distance=-2pt]left:$6$},rotate around={120:(o)}] (6) at (2) {};
				\draw[edge] (6) -- (3);
				\draw[edge] (4) -- (6);
				\draw[edge] (6) -- (2);
				\draw[edge] (3) -- (5);
				\draw[edge] (3) -- (2);
				\draw[edge] (5) -- (4);
				\draw[edge] (2) -- (4);
				\draw[edge] (2) -- (5);
				\draw[edge,col1] (1) to node[labelsty,right] {e} (6);
				\draw[edge,col1] (1) to node[labelsty,above] {o} (3);
				\draw[edge,col1] (4) to node[labelsty,right] {o} (1);
				\draw[edge,col1] (5) to node[labelsty,below] {e} (1); 
			\end{tikzpicture}
			\begin{tikzpicture}[scale=1.5]
				\coordinate (o) at (0,0);
				\node[vertex,label={[labelsty]right:$1$}] (1) at (1,0) {};
				\node[vertex,label={[labelsty,label distance=-2pt]right:$4$},rotate around={60:(o)}] (4) at (1) {};
				\node[vertex,label={[labelsty]right:$5$},rotate around={120:(o)}] (5) at (1) {};
				\node[vertex,label={[labelsty]left:$2$}] (2) at (-1,0) {};
				\node[vertex,label={[labelsty]left:$3$},rotate around={60:(o)}] (3) at (2) {};
				\node[vertex,label={[labelsty,label distance=-2pt]left:$6$},rotate around={120:(o)}] (6) at (2) {};
				\draw[edge] (1) -- (3);
				\draw[edge] (6) -- (1);
				\draw[edge] (5) -- (1);
				\draw[edge] (3) -- (5);
				\draw[edge] (3) -- (2);
				\draw[edge] (3) -- (6);
				\draw[edge] (2) -- (6);
				\draw[edge] (2) -- (5);
				\draw[edge,col4] (4) to node[labelsty,right] {e} (1);
				\draw[edge,col4] (4) to node[labelsty,below] {e} (2);
				\draw[edge,col4] (4) to node[labelsty,left] {o} (6);
				\draw[edge,col4] (4) to node[labelsty,above] {o} (5);
			\end{tikzpicture}
			\begin{tikzpicture}[scale=1.5]
				\coordinate (o) at (0,0);
				\node[vertex,label={[labelsty]right:$1$}] (1) at (1,0) {};
				\node[vertex,label={[labelsty,label distance=-2pt]right:$4$},rotate around={60:(o)}] (4) at (1) {};
				\node[vertex,label={[labelsty]right:$5$},rotate around={120:(o)}] (5) at (1) {};
				\node[vertex,label={[labelsty]left:$2$}] (2) at (-1,0) {};
				\node[vertex,label={[labelsty]left:$3$},rotate around={60:(o)}] (3) at (2) {};
				\node[vertex,label={[labelsty,label distance=-2pt]left:$6$},rotate around={120:(o)}] (6) at (2) {};
				\draw[edge] (1) -- (3);
				\draw[edge] (4) -- (1);
				\draw[edge] (5) -- (1);
				\draw[edge] (3) -- (5);
				\draw[edge] (3) -- (2);
				\draw[edge] (5) -- (4);
				\draw[edge] (2) -- (4);
				\draw[edge] (2) -- (5);
				\draw[edge,col6] (1) to node[labelsty,right] {o} (6);
				\draw[edge,col6] (6) to node[labelsty,below] {e} (3);
				\draw[edge,col6] (4) to node[labelsty,left] {e} (6);
				\draw[edge,col6] (6) to node[labelsty,above] {o} (2);
			\end{tikzpicture}
			\caption{Assignment for even and odd edges of pyramids. Only three pyramids are shown, since the assignment for the other three can be deduced as follows: $\{1, a\}$ is even/odd if and only if $\{2, a\}$ is so, and analogously for the other two pairs $(3,4)$ and $(5,6)$.}
			\label{figure:even_odd_pyramids}
		\end{figure}
		\begin{table}[ht]
			\centering
			\caption{Pyramidal bonds associated to the possible families of pyramids.
			A $1$ denotes that the bond is present, while $0$ denotes that it is not present.
			Moreover, pyramids are drawn in their standard representation.}
		  \begin{tabular}{cccccc}
				\toprule
				family & subfamily & \pbonddl & \pbonddc & \pbondar & \pbondac  \\
				                 & & \pbonddr & \pbonddo & \pbondal & \pbondao  \\
				\midrule
				\caseG & & 1 & 1 & 1 & 1 \\
				\midrule
				\multirow{2}{*}{\caseO} & \phantomas[l]{ antipodal}{\text{ coincide}} 
					& 1 & 1 
					& 1 & 0 \\
					& \text{ antipodal} & 1 & 1 & 0 & 1 \\
				\midrule
				\multirow{2}{*}{\caseE} & \phantomas[l]{ antipodal}{\text{ coincide}} 
				& 1 & 0 & 1 & 1 \\
				& \text{ antipodal} & 0 & 1 & 1 & 1 \\
				\midrule
				\multirow{4}{*}{\caseR} & Type 1 & 1 & 0 & 1 & 0 \\
				& Type 2 & 0 & 1 & 1 & 0 \\
				& Type 3 & 1 & 0 & 0 & 1 \\
				& Type 4 & 0 & 1 & 0 & 1 \\
				\midrule
				\multirow{4}{*}{\caseL} & Type 1 & 1 & 0 & 1 & 0 \\
				& Type 2 & 0 & 1 & 1 & 0 \\
				& Type 3 & 1 & 0 & 0 & 1 \\
				& Type 4 & 0 & 1 & 0 & 1 \\
				\bottomrule
			\end{tabular}
			\label{table:bonds}
		\end{table}
\end{enumerate}

The next rule describes the connection between octahedral bonds and edge lengths:
the presence of a bond for a motion determines linear relations between the edge lengths of a quadrilateral.

\begin{definition}
	\label{definition:directions}
	We choose the orientation of the edges of~$\GOct$ as in Figure~\ref{figure:directions}
	and denote this oriented graph by~$\GOctDir$. 
	Notice that this choice is equivariant under cyclic permutations of the vertices $(1,4,5,2,3,6)$.

	Given a labeling $\lambda \colon E_{\oct} \longrightarrow \R_{>0}$, 
	and given an oriented edge $(i,j)$ in~$\GOctDir$, 
	we define the number~$\ell_{ij}$ to be the length~$\lambda_{\{i, j\}}$.
	We define the number~$\ell_{ji}$ to be~$-\ell_{ij}$.
\end{definition}

	\begin{figure}[ht]
	\centering
		\begin{tikzpicture}[scale=1.5]
			\coordinate (o) at (0,0);
			\node[vertex,label={[labelsty]right:$1$}] (1) at (1,0) {};
			\node[vertex,label={[labelsty,label distance=-2pt]right:$4$},rotate around={60:(o)}] (4) at (1) {};
			\node[vertex,label={[labelsty]right:$5$},rotate around={120:(o)}] (5) at (1) {};
			\node[vertex,label={[labelsty]left:$2$}] (2) at (-1,0) {};
			\node[vertex,label={[labelsty]left:$3$},rotate around={60:(o)}] (3) at (2) {};
			
			\node[vertex,label={[labelsty,label distance=-2pt]left:$6$},rotate around={120:(o)}] (6) at (2) {};
			\draw[dedge] (1) -- (6);
			\draw[dedge] (1) -- (3);
			\draw[dedge] (4) -- (1);
			\draw[dedge] (5) -- (1);
			\draw[dedge] (6) -- (3);
			\draw[dedge] (4) -- (6);
			\draw[dedge] (6) -- (2);
			\draw[dedge] (3) -- (5);
			\draw[dedge] (3) -- (2);
			\draw[dedge] (5) -- (4);
			\draw[dedge] (2) -- (4);
			\draw[dedge] (2) -- (5);
		\end{tikzpicture}
		\caption{Fixed orientations in the graph~$\GOct$. We call this oriented graph~$\GOctDir$.}
		\label{figure:directions}
	\end{figure}
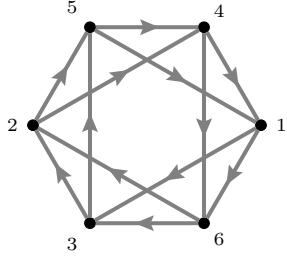

\begin{enumerate}\renewcommand{\theenumi}{\textbf{R\arabic{enumi}}}\renewcommand{\labelenumi}{\theenumi:}\setcounter{enumi}{1}
	\item\label{rule:lengths}
		For a motion having an octahedral bond with oriented edges $(t_1, s_1)$, $(t_2, s_2)$, $(t_3, s_3)$, $(t_4, s_4)$, 
		the following relation among the edge lengths of the octahedron holds:
		\begin{equation}
		\label{equation:condition_lengths}
			\ell_{t_1 \, s_1} + \ell_{t_2 \, s_2} + \ell_{t_3 \, s_3} + \ell_{t_4 \, s_4} = 0
			\, .
		\end{equation}
\end{enumerate}

From Rule~\ref{rule:lengths} we can already infer some properties of bonds of flexible octahedra.
We show that only some octahedral bonds may arise for a motion.

\begin{lemma}
\label{lemma:condition_bond}
	Consider an octahedral bond for a flexible octahedron, with an orientation $(t_1, s_1), \dotsc, (t_4, s_4)$.
	Equation~\eqref{equation:condition_lengths} from Rule~\ref{rule:lengths} has non-trivial solutions 
	only if exactly two of the oriented edges $(t_1, s_1), \dotsc, (t_4, s_4)$ coincide with the oriented edges induced by~$\GOctDir$ 
	(see Figure~\ref{figure:bond_orientation}).
\end{lemma}
\begin{figure}[ht]
	\centering
	\begin{tikzpicture}[scale=1.5]
		\coordinate (o) at (0,0);
		\node[vertex,label={[labelsty]right:$1$}] (1) at (1,0) {};
		\node[vertex,label={[labelsty,label distance=-2pt]right:$4$},rotate around={60:(o)}] (4) at (1) {};
		\node[vertex,label={[labelsty]right:$5$},rotate around={120:(o)}] (5) at (1) {};
		\node[vertex,label={[labelsty]left:$2$}] (2) at (-1,0) {};
		\node[vertex,label={[labelsty]left:$3$},rotate around={60:(o)}] (3) at (2) {};
		\node[vertex,label={[labelsty,label distance=-2pt]left:$6$},rotate around={120:(o)}] (6) at (2) {};
		\draw[dedge] (1) -- (6);
		\draw[dedge,col1] (1) -- (3);
		\draw[dedge,col1] (4) -- (1);
		\draw[dedge] (5) -- (1);
		\draw[dedge] (6) -- (3);
		\draw[dedge] (4) -- (6);
		\draw[dedge] (6) -- (2);
		\draw[dedge] (3) -- (5);
		\draw[dedge,col1] (3) -- (2);
		\draw[dedge] (5) -- (4);
		\draw[dedge,col1] (2) -- (4);
		\draw[dedge] (2) -- (5);
	\end{tikzpicture}
	\qquad
	\begin{tikzpicture}[scale=1.5]
		\coordinate (o) at (0,0);
		\node[vertex,label={[labelsty]right:$1$}] (1) at (1,0) {};
		\node[vertex,label={[labelsty,label distance=-2pt]right:$4$},rotate around={60:(o)}] (4) at (1) {};
		\node[vertex,label={[labelsty]left:$2$}] (2) at (-1,0) {};
		\node[vertex,label={[labelsty]left:$3$},rotate around={60:(o)}] (3) at (2) {};
		\draw[dedge,colds] (1) -- (3);
		\draw[dedge,coldo] (1) -- (4);
		\draw[dedge,coldo] (2) -- (3);
		\draw[dedge,colds] (2) -- (4);
	\end{tikzpicture}
 \caption{Orientations of the edges of a quadrilateral in $\GOctDir$ (left) and those induced by a bond of a flexible octahedron (right).
 Green edges \protect\tikz{\protect\draw[dedge,colds] (0,0) -- (0.75,0);} describe edges where the orientation coincides and
 red ones \protect\tikz{\protect\draw[dedge,coldo] (0,0) -- (0.75,0);} where they are opposite.}
 \label{figure:bond_orientation}
\end{figure}
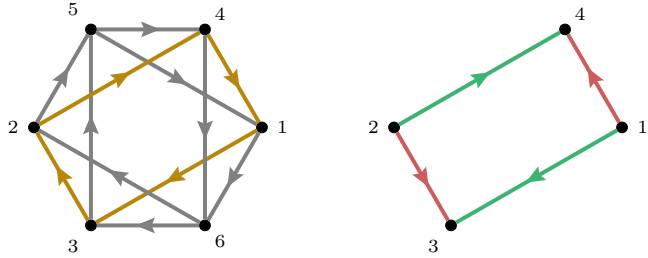
\begin{proof}
 If all (or no) oriented edges in the quadrilateral coincide with the ones induced by~$\GOctDir$,
 then in Equation~\eqref{equation:condition_lengths} we have that the sum of four positive quantities is zero, a contradiction.
 If one (or three) oriented edges in the quadrilateral coincide with the ones induced by~$\GOctDir$,
 then we obtain a relation of the form $\ell_1 = \ell_2 + \ell_3 + \ell_4$, where all quantities~$\ell_k$ are positive.
 This implies that all the vertices of the quadrilateral are collinear in a general realization of the flexible octahedron;
 hence, some faces are coplanar, and we excluded this possibility.
 Then the only situation left is the one from the statement.
\end{proof}

A simple inspection provides the following result.

\begin{proposition}
\label{proposition:positive_divisors}
	Out of the $16$ possible octahedral bonds supported on a quadrilateral in~$\GOct$, only $6$ fulfill the condition of Lemma~\ref{lemma:condition_bond}.
	They come in three pairs, where two orientations are in the same pair if one can be obtained from the other by 
	reversing the orientations of all edges. 
	One of these pairs is constituted of orientations with the following property: if $(t_1, s_1), \dotsc, (t_4, s_4)$ are the oriented edges, then 
	\[
	\left|
	\bigcup \bigl\{ 
			\{t_k, s_k\} \, \colon \, 
			(t_k, s_k) \text{ is a directed edge of } \GOctDir, \;
			k \in \{1, \dotsc, 4\} 
		\bigr\}
	\right| = 4 \,.
	\]
	This means that those edges that are oriented as in $\GOctDir$ span the vertices of the quadrilateral. 
	These two special orientations are depicted as case~$\ObondA$ in Figure~\ref{figure:typesABC}.
\end{proposition}

\begin{notation}
	We use the following notation for the $6$ possible octahedral bonds on a given quadrilateral as described by Proposition~\ref{proposition:positive_divisors}.
	Let $12$, $34$, and~$56$ be the three quadrilaterals of~$\GOct$.
	The six possible bonds associated to the quadrilateral~$ij$ are denoted $\Bond^{ij}_{\ObondA}$, 
	$\Bond^{ij}_{\ObondAc}$, $\Bond^{ij}_{\ObondB}$, $\Bond^{ij}_{\ObondBc}$, $\Bond^{ij}_{\ObondC}$, 
	$\Bond^{ij}_{\ObondCc}$ according to the following criterion.
	As we mentioned in Definition~\ref{definition:directions},
	the orientation in~$\GOctDir$ is equivariant under cyclic permutations of the indices $(1,4,5,2,3,6)$.
	Hence, it is enough to define the notation only for the bonds associated to the 
	quadrilateral~$56$, and extend the notion to the others using cyclic permutations.
	We define $\Bond^{56}_{\ObondA}$, $\Bond^{56}_{\ObondB}$, and $\Bond^{56}_{\ObondC}$ as the 
	bonds inducing the orientations as in Figure~\ref{figure:typesABC}.
	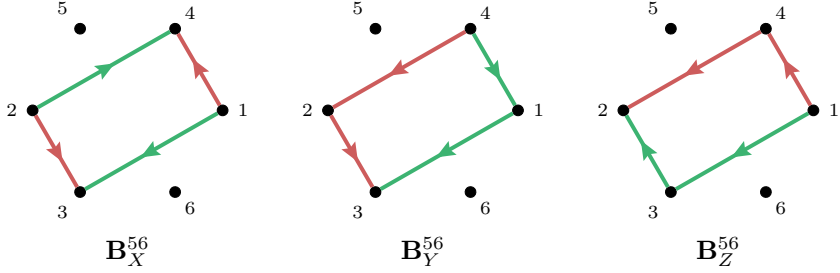
\begin{figure}[H]
		\centering
		\begin{tikzpicture}[scale=1.25]
				\coordinate (o) at (0,0);
				\node[vertex,label={[labelsty]right:$1$}] (1) at (1,0) {};
				\node[vertex,label={[labelsty,label distance=-2pt]right:$4$},rotate around={60:(o)}] (4) at (1) {};
				\node[vertex,label={[labelsty]right:$5$},rotate around={120:(o)}] (5) at (1) {};
				\node[vertex,label={[labelsty]left:$2$}] (2) at (-1,0) {};
				\node[vertex,label={[labelsty]left:$3$},rotate around={60:(o)}] (3) at (2) {};
				\node[vertex,label={[labelsty,label distance=-2pt]left:$6$},rotate around={120:(o)}] (6) at (2) {};
				\draw[dedge,colds] (1) -- (3);
				\draw[dedge,coldo] (1) -- (4);
				\draw[dedge,coldo] (2) -- (3);
				\draw[dedge,colds] (2) -- (4);
				\node[] at (0,-1.5) {$\Bond^{56}_{\ObondA}$};
			\end{tikzpicture}
			\quad
			\begin{tikzpicture}[scale=1.25]
				\coordinate (o) at (0,0);
				\node[vertex,label={[labelsty]right:$1$}] (1) at (1,0) {};
				\node[vertex,label={[labelsty,label distance=-2pt]right:$4$},rotate around={60:(o)}] (4) at (1) {};
				\node[vertex,label={[labelsty]right:$5$},rotate around={120:(o)}] (5) at (1) {};
				\node[vertex,label={[labelsty]left:$2$}] (2) at (-1,0) {};
				\node[vertex,label={[labelsty]left:$3$},rotate around={60:(o)}] (3) at (2) {};
				\node[vertex,label={[labelsty,label distance=-2pt]left:$6$},rotate around={120:(o)}] (6) at (2) {};
				\draw[dedge,colds] (1) -- (3);
				\draw[dedge,colds] (4) -- (1);
				\draw[dedge,coldo] (2) -- (3);
				\draw[dedge,coldo] (4) -- (2);
				\node[] at (0,-1.5) {$\Bond^{56}_{\ObondB}$};
			\end{tikzpicture}
			\quad
			\begin{tikzpicture}[scale=1.25]
				\coordinate (o) at (0,0);
				\node[vertex,label={[labelsty]right:$1$}] (1) at (1,0) {};
				\node[vertex,label={[labelsty,label distance=-2pt]right:$4$},rotate around={60:(o)}] (4) at (1) {};
				\node[vertex,label={[labelsty]right:$5$},rotate around={120:(o)}] (5) at (1) {};
				\node[vertex,label={[labelsty]left:$2$}] (2) at (-1,0) {};
				\node[vertex,label={[labelsty]left:$3$},rotate around={60:(o)}] (3) at (2) {};
				\node[vertex,label={[labelsty,label distance=-2pt]left:$6$},rotate around={120:(o)}] (6) at (2) {};
				\draw[dedge,colds] (1) -- (3);
				\draw[dedge,coldo] (1) -- (4);
				\draw[dedge,colds] (3) -- (2);
				\draw[dedge,coldo] (4) -- (2);
				\node[] at (0,-1.5) {$\Bond^{56}_{\ObondC}$};
			\end{tikzpicture}
		\caption{Three of the six possible octahedral bonds on the quadrilateral~$56$.}
		\label{figure:typesABC}
	\end{figure}
	The bonds $\Bond^{56}_{\ObondAc}$, $\Bond^{56}_{\ObondBc}$, and $\Bond^{56}_{\ObondCc}$ 
	are defined to be the bonds inducing the reversed orientations with respect 
	to the three previous ones. The two bonds~$\Bond^{56}_{\ObondA}$ 
	and~$\Bond^{56}_{\ObondAc}$ have the special property mentioned in 
	Proposition~\ref{proposition:positive_divisors}. 
	By applying cyclic permutations to the previous $6$ orientations, we obtain $36$ quadrilaterals with oriented edges. 
	The notation symbols for these bonds are obtained by applying cyclic permutations to the indices appearing in the symbols for the bonds $\Bond^{56}_{\bullet}$, 
	where $\bullet \in \{ \ObondA, \ObondB, \ObondC, \ObondAc, \ObondBc, \ObondCc \}$, and then by applying the following rules:
	\[
	\Bond^{ij}_{\ObondA} = \Bond^{ji}_{\ObondAc}, \quad 
	\Bond^{ij}_{\ObondB} = \Bond^{ji}_{\ObondB}, \quad
	\Bond^{ij}_{\ObondBc} = \Bond^{ji}_{\ObondBc}, \quad
	\Bond^{ij}_{\ObondC} = \Bond^{ji}_{\ObondC}, \quad 
	\Bond^{ij}_{\ObondCc} = \Bond^{ji}_{\ObondCc}.
	\]
\end{notation}

\begin{notation}
	We denote the $\mu$-number of the octahedral bond~$\Bond^{ij}_{\ObondA}$ by~$\mu_{\ObondA}^{ij}$, and similarly for the other bonds. 
\end{notation}
\begin{notation}
	If a pyramid~\quadrefV\ has a pyramidal bond $\Pyr_{\Pbondac}$, we write $\mu_{\Pbondac}^v = 1$;
	otherwise we write $\mu_{\Pbondac}^v = 0$.
	Similarly we define $\mu$-numbers for the other pyramidal bonds.
\end{notation}

The next rule explains what happens when we reverse directed edges in a bond.

\begin{enumerate}\renewcommand{\theenumi}{\textbf{R\arabic{enumi}}}\renewcommand{\labelenumi}{\theenumi:}\setcounter{enumi}{2}
	\item\label{rule:conjugates}
		The $\mu$-number of a pyramidal or octahedral bond coincides with the $\mu$-number of the bond obtained by reversing the directed edges.
		In particular, if a motion carries a bond then it also carries the corresponding bond with reversed directions.
\end{enumerate}

To start the classification, we need one last rule, linking $\mu$-numbers of octahedral bonds to $\mu$-numbers of pyramidal bonds. 
This rule, however, works only under an assumption on the pyramids of the octahedron, called simplicity. 

\begin{definition}
 Consider a flexible octahedron.
 We say that a pyramid is \emph{simple} if, 
 given a general realization of the pyramid for the induced motion, 
 there is exactly one non-degenerate realization of the octahedron that extends the one of the pyramid.
\end{definition}

We can now state the last rule and then we start the classification in the case of simple pyramids.
Afterwards, we deal with the situation of non-simple pyramids.

\begin{enumerate}\renewcommand{\theenumi}{\textbf{R\arabic{enumi}}}\renewcommand{\labelenumi}{\theenumi:}\setcounter{enumi}{3}
  \item\label{rule:equations} Given a motion of a flexible octahedron, suppose that all pyramids are simple.
		Then we have relations between the $\mu$-numbers of the possible pyramidal bonds supported on~\quadrefV\ 
		and the $\mu$-numbers of the $18$ possible octahedral bonds given by the following graphical rule (see Figure~\ref{figure:eq_derivation}). 
		We consider a possible pyramidal bond, for example $\Pyr_{\Pbondar}$ on \quadrefA. 
		We draw the orientation of the two edges specified by $\Pyr_{\Pbondar}$ on the representation of~$\GOct$ as in Figure~\ref{figure:Goct}. 
		The $\mu$-number of~$\Pyr_{\Pbondar}$ is then equal to the sum of the $\mu$-numbers of the octahedral bonds that ``extend'' the two oriented edges of~$\Pyr_{\Pbondar}$; 
		in this case, we have a unique way to extend them, namely by~$\Bond_{\ObondBc}^{56}$. 
		Hence, we get the relation $\mu_{\Pbondar}^1 = \mu_{\ObondBc}^{56}$. 
		If we start, instead, from~$\Pyr_{\Pbondao}$ again on pyramid~\quadrefA, we have two ways to extend it, namely by $\Bond_{\ObondA}^{56}$, and $\Bond_{\ObondC}^{56}$. 
		Therefore, the relation is $\mu_{\Pbondao}^1 = \mu_{\ObondA}^{56} + \mu_{\ObondC}^{56}$. 
		\begin{figure}[ht]
		  \centering
		    \begin{tikzpicture}
		      \begin{scope}
		        \node[] at (0,0) {$\Pyr_{\Pbondar}$};
		      \end{scope}
		      \draw[ultra thick,->] (0.4,0) -- (0.9,0);
		      \begin{scope}[xshift=1.5cm]
		        \node[] at (0,0) {\pbondar};
		      \end{scope}
		      \draw[ultra thick,->] (2.1,0) -- (2.6,0);
		      \begin{scope}[xshift=3.95cm,scale=0.8]
		        \coordinate (o) at (0,0);
						\node[vertex,label={[labelsty]right:$1$}] (1) at (1,0) {};
						\node[vertex,label={[labelsty,label distance=-2pt]right:$4$},rotate around={60:(o)}] (4) at (1) {};
						\node[vertex,label={[labelsty]right:$5$},rotate around={120:(o)}] (5) at (1) {};
						\node[vertex,label={[labelsty]left:$2$}] (2) at (-1,0) {};
						\node[vertex,label={[labelsty]left:$3$},rotate around={60:(o)}] (3) at (2) {};
						\node[vertex,label={[labelsty,label distance=-2pt]left:$6$},rotate around={120:(o)}] (6) at (2) {};
						\draw[edge] (1) -- (6);
						\draw[dedge,colpb] (3) -- (1);
						\draw[dedge,colpb] (1) -- (4);
						\draw[edge] (5) -- (1);
						\draw[edge,col1] (6) -- (3);
						\draw[edge,col1] (4) -- (6);
						\draw[edge] (6) -- (2);
						\draw[edge,col1] (3) -- (5);
						\draw[edge] (3) -- (2);
						\draw[edge,col1] (5) -- (4);
						\draw[edge] (2) -- (4);
						\draw[edge] (2) -- (5);
		      \end{scope}
		      \draw[ultra thick,->] (5.25,0) -- (5.75,0);
		      \begin{scope}[xshift=7.15cm,scale=0.8]
		        \coordinate (o) at (0,0);
						\node[vertex,label={[labelsty]right:$1$}] (1) at (1,0) {};
						\node[vertex,label={[labelsty,label distance=-2pt]right:$4$},rotate around={60:(o)}] (4) at (1) {};
						\node[vertex,label={[labelsty]right:$5$},rotate around={120:(o)}] (5) at (1) {};
						\node[vertex,label={[labelsty]left:$2$}] (2) at (-1,0) {};
						\node[vertex,label={[labelsty]left:$3$},rotate around={60:(o)}] (3) at (2) {};
						\node[vertex,label={[labelsty,label distance=-2pt]left:$6$},rotate around={120:(o)}] (6) at (2) {};
						\draw[edge] (1) -- (6);
						\draw[dedge,coldo] (3) -- (1);
						\draw[dedge,coldo] (1) -- (4);
						\draw[edge] (5) -- (1);
						\draw[edge] (6) -- (3);
						\draw[edge] (4) -- (6);
						\draw[edge] (6) -- (2);
						\draw[edge] (3) -- (5);
						\draw[dedge,colds] (3) -- (2);
						\draw[edge] (5) -- (4);
						\draw[dedge,colds] (2) -- (4);
						\draw[edge] (2) -- (5);
		      \end{scope}
		      \draw[ultra thick,->] (8.45,0) -- (8.95,0);
		      \begin{scope}[xshift=8.95cm]
		        \node[anchor=west] at (0,0) {$\Bond_{\ObondBc}^{56}$};
		      \end{scope}
		    \end{tikzpicture}
		    
		    \begin{tikzpicture}
		      \begin{scope}
		        \node[] at (0,0) {$\Pyr_{\Pbondao}$};
		      \end{scope}
		      \draw[ultra thick,->] (0.4,0) -- (0.9,0);
		      \begin{scope}[xshift=1.5cm]
		        \node[] at (0,0) {\pbondao};
		      \end{scope}
		      \draw[ultra thick,->] (2.1,0) -- (2.6,0);
		      \begin{scope}[xshift=3.95cm,scale=0.8]
		        \coordinate (o) at (0,0);
						\node[vertex,label={[labelsty]right:$1$}] (1) at (1,0) {};
						\node[vertex,label={[labelsty,label distance=-2pt]right:$4$},rotate around={60:(o)}] (4) at (1) {};
						\node[vertex,label={[labelsty]right:$5$},rotate around={120:(o)}] (5) at (1) {};
						\node[vertex,label={[labelsty]left:$2$}] (2) at (-1,0) {};
						\node[vertex,label={[labelsty]left:$3$},rotate around={60:(o)}] (3) at (2) {};
						\node[vertex,label={[labelsty,label distance=-2pt]left:$6$},rotate around={120:(o)}] (6) at (2) {};
						\draw[edge] (1) -- (6);
						\draw[dedge,colpb] (1) -- (3);
						\draw[dedge,colpb] (1) -- (4);
						\draw[edge] (5) -- (1);
						\draw[edge,col1] (6) -- (3);
						\draw[edge,col1] (4) -- (6);
						\draw[edge] (6) -- (2);
						\draw[edge,col1] (3) -- (5);
						\draw[edge] (3) -- (2);
						\draw[edge,col1] (5) -- (4);
						\draw[edge] (2) -- (4);
						\draw[edge] (2) -- (5);
		      \end{scope}
		      \draw[ultra thick,->] (5.25,0) -- (5.75,0);
		      \draw[black] (6.2,-2.3) -- (5.9,-2.3) -- (5.9,2.3) -- (6.2,2.3);
		      \begin{scope}[xshift=7.15cm,yshift=1.2cm,scale=0.8]
		        \coordinate (o) at (0,0);
						\node[vertex,label={[labelsty]right:$1$}] (1) at (1,0) {};
						\node[vertex,label={[labelsty,label distance=-2pt]right:$4$},rotate around={60:(o)}] (4) at (1) {};
						\node[vertex,label={[labelsty]right:$5$},rotate around={120:(o)}] (5) at (1) {};
						\node[vertex,label={[labelsty]left:$2$}] (2) at (-1,0) {};
						\node[vertex,label={[labelsty]left:$3$},rotate around={60:(o)}] (3) at (2) {};
						\node[vertex,label={[labelsty,label distance=-2pt]left:$6$},rotate around={120:(o)}] (6) at (2) {};
						\draw[edge] (1) -- (6);
						\draw[dedge,colds] (1) -- (3);
						\draw[dedge,coldo] (1) -- (4);
						\draw[edge] (5) -- (1);
						\draw[edge] (6) -- (3);
						\draw[edge] (4) -- (6);
						\draw[edge] (6) -- (2);
						\draw[edge] (3) -- (5);
						\draw[dedge,coldo] (2) -- (3);
						\draw[edge] (5) -- (4);
						\draw[dedge,colds] (2) -- (4);
						\draw[edge] (2) -- (5);
		      \end{scope}
		      \begin{scope}[xshift=7.15cm,yshift=-1.2cm,scale=0.8]
		        \coordinate (o) at (0,0);
						\node[vertex,label={[labelsty]right:$1$}] (1) at (1,0) {};
						\node[vertex,label={[labelsty,label distance=-2pt]right:$4$},rotate around={60:(o)}] (4) at (1) {};
						\node[vertex,label={[labelsty]right:$5$},rotate around={120:(o)}] (5) at (1) {};
						\node[vertex,label={[labelsty]left:$2$}] (2) at (-1,0) {};
						\node[vertex,label={[labelsty]left:$3$},rotate around={60:(o)}] (3) at (2) {};
						\node[vertex,label={[labelsty,label distance=-2pt]left:$6$},rotate around={120:(o)}] (6) at (2) {};
						\draw[edge] (1) -- (6);
						\draw[dedge,colds] (1) -- (3);
						\draw[dedge,coldo] (1) -- (4);
						\draw[edge] (5) -- (1);
						\draw[edge] (6) -- (3);
						\draw[edge] (4) -- (6);
						\draw[edge] (6) -- (2);
						\draw[edge] (3) -- (5);
						\draw[dedge,colds] (3) -- (2);
						\draw[edge] (5) -- (4);
						\draw[dedge,coldo] (4) -- (2);
						\draw[edge] (2) -- (5);
		      \end{scope}
		      \draw[ultra thick,->] (8.45,1.2) -- (8.95,1.2);
		      \draw[ultra thick,->] (8.45,-1.2) -- (8.95,-1.2);
		      \begin{scope}[xshift=8.95cm]
		        \node[anchor=west] at (0,1.2) {$B_{\ObondA}^{56}$};
		        \node[anchor=west] at (0,-1.2) {$B_{\ObondC}^{56}$};
		      \end{scope}
		    \end{tikzpicture}
		  \caption{Graphical derivation of the linear relations among $\mu$-numbers of pyramidal and octahedral bonds.}
		  \label{figure:eq_derivation}
		\end{figure}
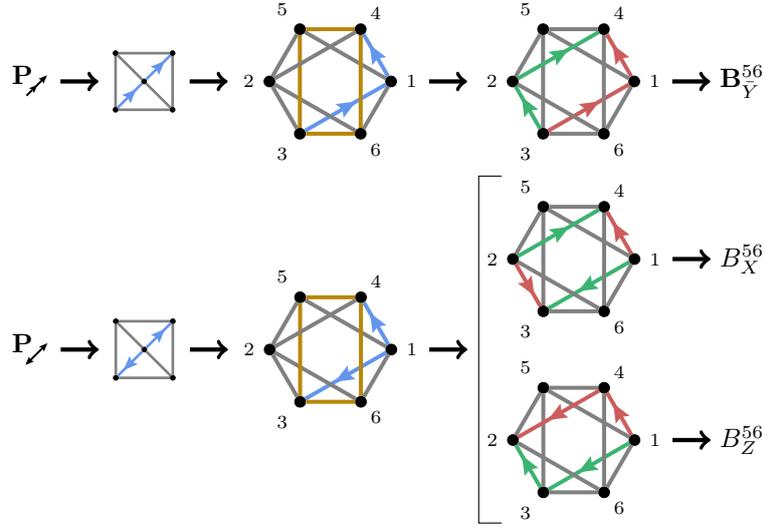
		
		By applying the graphical procedure to all pyramidal bonds, and taking into account the equalities from Rule~\ref{rule:conjugates}, 
		we obtain the following linear system:
		\begin{equation}
		\label{equation:mu}
		\begin{aligned}
			\mu_{\Pbonddl}^{1} &= \mu_{\ObondC}^{34}                & \mu_{\Pbonddl}^{3} &= \mu_{\ObondC}^{56}                & \mu_{\Pbonddl}^{5} &= \mu_{\ObondC}^{12}                \\
			\mu_{\Pbonddo}^{1} &= \mu_{\ObondA}^{34} + \mu_{\ObondB}^{34} & \mu_{\Pbonddo}^{3} &= \mu_{\ObondA}^{56} + \mu_{\ObondB}^{56} & \mu_{\Pbonddo}^{5} &= \mu_{\ObondA}^{12} + \mu_{\ObondB}^{12} \\
			\mu_{\Pbondal}^{1} &= \mu_{\ObondB}^{56}                & \mu_{\Pbondal}^{3} &= \mu_{\ObondB}^{12}                & \mu_{\Pbondal}^{5} &= \mu_{\ObondB}^{34}                \\
			\mu_{\Pbondao}^{1} &= \mu_{\ObondA}^{56} + \mu_{\ObondC}^{56} & \mu_{\Pbondao}^{3} &= \mu_{\ObondA}^{12} + \mu_{\ObondC}^{12} & \mu_{\Pbondao}^{5} &= \mu_{\ObondA}^{34} + \mu_{\ObondC}^{34} \\
			\mu^1_{\bullet} &= \mu^{2}_{\bullet} & \mu^3_{\bullet} &= \mu^{4}_{\bullet} & \mu^5_{\bullet} &= \mu^6_{\bullet}
		\end{aligned}
		\end{equation}
		where $\bullet$ is any of the symbols $\{ \Pbonddl, \Pbonddo, \,\,\Pbondal, \,\,\Pbondao \}$.
\end{enumerate}

With the rules at hand, we are ready to attack the classification.

\subsection{Classification}
\label{classification:classification}
From now on, we suppose that the hypothesis in Rule~\ref{rule:equations} holds, 
namely that we are given a motion of an octahedron
and that all pyramids are simple.
At the end of the section we analyze the cases when 
some pyramids are not simple.
We distinguish four cases, parametrized by the sums of the $\mu$-numbers of octahedral bonds.

\begin{definition}
 For each quadrilateral~$ij$ in~$\GOct$, we define~$\mu^{ij}$ to be the quantity:
 \[
  \mu^{ij} := \mu_{\ObondA}^{ij} + \mu_{\ObondB}^{ij} + \mu_{\ObondC}^{ij} +\mu_{\ObondAc}^{ij} + \mu_{\ObondBc}^{ij} + \mu_{\ObondCc}^{ij} \stackrel{\ref{rule:conjugates}}{=} 
  2 (\mu_{\ObondA}^{ij} + \mu_{\ObondB}^{ij} + \mu_{\ObondC}^{ij}) \,.
 \]
\end{definition}

\begin{lemma}
\label{lemma:four_cases}
 There are only $4$ possibilities (up to swapping quadrilaterals) for the numbers $(\mu^{12}, \mu^{34}, \mu^{56})$:
\[
 (\mu^{12}, \mu^{34}, \mu^{56}) \in 
 \bigl\{  
  (4,4,4),
  (4,4,2),
  (4,2,2),
  (2,2,2)
 \bigr\} \,.
\]
\end{lemma}
\begin{proof}
By Table~\ref{table:bonds} from Rule~\ref{rule:table}, we have $1 \leq \mu_{\Pbonddl}^{v} + \mu_{\Pbonddo}^{v} \leq 2$ for every $v \in \{1, \dotsc, 6\}$, 
and similarly for $\mu_{\Pbondal}^{v} + \mu_{\Pbondao}^{v}$. 
It follows by Equation~\eqref{equation:mu} from Rule~\ref{rule:equations} that $\mu^{ij} \in \{2, 4\}$ for all $ij \in \{12, 34, 56\}$. 
The statement is then proven.
\end{proof}

Now we analyze the cases from Lemma~\ref{lemma:four_cases} one by one.

\begin{description}
  \item[Case $(4,4,4)$:]
   From Equation~\eqref{equation:mu}, we know that for all quadrilaterals~$ij$ in~$\GOct$
   \begin{align*}
    \mu_{\ObondA}^{ij} + \mu_{\ObondB}^{ij} &= \mu_{\Pbonddo}^k \in \{0, 1\} \text{ for a suitable } k\,, \\
    \mu_{\ObondA}^{ij} + \mu_{\ObondC}^{ij} &= \mu_{\Pbondao}^\ell \in \{0, 1\} \text{ for a suitable } \ell\,.
   \end{align*}
   Moreover, by assumption we have
   \[
    2 (\mu_{\ObondA}^{ij} + \mu_{\ObondB}^{ij} + \mu_{\ObondC}^{ij}) = 4 \,.
   \]
   This implies
   \[
    \mu_{\ObondA}^{ij} = 0, \quad
    \mu_{\ObondB}^{ij} = \mu_{\ObondC}^{ij} = 1 \,.
   \]
   The equations on the edge lengths from Rule~\ref{rule:lengths} imposed by the fact that $\mu_{\ObondB}^{ij} = \mu_{\ObondC}^{ij} = 1$ are, in the case $ij = 56$:
   \begin{align*}
    \ell_{13} - \ell_{32} - \ell_{24} + \ell_{41} &= 0, \\
    \ell_{13} + \ell_{32} - \ell_{24} - \ell_{41} &= 0.
   \end{align*}
   This implies that $\ell_{13} = \ell_{24}$ and $\ell_{32} = \ell_{41}$. 
Namely, opposite edges in the three quadrilaterals of~$\GOct$ have the same length (see Figure~\ref{figure:same_length}). 
   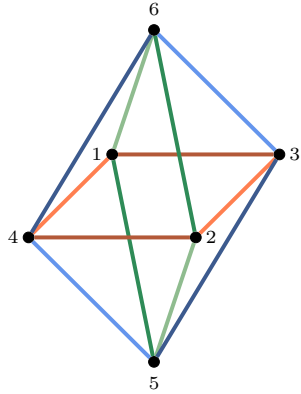
\begin{figure}[ht]
		\centering
			\begin{tikzpicture}[scale=1.1,baseline={(0,0)}]
				\node[vertex,label={[labelsty,label distance=-2pt]left:$1$}] (1) at ($(-1,0)+1/2*(1,1)$) {};
				\node[vertex,label={[labelsty,label distance=-2pt]right:$2$}] (2) at ($(1,0)-1/2*(1,1)$) {};
				\node[vertex,label={[labelsty,label distance=-2pt]right:$3$}] (3) at ($(1,0)+1/2*(1,1)$) {};
				\node[vertex,label={[labelsty,label distance=-2pt]left:$4$}] (4) at ($(-1,0)-1/2*(1,1)$) {};
				\node[vertex,label={[labelsty]below:$5$}] (5) at (0,-2) {};
				\node[vertex,label={[labelsty]above:$6$}] (6) at (0,2) {};
				\draw[edge,cole6] (1)edge(3);
				\draw[edge,cole4] (1)edge(4);
				\draw[edge,cole3] (1)edge(5);
				\draw[edge,cole5] (1)edge(6);
				\draw[edge,cole4] (2)edge(3);
				\draw[edge,cole6] (2)edge(4);
				\draw[edge,cole5] (2)edge(5);
				\draw[edge,cole3] (2)edge(6);
				\draw[edge,cole2] (3)edge(5);
				\draw[edge,cole1] (3)edge(6);
				\draw[edge,cole1] (4)edge(5);
				\draw[edge,cole2] (4)edge(6);
			\end{tikzpicture}
    \caption{The edge length situation in Case~$(4,4,4)$: equal color corresponds to equal length.}
    \label{figure:same_length}
   \end{figure}
   Now notice that a parameter count shows that 
   an octahedron whose opposite edges in each quadrilateral have equal length possesses a line-symmetric motion. 
   Since all pyramids are simple, there is exactly one way in the motion under consideration to extend a realization of a pyramid. 
   Since all pyramids are general, each of them admits exactly one motion. 
   Therefore, such unique extension must be in the line-symmetric motion. 
  \item[Case $(4,4,2)$:]
   From Equation~\eqref{equation:mu} and Table~\ref{table:bonds} from Rule~\ref{rule:table}, 
   we infer that the pyramids \quadrefE\ and \quadrefF\ are general, 
   while \quadrefA\ and \quadrefB\ are odd deltoids and \quadrefC\ and \quadrefD\ are even deltoids. 
   Moreover, from the fact that $\mu^{12} = \mu^{34}= 4$, we deduce as in Case $(4,4,4)$  
   that the opposite edges in the quadrilaterals~$12$ and~$34$ have the same length. 
   We now show that the opposite edges in the quadrilateral~$56$ have the same length, 
   so as in Case $(4,4,4)$ we conclude that we have a Type~I flexible octahedron. 
   Consider a realization for which the pyramid~\quadrefA\ is flat;
   then we have that $1$, $3$, and~$4$ are collinear. 
   Let us now look at the pyramid~\quadrefC\ for that realization: we would like to conclude that \quadrefC\ is flat as well. 
   Since \quadrefA\ is flat, we have that the dihedral angle between the faces $135$ and $136$ is either $0$ or~$\pi$; 
   however, this is a simple angle for~\quadrefC, hence by Lemma~\ref{fact:flat_deltoid} also \quadrefC\ is flat. 
Therefore, the vertices~$1$, $2$, $3$, and~$4$ are collinear in that realization, and all the vertices are coplanar.
Then the quadrilateral~$34$ is, in that realization, a parallelogram or an antiparallelogram (see Figure~\ref{figure:collinear_flat}).
Thus, the footpoint of the midpoint of the diagonal $\{5,6\}$ on the line $1234$ is the midpoint of the diagonal~$\{1,2\}$. 
By considering the quadrilateral~$12$, we get that the footpoint of the midpoint of the diagonal $\{5,6\}$ on the line $1234$ is the midpoint of the diagonal~$\{3,4\}$.
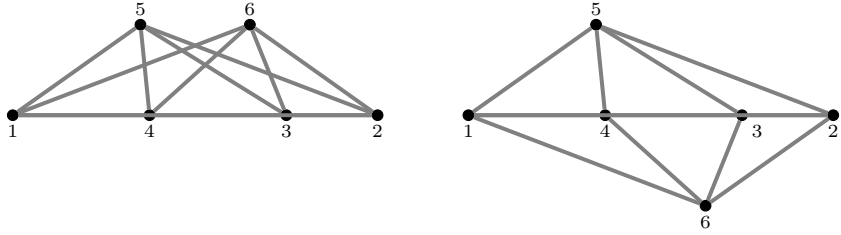
\begin{figure}[ht]
	\centering
	\begin{tikzpicture}[scale=1.2,baseline={(0,0)}]
		\node[vertex,label={[labelsty,label distance=-2pt]below:$1$}] (1) at (-2,0) {};
		\node[vertex,label={[labelsty,label distance=-2pt]below:$2$}] (2) at (2,0) {};
		\node[vertex,label={[labelsty,label distance=-2pt]below:$3$}] (3) at (1,0) {};
		\node[vertex,label={[labelsty,label distance=-2pt]below:$4$}] (4) at (-0.5,0) {};
		\node[vertex,label={[labelsty,label distance=-2pt]above:$5$}] (5) at (-0.6,1) {};
		\node[vertex,label={[labelsty,label distance=-2pt]above:$6$}] (6) at (0.6,1) {};
		\draw[edge] (2)edge(5);
		\draw[edge] (2)edge(6);
		\draw[edge] (2)edge(3);
		\draw[edge] (2)edge(4);
		\draw[edge] (3)edge(5);
		\draw[edge] (3)edge(6);
		\draw[edge] (4)edge(5);
		\draw[edge] (4)edge(6);
		\draw[edge] (3)edge(1);
		\draw[edge] (4)edge(1);
		\draw[edge] (5)edge(1);
		\draw[edge] (6)edge(1);
	\end{tikzpicture}
	\qquad
	\begin{tikzpicture}[scale=1.2,baseline={(0,0)}]
		\node[vertex,label={[labelsty,label distance=-2pt]below:$1$}] (1) at (-2,0) {};
		\node[vertex,label={[labelsty,label distance=-2pt]below:$2$}] (2) at (2,0) {};
		\node[vertex,label={[labelsty,label distance=-2pt]-40:$3$}] (3) at (1,0) {};
		\node[vertex,label={[labelsty,label distance=-2pt]below:$4$}] (4) at (-0.5,0) {};
		\node[vertex,label={[labelsty,label distance=-2pt]above:$5$}] (5) at (-0.6,1) {};
		\node[vertex,label={[labelsty,label distance=-2pt]below:$6$}] (6) at (0.6,-1) {};
		\draw[edge] (2)edge(5);
		\draw[edge] (2)edge(6);
		\draw[edge] (2)edge(3);
		\draw[edge] (2)edge(4);
		\draw[edge] (3)edge(5);
		\draw[edge] (3)edge(6);
		\draw[edge] (4)edge(5);
		\draw[edge] (4)edge(6);
		\draw[edge] (3)edge(1);
		\draw[edge] (4)edge(1);
		\draw[edge] (5)edge(1);
		\draw[edge] (6)edge(1);
	\end{tikzpicture}
	\caption{Flat realization in Case $(4,4,2)$: all vertices are coplanar, four of them are collinear, and the quadrilateral $34$ can be a parallelogram or an antiparallelogram.}
	\label{figure:collinear_flat}
\end{figure}
Hence we obtain
\[
 \ell_{13} = \ell_{24}
 \quad \text{and} \quad
 \ell_{41} = \ell_{32} \,.
\]
Thus this case is a special case of a Type~I flexible octahedron allowing a flat realization.
  \item[Case $(4,2,2)$:]
   Here we see that the pyramids \quadrefC\ and \quadrefD\ are even deltoids, and the pyramids~\quadrefE\ and \quadrefF\ are odd deltoids, 
while the pyramids~\quadrefA\ and \quadrefB\ are either rhomboids or lozenges. 
Let us suppose that we are in a flat realization for the pyramid~\quadrefC. 
Then the rhomboid~\quadrefB\ has one of the angles which is $0$ or~$\pi$, hence it is flat as well. 
This implies that we have two flat realizations for the octahedron as a whole. 
Since we have deltoids, as in Case $(4,4,2)$ we have collinearities in a flat realization, 
namely the following triples of vertices are collinear 
(keep into account that \quadrefC\ and \quadrefD\ are even deltoids, while \quadrefE\ and \quadrefF\ are odd deltoids):
\[
 \{1, 2, 3\} \qquad
 \{1, 2, 6\} \qquad
 \{1, 2, 5\} \qquad
 \{1, 2, 4\} \,. 
\]
Therefore, all the vertices are collinear, unless in this special flat realization we have that $1$ and $2$ coincide%
\footnote{Recall that we forbid two vertices to coincide for a general realization in a motion, 
but they are allowed to coincide in special realizations.}. 
If the vertices are collinear in this special realization, then all the triangular faces are degenerate,
and so all vertices are collinear in any realization of the motion, 
but in this case the octahedron cannot move at all. 
Hence only the situation where $1$ and $2$ coincide can happen (see Figure~\ref{figure:global_flat}). 
\begin{figure}[ht]
	\centering
	\begin{tikzpicture}[scale=1.2]
		\node[vertex,label={[labelsty,label distance=-2pt]left:$3$}] (3) at (-2,3) {};
		\node[vertex,label={[labelsty,label distance=-2pt]right:$4$}] (4) at (2,3) {};
		\node[vertex,label={[labelsty,label distance=-2pt]left:$5$}] (5) at (-1.5,1) {};
		\node[vertex,label={[labelsty,label distance=-2pt]right:$6$}] (6) at (1.5,1) {};
		\node[vertex,label={[labelsty]below:$1=2$}] (1) at (0,0) {};
		\draw[edge] (3)edge(5);
		\draw[edge] (3)edge(6);
		\draw[edge] (4)edge(5);
		\draw[edge] (4)edge(6);
		\draw[edge] (3)edge(1);
		\draw[edge] (4)edge(1);
		\draw[edge] (5)edge(1);
		\draw[edge] (6)edge(1);
	\end{tikzpicture}
	\caption{Global flat realization of an octahedron in Case~$(4,2,2)$: 
	the vertices~$1$ and~$2$ must coincide in this realization.}
	\label{figure:global_flat}
\end{figure}
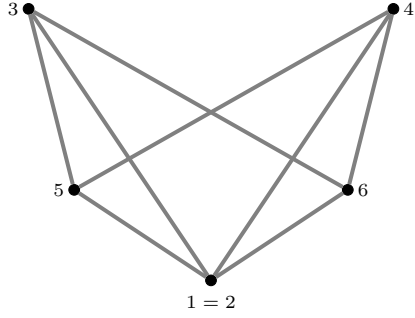

For this situation to happen, we must have
\[
 \ell_{16} = \ell_{26} \,,
 \quad
 \ell_{13} = \ell_{32} \,,
 \quad
 \ell_{41} = \ell_{24} \,,
 \quad
 \ell_{15} = \ell_{25} \,.
\]
Moreover, the fact that $\mu^{12} = 4$ implies, as in Case $(4,4,4)$, that
\[
 \ell_{36} = \ell_{45}
 \quad \text{and} \quad 
 \ell_{46} = \ell_{35} \,.
\]
Altogether, this implies that for a general realization in this motion the vertices $3$, $4$, $5$, and~$6$ are coplanar 
and that $1$ and $2$ are symmetric with respect to the plane spanned by the coplanar vertices.
Moreover, the planar quadrilateral~$12$ is either a parallelogram or an antiparallelogram. 
Furthermore, from the fact that in the global flat realization of the octahedron the vertices~$1$ and~$2$ coincide,
it follows that all the deltoids are of ``coinciding'' type. 
Using Table~\ref{table:bonds} from Rule~\ref{rule:table} we get that for the two odd deltoids
\[
 \mu_{\Pbondar} = 1 
 \quad \text{and} \quad
 \mu_{\Pbondao} = 0 \,,
\]
while for the two even deltoids
\[
 \mu_{\Pbonddr} = 1 
 \quad \text{and} \quad
 \mu_{\Pbonddo} = 0 \,.
\]
Therefore by Equations~\eqref{equation:mu} from Rule~\ref{rule:equations}, we obtain
\begin{align*}
 \mu_{\ObondC}^{56} &= 1, & \mu_{\ObondA}^{56} = \mu_{\ObondB}^{56} &= 0, \\
 \mu_{\ObondB}^{34} &= 1, & \mu_{\ObondA}^{34} = \mu_{\ObondC}^{34} &= 0.
\end{align*}
By using Rule~\ref{rule:lengths} we get the constraints
\[
 -\ell_{41} - \ell_{24} + \ell_{32} + \ell_{13} = 0
 \quad \text{and} \quad 
 \ell_{25} + \ell_{51} - \ell_{16} - \ell_{62} = 0  \,.
\]
Taking into account the previous relations between lengths, these imply the equalities
\[
 \ell_{32} = \ell_{24}
 \quad \text{and} \quad
 \ell_{25} = \ell_{16} \,.
\]
Altogether, these equations imply that, if the quadrilateral~$12$ is an antiparallelogram, then the projection 
of the vertices~$1$ and~$2$ on the plane spanned by~$3$,$4$,$5$, and~$6$ lies, for all realizations of the motion, on the 
symmetry line of the antiparallelogram. Hence we get a Type II flexible octahedron. 
If the quadrilateral~$12$ were a parallelogram, then the projection of the vertices~$1$ and~$2$ would be at the 
intersection of its two symmetry axes; but then we would have a convex flexible octahedron, 
and this conflicts with Cauchy's theorem.
\item[Case $(2,2,2)$:]
  In this case, all the $6$ pyramids are rhomboids or lozenges.
  Moreover, we have
  \[
   \mu_{\ObondA}^{ij} + \mu_{\ObondB}^{ij} + \mu_{\ObondC}^{ij} = 1
  \]
  for any $ij \in \{12, 34, 56\}$, and so exactly one of these three quantities equals~$1$, 
  while the other two are zero. 
  We hence obtain three linear constraints for the edge lengths, 
  one for each of the three quadrilaterals in~$\GOct$. 
  Therefore we have an octahedron of Type~III.
 \end{description}
 
The classification when all the pyramids are simple is then completed.
We conclude this section by showing that we can always reduce to the simple case.
Let us describe this reduction procedure as follows.

\begin{reduction}
Suppose that a pyramid, say~\quadrefA, is not simple. 
This means that there exist at least two realizations of the octahedron extending a general realization of~\quadrefA.
This implies that in all those realizations the points $3$, $4$, $5$, and~$6$ must be coplanar.
Then we construct another octahedron by substituting the realization of vertex~$2$ 
with the mirror of the realization of the vertex~$1$ with respect to the plane spanned by~$3$, $4$, $5$, and~$6$; see Figure~\ref{figure:reduction}.
By the hypothesis on the initial octahedron we get that the new octahedron is flexible, 
and it has the further property that pyramids~\quadrefA\ and~\quadrefB\ are simple. 
Here the fact that \quadrefA\ and \quadrefB\ are simple is ensured by Assumption~\ref{assumption}, 
which prevents different vertices from having the same realization.
\end{reduction}

\begin{figure}[ht]
 \centering
 \includegraphics[width=.4\textwidth]{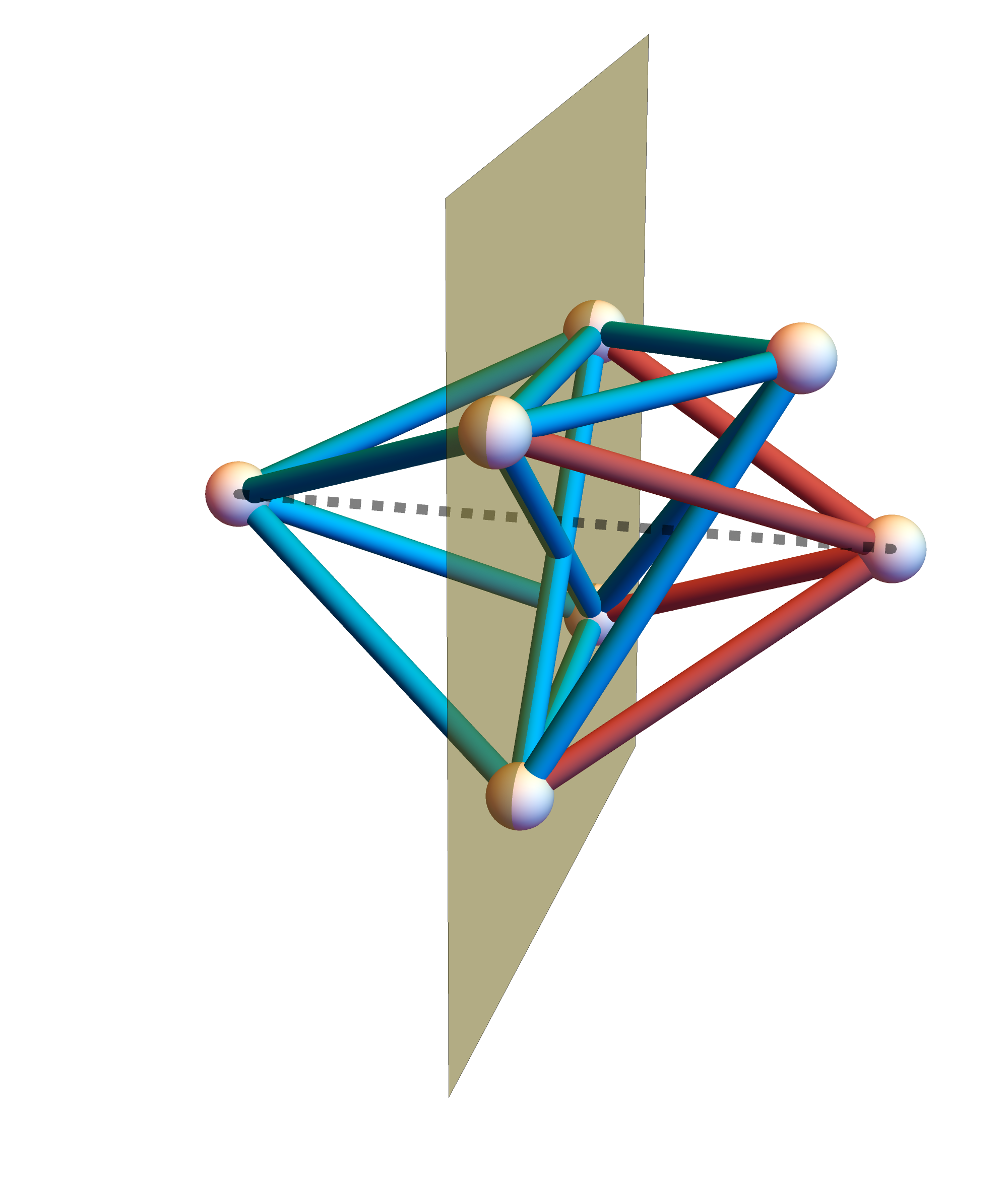}
 \caption{An illustration of the reduction process: the original octahedron (in blue) is transformed into one where the red pyramid substitutes the blue pyramid on the right of the planar quadrilateral.}
 \label{figure:reduction}
\end{figure}

We claim that we can repeat this procedure finitely many times (actually, three times)
and obtain a situation where all the pyramids are simple.
In fact, notice that the reduction process preserves coplanarity in the following sense.
Suppose that pyramid~\quadrefA\ is not simple and apply the reduction.
This means that vertices $3,4,5,6$ are coplanar, and now $1$ and~$2$ are symmetric with respect to that plane,
so in particular they lie on a perpendicular line to the plane $3456$.
Suppose, furthermore, that after the reduction pyramid~\quadrefC\ is not simple, thus $1, 2, 5, 6$ are coplanar. 
In this situation, the mirror of~$3$ with respect to the plane $1256$ equals
the mirror of~$3$, in the plane $3456$, with respect to the line spanned by~$5$ and~$6$.
Hence after the second reduction, we have that $3,4,5,6$ are coplanar, and $1,2,5,6$ are coplanar.
Therefore the reduction can be applied only thrice. 

As a by-product of the previous reduction we have that when four vertices of the octahedron are coplanar,
the other two vertices are symmetric with respect to that plane. 
This implies that after the reduction the four pyramids~\quadrefV\ with vertices~$v$ on that plane can only be deltoids or lozenges. 

To refine the by-product stated in the last paragraph,
we introduce the notion of \emph{multiplicity} of an edge of the octahedron.
A specific rule discusses the behavior of edge multiplicity.

\begin{definition}
\label{definition:multiplicity}
 Consider a motion of an octahedron. 
 The \emph{multiplicity} of an edge of the octahedron is the number (up to rotations and translations)
 of realizations of the octahedron that have the same general value of the dihedral  angle 
 between the two triangular faces adjacent to the edge. 
\end{definition}

\begin{enumerate}\renewcommand{\theenumi}{\textbf{R\arabic{enumi}}}\renewcommand{\labelenumi}{\theenumi:}\setcounter{enumi}{4}
  \item\label{rule:edge_multiplicity}
  Edges may have multiplicity~$1$, $2$, or~$4$. 
  Two opposite edges of a pyramid~\quadrefV\ incident to~$v$ have the same multiplicity;
  hence all the edges in a quadrilateral of~$\GOct$ have the same multiplicity.
  The multiplicity of two neighboring edges incident to~$v$ in a pyramid~\quadrefV\ 
  may at most differ by a factor of~$2$.
  A general pyramid has all edges of multiplicity~$2$ or~$4$. 
  The edges of a deltoid have multiplicity $(2,4)$ or $(1,2)$.
  The edges of a rhomboid or a lozenge have all multiplicity~$1$ or all multiplicity~$2$.
\end{enumerate}

With the notion of multiplicity at hand, we can say that
if we apply the reduction at pyramid~\quadrefA, 
then \quadrefC, \quadrefD, \quadrefE, and~\quadrefF\ are deltoids --- since they are symmetric
with respect to the plane $3456$ --- whose edges incident to~$1$ or~$2$ are simple 
and whose other edges are double, or lozenges with only simple edges.

We now explore all three possible cases that may appear after the reduction,
namely we can have three, two, or one planar quadrilateral in the octahedron.

It is easy to see that there cannot be three planar quadrilaterals: 
all vertices would have to lie on coordinate axes, 
and Pythagoras' Theorem would give an easy proof of rigidity.

Assume that the vertices $3, 4, 5, 6$ are coplanar and $1, 2, 5, 6$ are coplanar as well. 
By the above properties of the octahedron, it follows that all edges are simple and all pyramids are lozenges. 
Thus the reduced octahedron belongs to Case $(2, 2, 2)$. 
Therefore it has two flat realizations. 
However, when a lozenge is in a flat realization, then two opposite edges have to coincide. 
This, however, cannot happen for all lozenges.
In fact, in a flat configuration either the points~$1$ and~$2$, or the points~$3$ and~$4$ must coincide,
since the planes $3456$ and $1256$ are orthogonal to each other in a general realization of the motion,
and $1$ and $2$ are symmetric, as well as $3$ and $4$. 
Moreover, in any case the points $1=2$ or $3=4$ are collinear with $5$ and $6$ in the flat realization.
For simplicity, let us suppose to be in a flat position where $3$ and $4$ coincide. 
Hence the situation is the one depicted in Figure~\ref{figure:flatpos} (recall that $1$ and $2$ are symmetric with respect to the line $56$).
\begin{figure}[ht]
  \centering
  \begin{tikzpicture}[baseline={(0,0)},scale=1.1]
    \clip (-0.75,-1.2) rectangle (3.35,1.5);
    \node[vertex,label={[labelsty]left:$1=2$}] (1) at (0,0) {};
    \node[vertex,label={[labelsty]30:$3$}] (3) at (1,1) {};
    \node[vertex,label={[labelsty]left:$4$}] (4) at (1,-1) {};
    \node[vertex,label={[labelsty]right:$6$}] (6) at (3,0) {};
    \coordinate (o5) at ($(6)!1.5!(3)$);
    \coordinate (base5) at ($(1)!(o5)!(3)$);
    \coordinate (dir5) at ($(o5)!5!(base5)$);
    \path[name path=edge16] (1)--(6);
    \path[name path=edge3dir5] (base5)--(dir5);
    \path[name intersections={of=edge16 and edge3dir5, by={n5}}]; 
    \node[vertex,label={[labelsty,label distance=-2pt]-30:$5$}] (5) at (n5) {};    
    \begin{scope}
      \clip (3) -- ($(3)!-1!(6)$) -- ($(3)!0.5!(1)$) -- cycle;
      \draw[col5] (3) circle[radius=0.5];
    \end{scope}
    \begin{scope}
      \clip ($(3)$) -- ($(1)$) -- ($(6)$) -- ($(3)$) -- cycle;
      \draw[col6] (3) circle[radius=0.55];
    \end{scope}
    \node[pin={[labelsty,pin distance=0.6cm,col5,pin edge={col5,shorten <=0.25cm}]182:$\beta$}] at (3) {};
    \node[pin={[labelsty,pin distance=0.5cm,col6,pin edge={col6,shorten <=0.25cm,shorten >=-0.1cm}]-90:$\pi-\beta$}] at (3) {};
    \draw[dashed] (3)edge($(3)!-0.5!(6)$);
    \draw[edge] (1)edge(3) (1)edge(4) (1)edge(5) (1)edge(6) (3)edge(5) (3)edge(6) (4)edge(5) (4)edge(6);
  \end{tikzpicture}
  \begin{tikzpicture}[baseline={(0,0)},scale=1.1]
    \node[vertex,label={[labelsty]left:$1$}] (1) at (0,0) {};
    \node[vertex,label={[labelsty,label distance=-2pt]20:$3=4$}] (3) at (1,1) {};
    \node[vertex,label={[labelsty]right:$6$}] (6) at (3,0) {};
    \node[vertex,label={[labelsty,label distance=1pt]left:$5$}] (5) at ($(6)!1.5!(3)$) {};
    \coordinate (old2) at (0,0);
    \coordinate (base2) at ($(3)!(old2)!(6)$);
    \node[vertex,label={[labelsty]right:$2$}] (2) at ($(old2)!2!(base2)$) {};
    \begin{scope}
      \clip (3) -- ($(3)!-1!(6)$) -- ($(3)!0.5!(1)$) -- cycle;
      \draw[col5] (3) circle[radius=0.5];
    \end{scope}
    \begin{scope}
      \clip ($(3)$) -- ($(1)$) -- ($(6)$) -- ($(3)$) -- cycle;
      \draw[col6] (3) circle[radius=0.55];
    \end{scope}
    \node[pin={[labelsty,pin distance=0.6cm,col5,pin edge={col5,shorten <=0.25cm}]182:$\beta$}] at (3) {};
    \node[pin={[labelsty,pin distance=0.5cm,col6,pin edge={col6,shorten <=0.25cm,shorten >=-0.1cm}]-90:$\pi-\beta$}] at (3) {};
    \draw[edge] (1)edge(3) (1)edge(5) (1)edge(6) (2)edge(3) (2)edge(5) (2)edge(6) (3)edge(5) (3)edge(6);
  \end{tikzpicture}
  \caption{Flat positions of an octahedron obtained by applying the reduction process twice. This case, actually, does never occur.}
  \label{figure:flatpos}
\end{figure}
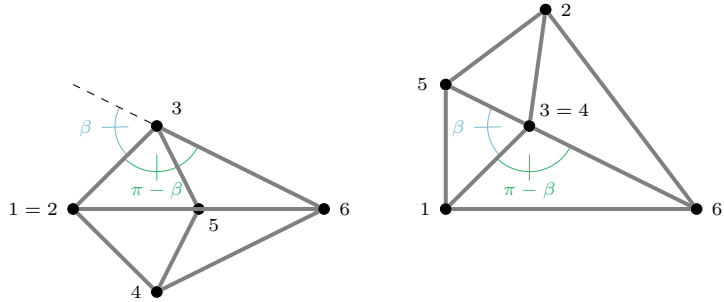
We now show that, in this situation, the octahedron is actually rigid, so this case can never happen.
In fact, we prove that the constraints derived from the flat realization, together with the fact
that we have six lozenges, are not compatible with the orthogonality of the planes $1256$ and $3456$. 
To show this, we focus on the dihedral angle at the edge~$23$: 
using the constraints from the flat realization, we can fix (up to scaling) vertex $2$ to be at $(0,0,0)$, 
vertex $3$ to be at $(b,0,0)$ (for some $b \in \R_{>0} \setminus \{2\}$) for the whole motion; 
we parametrize vertex $5$ as $\bigl(1,-r \cos(t),r \sin(t)\bigr)$ for some $r \in \R_{>0}$, 
so vertex~$6$ has coordinates $\bigl(\frac{b}{2-b}, \frac{br}{2-b}, 0\bigr)$; see Figure~\ref{figure:parametrization}.
\begin{figure}[ht]
\centering
  \begin{tikzpicture}[scale=1.6]
    \node[vertex,label={[labelsty]left:$2$}] (2) at (0,0) {};
    \coordinate (a) at (1,0);
    \node[vertex,label={[labelsty]right:$3$}] (3) at (1.2,0) {};
    \node[vertex,label={[labelsty]below:$5$}] (5) at (1,-0.5) {};
    \node[vertex,label={[labelsty]above:$6$}] (6) at (1.5,0.75) {};
    \begin{scope}
      \clip ($(2)$) -- ($(6)$) -- ($(3)$) -- ($(2)$) -- cycle;
      \draw[col6] (2) circle[radius=0.4];
    \end{scope}
    \begin{scope}
      \clip ($(2)$) -- ($(5)$) -- ($(3)$) -- ($(2)$) -- cycle;
      \draw[col6] (2) circle[radius=0.45];
    \end{scope}
    \node[pin={[labelsty,pin distance=0.6cm,col6,pin edge={col6,shorten <=0.25cm}]6:$\beta$}] at (2) {};
    \node[pin={[labelsty,pin distance=0.7cm,col6,pin edge={col6,shorten <=0.35cm}]-6:$\beta$}] at (2) {};
    \draw[edge] (2)--(3)--(6)--(2);
    \draw[edge] (2)--(5)--(3);
  \end{tikzpicture}
  \caption{To show that the case of two reductions does not occur, we focus n the dihedral angle at edge~$23$: in the flat position, the angles $\hat{325}$ and $\hat{326}$ are equal, and the vertices $3$, $5$, and~$6$ are collinear.}
  \label{figure:parametrization}
\end{figure}
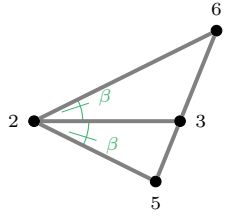
However, a computation shows that the inner product between two normal vectors to the planes~$356$ and~$256$
during the motion is not 0 for any choice of~$b$ and~$r$. 
This is not compatible with the fact that the planes $1256$ and $3456$ are orthogonal.

Assume now that only $3, 4, 5, 6$ are coplanar. 
By what we said before, this means we applied the reduction process only once.
Then the eight edges incident to~$1$ or~$2$ are simple. 
The remaining four edges can either be simple or be double. 
We distinguish two cases.

\begin{description}
 \item[\textbf{Case A}.]
 All edges are simple. 
 Then by Rule~\ref{rule:edge_multiplicity} we are again in Case $(2, 2, 2)$, 
 now with four lozenges (namely~\quadrefC, \quadrefD, \quadrefE, and~\quadrefF) and two rhomboids or lozenges (namely~\quadrefA\ and~\quadrefB). 
 Since we are in Case $(2,2,2)$, we have two flat realizations. In one of them, the vertices $1$ and $2$ coincide.
 In the other, using the fact that~\quadrefC, \quadrefD, and \quadrefE\ are lozenges, we have that $3$, $4$, $5$, and $6$ are collinear.
 In this realization, then $1$ and $2$ are symmetric with respect to the line $3456$. 
 Now we invoke the last of our rules:
 \begin{enumerate}\renewcommand{\theenumi}{\textbf{R\arabic{enumi}}}\renewcommand{\labelenumi}{\theenumi:}\setcounter{enumi}{5}
  \item\label{rule:parallelogram}
  In the situation of Case~A, the plane quadrilateral~$12$ is either an antiparallelogram or a parallelogram.
 \end{enumerate} 
 Consider the flat realization of the octahedron where the four vertices $3$, $4$, $5$, and~$6$ are collinear. 
 Because the plane quadrilateral~$12$ is an antiparallelogram or a parallelogram, 
 it follows that the edges~$35$ and~$46$ are equal in length. 
 Because the pyramid~$\quadrefA$ is a rhomboid or a lozenge, 
 it follows that the angles at~$1$ in the two triangles~$135$ and~$146$ are equal --- 
 they could not be supplementary because this would contradict collinearity of $3,4,5,6$. 
 Hence the triangles $135$ and $146$ have one side in common, the opposite angle in common,
 and the normal height in common. 
 It follows that the two triangles are congruent. 
 It follows that, for all configurations, the footpoint of vertex~$1$ to the plane lies on 
 the symmetry line of the antiparallelogram or in the midpoint of the parallelogram, depending whether the plane
 quadrilateral~$12$ is an antiparallelogram or a parallelogram. 
 Then the footpoint of vertex~$1$ lies in the symmetry line of the plane antiparallelogram~$12$ or in the midpoint of the parallelogram,  
 also for the original octahedron.
 The same holds for the footpoint of vertex~$2$, analogously. 
 It follows that the original octahedron, before the reduction process, is plane-symmetric.
 \item[\textbf{Case B}.] 
 The four edges are double. 
 By Rule~\ref{rule:edge_multiplicity} we have four deltoids and two rhomboids or lozenges,
 thus we are in the $(4, 2, 2)$ case. 
 Then the plane quadrilateral~$12$ is an antiparallelogram,
 and the footpoint of vertex~$1$ to the plane lies on the symmetry line of the antiparallelogram.
 Say we had before reduced by replacing~$2$ by the mirror of~$1$ at the plane~$3456$.
 Then the footpoint of vertex~$1$ lies in the symmetry line of the plane antiparallelogram~$12$, 
 also for the original octahedron.
 The same holds for the footpoint of vertex~$2$, analogously. 
 It follows that the original octahedron is plane-symmetric.
\end{description}

\section{From the space to the sphere}
\label{reduction}

Now that we showed that the classification of flexible octahedra can be achieved once 
we accept the rules introduced in Section~\ref{classification}, 
we are left with the task of explaining why the rules are correct.

We start by reducing the problem of flexibility of octahedra to a 
problem of flexibility of graphs on the sphere,
as in \cite{Izmestiev2017,Kokotsakis1933,Stachel2010}. 
For each realization in $3$-space of an octahedron compatible with a given edge labeling, 
the normalized vectors of the edges define a configuration of points on the unit sphere. 
For any triangular face of the octahedron, the angle between two edge vectors is determined
by the edge lengths of the octahedron. 
Let us define~$\GEdg$ to be the graph whose vertices are the edges of~$\GOct$, 
and where two vertices are connected by an edge 
when the corresponding edges in~$\GOct$ belong to the same triangular face of the octahedron 
(see Figure~\ref{figure:graph_edg}).
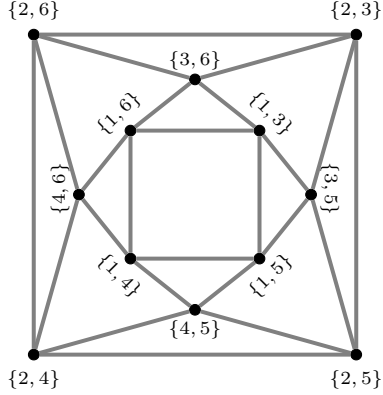
\begin{figure}[ht]
	\centering
	\begin{tikzpicture}[scale=1.2]
		\node[vertex,label={[labelsty,rotate=-45,label distance=-1pt]above:$\{1,3\}$}] (1) at (0.707107, 0.707107) {};
		\node[vertex,label={[labelsty,rotate=-45,label distance=-1pt]below:$\{1,4\} $}] (2) at (-0.707107, -0.707107) {};
		\node[vertex,label={[labelsty,rotate=45,label distance=-1pt]below:$\{1,5\}$}] (3) at (0.707107, -0.707107) {};
		\node[vertex,label={[labelsty,rotate=45,label distance=-1pt]above:$\{1,6\}$}] (4) at (-0.707107, 0.707107) {};
		\node[vertex,label={[labelsty]above:$\{2,3\}$}] (5) at (1.76777, 1.76777) {};
		\node[vertex,label={[labelsty]below:$\{2,4\}$}] (6) at (-1.76777, -1.76777) {};
		\node[vertex,label={[labelsty]below:$\{2,5\}$}] (7) at (1.76777, -1.76777) {};
		\node[vertex,label={[labelsty]above:$\{2,6\}$}] (8) at (-1.76777, 1.76777) {};
		\node[vertex,label={[labelsty,rotate=-90]above:$\{3,5\}$}] (9) at (1.27279, 0.) {};
		\node[vertex,label={[labelsty,label distance=-2pt]above:$\{3,6\}$}] (10) at (0., 1.27279) {};
		\node[vertex,label={[labelsty,label distance=-2pt]below:$\{4,5\}$}] (11) at (0., -1.27279) {};
		\node[vertex,label={[labelsty,rotate=90]above:$\{4,6\}$}] (12) at (-1.27279, 0.) {};
		\draw[edge] (1)edge(3) (1)edge(4) (1)edge(9) (1)edge(10) (2)edge(3) (2)edge(4) (2)edge(11) (2)edge(12) (3)edge(9) (3)edge(11) (4)edge(10) (4)edge(12) (5)edge(7) (5)edge(8) (5)edge(9) (5)edge(10) (6)edge(7) (6)edge(8) (6)edge(11) (6)edge(12) (7)edge(9) (7)edge(11) (8)edge(10) (8)edge(12);
	\end{tikzpicture}
 \caption{The graph~$\GEdg$: its vertices are the edges of the octahedron, and two vertices are adjacent if they come from the same face of the octahedron.}
 \label{figure:graph_edg}
\end{figure}
From the previous discussion we get that a labeling for the edges of~$\GOct$ induces a labeling of the edges of~$\GEdg$ 
given by the cosine of the angles between edge vectors belonging to the same face.
In formulas, if $\lambda$ is the labeling for~$\GOct$, then the induced labeling for~$\GEdg$ is the map:
\[
 (\{i,j\}, \{m,j\}) 
 \mapsto 
 -\frac{\lambda_{\{i,m\}}^2 - \lambda_{\{i,j\}}^2 - \lambda_{\{m,j\}}^2}{2 \lambda_{\{i,j\}} \, \lambda_{\{m,j\}}} \,.
\]
Hence there is a bijective correspondence, modulo translations, between realizations of the octahedron in $3$-space compatible with~$\lambda$ 
and realizations of the edge graph~$\GEdg$ on the unit sphere compatible with the labeling induced by~$\lambda$.

The choice of the normalized vector corresponding to an edge in~$\GOct$ is not unique 
and depends on an orientation of the edges of the octahedron (any orientation is, in principle, fine). 
Recall that we have already fixed an orientation in Definition~\ref{definition:directions}; 
from now on, we will always refer to this choice of orientation.
Hence, given a realization $\rho \colon \{1, \dotsc, 6\} \longrightarrow \R^3$ of~$\GOct$, 
for each edge $\{i,j\} \in E_{\oct}$ we define the point $q_{\{i,j\}}$ in the unit sphere to be the one such that 
\[
 \rho(i) - \rho(j) = \ell_{ij} \, q_{\{i,j\}} \,,
\]
where we recall from Definition~\ref{definition:directions} that $\ell_{ij} > 0$ 
if $(i,j)$ is an oriented edge in~$\GOctDir$, and $\ell_{ij} = -\ell_{ji}$.
Hence, if $\rho$ is a realization of~$\GOct$ compatible with a labeling~$\lambda$
then the map that associates
\[
 \{ i, j \} \, \mapsto \, q_{\{i,j\}} 
 \quad \text{for all } \{i, j\} \in E_{\oct} = V_{\mathrm{edg}}
\]
is the induced realization of~$\GEdg$ on the unit sphere.

Notice that, once we have a triangle in the octahedron, the labeling induced on the unit vectors of the edges 
forces the three points on the unit sphere to lie on the same great circle. 
Therefore, realizations of~$\GEdg$ induced by realizations of~$\GOct$ look like the one in Figure~\ref{figure:realization_gedg}.
\begin{figure}[ht]
 \centering
 \tikz[baseline={(0,0)}]{\node at (0,0) {\includegraphics[width=.25\textwidth]{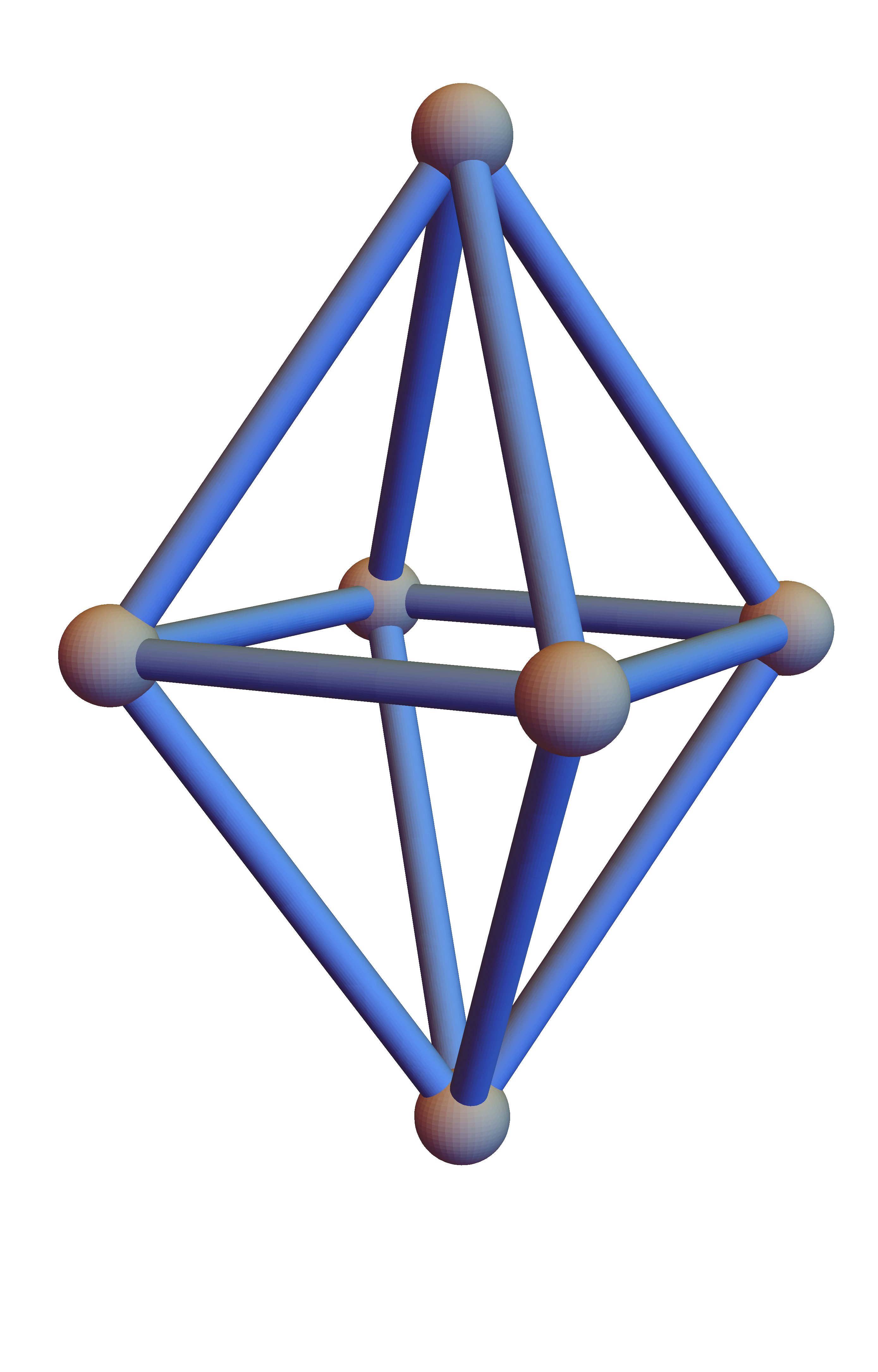}}} \qquad
 \tikz[baseline={(0,0)}]{\node at (0,0) {\includegraphics[width=.3\textwidth]{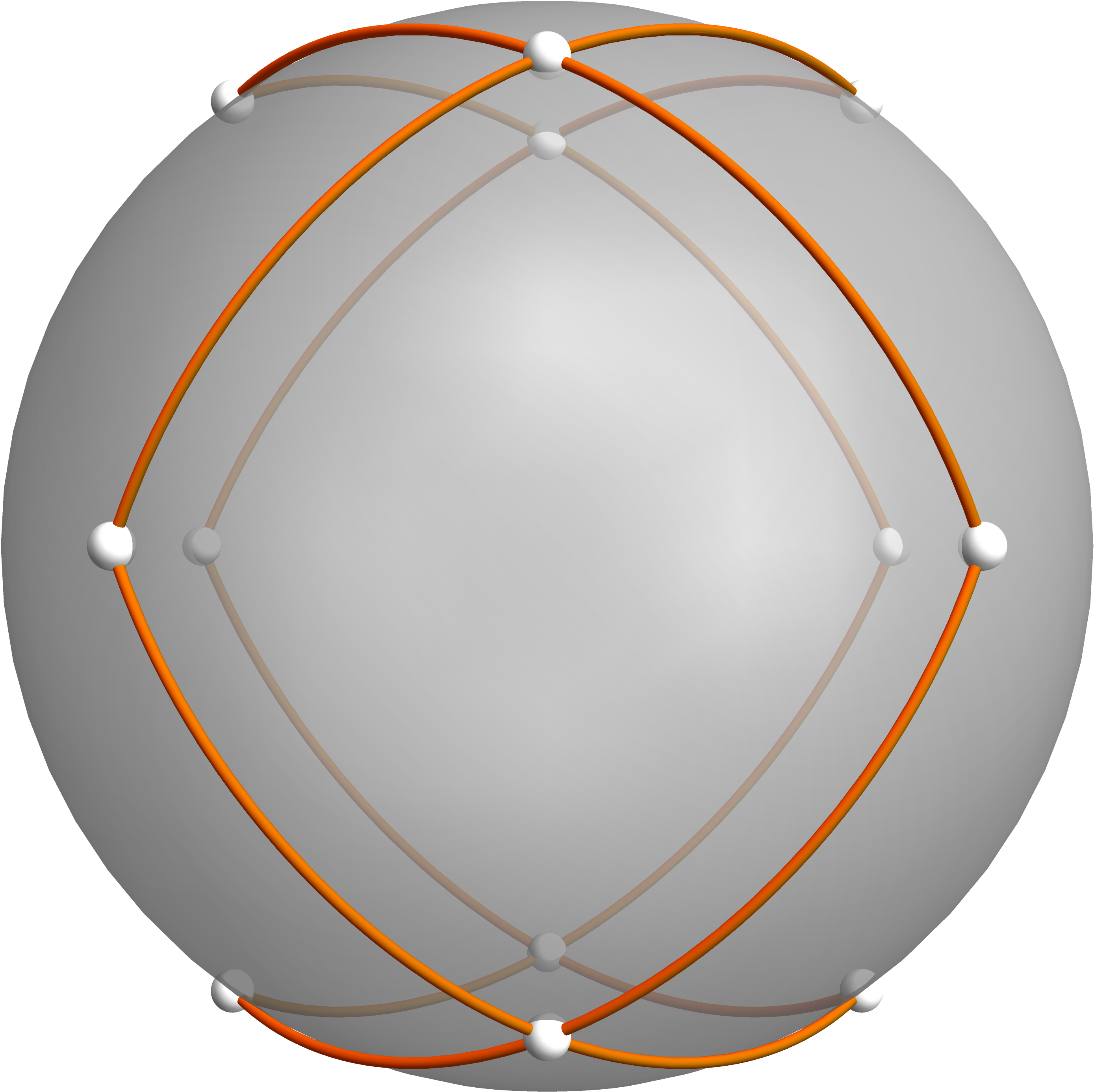}}}
 \caption{A realization of~$\GEdg$ in~$S^2$ (on the right) induced by one of~$\GOct$ in~$\R^3$ (on the left).}
 \label{figure:realization_gedg}
\end{figure}

The paper~\cite{Gallet2019} contains necessary criteria for the flexibility of any graph on a sphere, 
as well as a detailed analysis of spherical quadrilaterals; 
these arise in the current paper as the edges of a pyramid incident with the vertex.
The technique in~\cite{Gallet2019} requires to extend to the complex numbers many notions we encountered so far:
realization, flexibility, and also the unit sphere. 
Therefore, from now on realizations of~$\GOct$ will be maps $\rho \colon \{ 1, \dotsc, 6 \} \longrightarrow \C^3$,
and two realizations will be considered congruent if they differ by a complex isometry, 
which is given by the action of a complex orthogonal matrix followed by a complex translation.
Flexibility for a graph labeling will always mean admitting infinitely many compatible non-congruent realizations,
where now congruence is meant over the complex numbers, but we still consider real-valued labelings.
Compatibility of a realization~$\rho$ with a labeling~$\lambda$ now means that 
\[
 \left\langle
  \rho(i) - \rho(j),
  \rho(i) - \rho(j)
 \right\rangle
 = \lambda_{\{i,j\}}^2
 \quad \text{for all } \{i, j\} \in E_{\oct} \,,
\]
where $\left\langle \cdot, \cdot \right\rangle$ is considered just as a quadratic form, and not a scalar product.
The complexification of the unit sphere will be denoted by
\[
 S^2_{\C} = 
 \bigl\{
  (x,y,z) \in \C^3 
  \, \colon \,
  x^2 + y^2 + z^2 = 1
 \bigr\} \,,
\]
so realizations of~$\GEdg$ will be maps $V_{\mathrm{edg}} \longrightarrow S^2_{\C}$. 
Two such realizations will be congruent if they differ by a complex orthogonal matrix.
As in the spatial case, labelings are real-valued functions $E_{\mathrm{edg}} \longrightarrow \R$.
Compatibility of a realization in~$S^2_{\C}$ with a labeling is again tested via
the standard quadratic form $\left\langle \cdot, \cdot \right\rangle$, 
which in the real setting gives the cosine of the angle between two unit vectors.
As we see from their definition, the construction of the points~$q_{\{i,j\}}$ 
starting from a realization of~$\GOct$ carries over the complex numbers.
Recall however that the numbers~$\ell_{ij}$ are always real.
Flexibility of graphs on the complex sphere is defined analogously to flexibility in~$\C^3$.

From the discussion and the construction in this section, we then obtain that 
a labeling of~$\GOct$ is flexible in~$\C^3$ if and only if 
the corresponding induced labeling for~$\GEdg$ is flexible in~$S_{\C}^2$.

\section{Justification of bonds and their rules}
\label{justifications}

This section provides geometric counterparts of the notions of
octahedral and pyramidal bonds introduced in Section~\ref{classification:objects},
and gives justifications for the rules in Sections~\ref{classification:rules} and~\ref{classification:classification}. 
The needed theory is the one developed by the authors in~\cite{Gallet2019} about flexibility of graphs on the sphere, 
together with a new finding related to Rule~\ref{rule:lengths}, 
namely to the connection between bonds and linear conditions on the edge lengths of flexible octahedra. 
We recall here the main concepts of~\cite{Gallet2019} and refer to that work for proofs and precise constructions.

The geometric concept of \emph{bond} arises as follows. 
First of all, we define what we mean by \emph{configuration space} of a flexible labeling of a graph.
This notion makes it possible to consider ``realizations up to isometries'' as an algebraic variety,
and so it makes it possible to use the tools of algebraic geometry to study it.
It turns out that these varieties are not compact, and there are several possible ways to compactify them.
By doing this, we add ``points at infinity'' to the configuration space, namely points that do not correspond to realizations.
These points are the bonds. 
Although they do not correspond to realizations, they still carry deep geometric information: 
by extracting it, we will be able to explain the rules we stated in Section~\ref{classification}.

We now describe the notion of configuration space and its compactification for realizations of graphs on the sphere, as it is introduced in~\cite{Gallet2019}.
This is accomplished by noticing that it is possible to associate 
to each general $n$-tuple of points in~$S_{\C}^2$ a $2n$-tuple of points in~$\p_{\C}^1$ 
in such a way that two $n$-tuples on the sphere differ by a complex rotation (namely, by an element in~$\mathrm{SO}_3(\C)$) if and only if 
the corresponding two $2n$-tuples in~$\p_{\C}^1$ are $\p\mathrm{GL}(2, \C)$-equivalent. 
The association works as follows: 
consider $S_{\C}^2$ as the affine part of a smooth quadric in~$\p^3_{\C}$, which is covered by two families of lines;
given a point $O \in S_{\C}^2$, we can consider the two projective lines in~$S_{\C}^2$ passing through~$O$;
each of these two lines intersects the plane at infinity in a single point; 
the two points that we obtain are called the \emph{left} and \emph{right lift} of~$O$, respectively.
The left and right lift belong to the intersection of the projective closure of~$S_{\C}^2$ with the plane at infinity, 
which is a smooth plane conic, hence isomorphic to~$\p^1_{\C}$. 
This means that we can consider general realizations on the complex unit sphere, up to complex rotations, 
as points in the moduli space~$\mscr{M}_{0,2n}$ of $2n$ distinct points on the projective line. 
Moreover, one notices that constraints in terms of spherical distances on~$S_{\C}^2$ can be translated 
into relations among the lifts in~$\p_{\C}^1$ in terms of their cross-ratios. 
Therefore, one can encode realizations of graphs on the sphere compatible with a given labeling by algebraic subvarieties of~$\mscr{M}_{0,2n}$. 
This moduli space is non-compact, and a possible (projective) compactification is provided by the so-called 
\emph{moduli space of rational stable curves with marked points}, introduced by Knudsen and Mumford, and denoted~$\M_{0,2n}$. 
In this way, it is possible to assign to each graph~$G = (V, E)$, together with a labeling $\lambda \colon E \longrightarrow \R$,
a projective variety~$C_G$ inside $\M_{0,2|V|}$ whose intersection with $\mscr{M}_{0,2|V|}$ encodes the realizations of~$G$ in~$S_{\C}^2$ compatible with~$\lambda$, up to rotations.

Given this premise, we can define the notion of bond of motion of a graph.

\begin{definition}
\label{definition:bond}
 Given a graph~$G = (V, E)$ and a labeling~$\lambda \colon E \longrightarrow \R$,
 the projective variety~$C_G \subseteq \M_{0, 2|V|}$ is called the \emph{configuration space} of realizations of~$G$ in~$S_{\C}^2$ compatible with~$\lambda$. 
 Since the labeling~$\lambda$ takes real values, the variety~$C_{G}$ is real as well.
 The components of~$C_{G}$ that intersect $\mscr{M}_{0, 2|V|}$ nontrivially are called \emph{motions} of~$G$.
 The points in $C_G \cap (\M_{0, 2|V|} \setminus \mscr{M}_{0, 2|V|})$ are called the \emph{bonds} of~$G$,
 and if $K \subseteq C_G$ is a motion, bonds of~$G$ that lie in~$K$ are called bonds of~$K$.
 Since $C_G$ is a real variety and there are no real points on $\M_{0, 2|V|} \setminus \mscr{M}_{0, 2|V|}$, 
 bonds come in complex conjugate pairs.
\end{definition}

The following is one of the main results of~\cite{Gallet2019}.

\begin{proposition}
 A graph~$G$ with a flexible labeling~$\lambda$ on~$S^2_{\C}$ admits at least a bond. 
\end{proposition}

It is interesting to notice that also Connelly, in the introduction of~\cite{Connelly1978}, 
highlights the fact that extending the field to the complex numbers and ``going to infinity'' 
(as we do here with bonds) may help understanding the geometric properties of flexible objects.

\subsection*{Justification of the objects}

Now we are ready to explain why we introduced octahedral and pyramidal bonds in Section~\ref{classification}.

From Section~\ref{reduction} we know that an octahedron has a flexible labeling if and only if 
the induced labeling for the graph~$\GEdg$ is flexible on the sphere. 
This means that when we have a flexible octahedron, we get bonds for~$\GEdg$. 
The presence of bonds imposes combinatorial restrictions to graphs in terms of \emph{colorings}, which arise as follows.
Let us again consider an arbitrary graph $G = (V, E)$ on~$n$ vertices. 
The boundary $\M_{0, 2n} \setminus \mscr{M}_{0, 2n}$ is constituted of divisors (i.e., subvarieties of codimension~$1$) 
that are denoted by~$D_{I,J}$, where $(I,J)$ is a partition of the set $\{P_1, \dotsc, P_n, Q_1, \dotsc, Q_n\}$ 
of marked points --- we denote the marked points in this way to recall that we interpret them as left and right lifts of points on~$S_{\C}^2$. 
If the configuration curve~$C_G$ meets a divisor~$D_{I,J}$, then the partition~$(I,J)$ induces a coloring on the graph~$G$ as follows: 
an edge~$\{i,j\}$ of~$G$ is \emph{red} if at least three of $\{P_i, P_j, Q_i, Q_j\}$ belong to~$I$; it is \emph{blue} otherwise. 
Properties of the moduli space~$\M_{0,2n}$ imply that in each of these colorings there is no path of length~$3$ in which the colors are alternated. 
For this reason, these coloring are called \emph{NAP} (for Not Alternating Path) if they are surjective.
The main result about NAP-colorings in~\cite{Gallet2019} is that their presence characterizes flexibility on the sphere: 
a graph~$G$ admits a flexible labeling~$\lambda$ if and only if $G$ admits a NAP-coloring.

Let us now describe the NAP-colorings of the graph~$\GEdg$. 
We will see that these colorings are in bijection with quadrilaterals in~$\GOct$. 
To make the notation easier, from now on and for the rest of the paper 
we denote the marked points on the stable curves of~$\M_{0,24}$
not by $P_u, Q_u$ for $u \in \{1,\dots, 12\}$, but rather by $P_{ij}$, 
$P_{ji}$ for $\{i,j\} \in E_\oct$, with $i < j$,
since the vertices of~$\GEdg$ are labeled by pairs of indices.

\begin{definition}
\label{definition:induced_NAP}
 Each of the three quadrilaterals in~$\GOct$ determines a NAP-coloring 
of~$\GEdg$ as follows. Let $i,j,k,\ell$ be the vertices of the quadrilateral. 
There are exactly four vertices of~$\GEdg$ given by unordered pairs of 
elements in $\{i,j,k,\ell\}$. A direct inspection shows that those four 
vertices form a \emph{disconnecting set} for~$\GEdg$, namely if they are removed the 
resulting graph has two connected components. 
One then gets a NAP-coloring by coloring all the edges with their endpoints in the 
same component by the same color; see Figure~\ref{figure:NAP_colorings}. 
\end{definition}

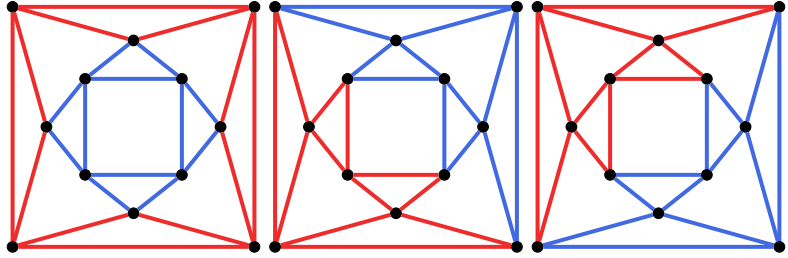
\begin{figure}[ht]
	\centering
   \begin{tikzpicture}[scale=0.9]
		\node[vertex] (1) at (0.707107, 0.707107) {};
		\node[vertex] (2) at (-0.707107, -0.707107) {};
		\node[vertex] (3) at (0.707107, -0.707107) {};
		\node[vertex] (4) at (-0.707107, 0.707107) {};
		\node[vertex] (5) at (1.76777, 1.76777) {};
		\node[vertex] (6) at (-1.76777, -1.76777) {};
		\node[vertex] (7) at (1.76777, -1.76777) {};
		\node[vertex] (8) at (-1.76777, 1.76777) {};
		\node[vertex] (9) at (1.27279, 0.) {};
		\node[vertex] (10) at (0., 1.27279) {};
		\node[vertex] (11) at (0., -1.27279) {};
		\node[vertex] (12) at (-1.27279, 0.) {};
		\draw[bedge] (1)edge(3);
		\draw[bedge] (1)edge(4);
		\draw[bedge] (1)edge(9);
		\draw[bedge] (1)edge(10);
		\draw[bedge] (2)edge(3);
		\draw[bedge] (2)edge(4);
		\draw[bedge] (2)edge(11);
		\draw[bedge] (2)edge(12);
		\draw[bedge] (3)edge(9);
		\draw[bedge] (3)edge(11);
		\draw[bedge] (4)edge(10);
		\draw[bedge] (4)edge(12);
		\draw[redge] (5)edge(7);
		\draw[redge] (5)edge(8);
		\draw[redge] (5)edge(9);
		\draw[redge] (5)edge(10);
		\draw[redge] (6)edge(7);
		\draw[redge] (6)edge(8);
		\draw[redge] (6)edge(11);
		\draw[redge] (6)edge(12);
		\draw[redge] (7)edge(9);
		\draw[redge] (7)edge(11);
		\draw[redge] (8)edge(10);
		\draw[redge] (8)edge(12);
	\end{tikzpicture}
	\begin{tikzpicture}[scale=0.9]
		\node[vertex] (1) at (0.707107, 0.707107) {};
		\node[vertex] (2) at (-0.707107, -0.707107) {};
		\node[vertex] (3) at (0.707107, -0.707107) {};
		\node[vertex] (4) at (-0.707107, 0.707107) {};
		\node[vertex] (5) at (1.76777, 1.76777) {};
		\node[vertex] (6) at (-1.76777, -1.76777) {};
		\node[vertex] (7) at (1.76777, -1.76777) {};
		\node[vertex] (8) at (-1.76777, 1.76777) {};
		\node[vertex] (9) at (1.27279, 0.) {};
		\node[vertex] (10) at (0., 1.27279) {};
		\node[vertex] (11) at (0., -1.27279) {};
		\node[vertex] (12) at (-1.27279, 0.) {};
		\draw[bedge] (1)edge(3);
		\draw[bedge] (1)edge(4);
		\draw[bedge] (1)edge(9);
		\draw[bedge] (1)edge(10);
		\draw[redge] (2)edge(3);
		\draw[redge] (2)edge(4);
		\draw[redge] (2)edge(11);
		\draw[redge] (2)edge(12);
		\draw[bedge] (3)edge(9);
		\draw[redge] (3)edge(11);
		\draw[bedge] (4)edge(10);
		\draw[redge] (4)edge(12);
		\draw[bedge] (5)edge(7);
		\draw[bedge] (5)edge(8);
		\draw[bedge] (5)edge(9);
		\draw[bedge] (5)edge(10);
		\draw[redge] (6)edge(7);
		\draw[redge] (6)edge(8);
		\draw[redge] (6)edge(11);
		\draw[redge] (6)edge(12);
		\draw[bedge] (7)edge(9);
		\draw[redge] (7)edge(11);
		\draw[bedge] (8)edge(10);
		\draw[redge] (8)edge(12);
	\end{tikzpicture}
	\begin{tikzpicture}[scale=0.9]
		\node[vertex] (1) at (0.707107, 0.707107) {};
		\node[vertex] (2) at (-0.707107, -0.707107) {};
		\node[vertex] (3) at (0.707107, -0.707107) {};
		\node[vertex] (4) at (-0.707107, 0.707107) {};
		\node[vertex] (5) at (1.76777, 1.76777) {};
		\node[vertex] (6) at (-1.76777, -1.76777) {};
		\node[vertex] (7) at (1.76777, -1.76777) {};
		\node[vertex] (8) at (-1.76777, 1.76777) {};
		\node[vertex] (9) at (1.27279, 0.) {};
		\node[vertex] (10) at (0., 1.27279) {};
		\node[vertex] (11) at (0., -1.27279) {};
		\node[vertex] (12) at (-1.27279, 0.) {};
		\draw[bedge] (1)edge(3);
		\draw[redge] (1)edge(4);
		\draw[bedge] (1)edge(9);
		\draw[redge] (1)edge(10);
		\draw[bedge] (2)edge(3);
		\draw[redge] (2)edge(4);
		\draw[bedge] (2)edge(11);
		\draw[redge] (2)edge(12);
		\draw[bedge] (3)edge(9);
		\draw[bedge] (3)edge(11);
		\draw[redge] (4)edge(10);
		\draw[redge] (4)edge(12);
		\draw[bedge] (5)edge(7);
		\draw[redge] (5)edge(8);
		\draw[bedge] (5)edge(9);
		\draw[redge] (5)edge(10);
		\draw[bedge] (6)edge(7);
		\draw[redge] (6)edge(8);
		\draw[bedge] (6)edge(11);
		\draw[redge] (6)edge(12);
		\draw[bedge] (7)edge(9);
		\draw[bedge] (7)edge(11);
		\draw[redge] (8)edge(10);
		\draw[redge] (8)edge(12);
	\end{tikzpicture}
 \caption{The three NAP-colorings of~$\GEdg$ induced by the three quadrilaterals in~$\GOct$. }
 \label{figure:NAP_colorings}
\end{figure}

By sorting out all the cases, helped by the fact that there are several triangles in~$\GEdg$, which must be monochromatic in a NAP-coloring, 
one proves the following result.

\begin{proposition}
 The only NAP-colorings of~$\GEdg$ are those induced by the three quadrilaterals in~$\GOct$ as in Definition~\ref{definition:induced_NAP}.
\end{proposition}

\begin{remark}
\label{remark:vital_NAP}
 There are $16$  divisors $D_{I,J}$ inducing the same NAP-coloring. 
 For example, if we consider the quadrilateral $\{1,2,3,4\}$, then one of these divisors is given by
 \[
  I = \{ P_{5 \ast}, P_{\ast 5}, P_{13}, P_{23}, P_{14}, P_{24} \}
  \quad
  J = \{ P_{6 \ast}, P_{\ast 6}, P_{31}, P_{32}, P_{41}, P_{42} \},
 \]
 where $\ast$ takes all the values in $\{1,2,3,4\}$.
 The other divisors are obtained by swapping the pairs $(P_{13}, P_{31})$, $(P_{23},P_{32})$, $(P_{14}, P_{41})$, and $(P_{24}, P_{42})$.
 Hence we get a total of $48 = 16 \times 3$ divisors that can be intersected by the configuration space in~$\M_{0,24}$ of a flexible octahedron.
\end{remark}

By examining the shape of the partitions $(I,J)$, we get the following graphical description of the divisors~$D_{I,J}$.

\begin{proposition}
\label{proposition:divisor_orientation}
Let $\{m,n\}$ be any of $\{1,2\}$, $\{3,4\}$, $\{5,6\}$ and let $\mcal{Q}$ be the quadrilateral in~$\GOct$ with vertices $\{1, \dotsc, 6\} \setminus \{m, n\}$. 
There is a bijection between the divisors~$D_{I,J}$ inducing the NAP-coloring determined by $\{1, \dotsc, 6\} \setminus \{m,n\}$ 
and the set of orientations of the edges of~$\mcal{Q}$. 
The bijection works as follows.
Write $I = \{ P_{m \ast}, P_{\ast m}, P_{t_1 \, s_1}, \dotsc, P_{t_4 \, s_4}\}$, then $\{t_1, s_1\}, \dotsc, \{t_4, s_4\}$ 
are the edges of the (undirected) $4$-cycle~$\mcal{Q}$.
We then declare that $D_{I,J}$ determines the orientations $(t_1, s_1), \dotsc, (t_4, s_4)$ of the edges of~$\mcal{Q}$; 
see Figure~\ref{figure:bond_orientation} for the quadrilateral corresponding to the example in Remark~\ref{remark:vital_NAP}.
\end{proposition}

Hence Proposition~\ref{proposition:divisor_orientation} explains why we defined octahedral bonds in Section~\ref{classification} in that way: 
oriented quadrilaterals of~$\GOct$ codify the divisors~$D_{I,J}$ that may be intersected by the configuration curve~$C_{\GEdg}$ of flexible labeling of~$\GEdg$. 
The $\mu$-number of an octahedral bond reports the sum of the intersection multiplicities between~$C_{\GEdg}$ and a divisor~$D_{I,J}$,
and is defined as the degree of the divisor on~$C_{\GEdg}$ cut out by~$D_{I,J}$.

To explain the origin of pyramidal bonds, notice that 
if we apply the reduction from the space to sphere described in Section~\ref{reduction} to a pyramid \quadrefV,
taking into account only the edges incident to~$v$,
we obtain a $4$-cycle on the sphere. 
If we take as graph~$G$ a $4$-cycle, whose vertices are $\{1,2,3,4\}$ and whose edges are $\bigl\{ \{1,2\}, \{2,3\}, \{3,4\}, \{1,4\} \bigr\}$,
then the moduli space where its configuration space lives is~$\M_{0,8}$. 
Let us, for a moment, switch back to the notation $P_u$, $Q_u$ for the marked points of stable curves,
just to make the notation in this particular case less heavy.
There are four divisors~$D_{I,J}$ in~$\M_{0,8}$ given by the partitions
\begin{align*}
 I &= \{ P_1, Q_1, P_2, P_4 \}, & J &= \{ P_3, Q_3, Q_2, Q_4 \}, \\
 I &= \{ P_2, Q_2, P_1, P_3 \}, & J &= \{ P_4, Q_4, Q_1, Q_3 \}, \\
 I &= \{ P_1, Q_1, P_2, Q_4 \}, & J &= \{ P_3, Q_3, Q_2, P_4 \}, \\
 I &= \{ P_2, Q_2, P_1, Q_3 \}, & J &= \{ P_4, Q_4, Q_1, P_3 \}.
\end{align*}
By swapping the $P$'s with the $Q$'s in the previous partitions, we obtain four other divisors, which are the complex conjugates of the previous ones. 
Let us focus on the $I$-part of the partition: we see that we always have a pair $P_k$, $Q_k$. 
If $k$ is even we say that the divisor is \emph{even} (e), while we say that it is \emph{odd} (o) if $k$ is \emph{odd}.
Moreover, we see that in the $I$-part we have another pair of marked points of the form either~$(P_i, P_j)$ or~$(P_i, Q_j)$.
In the first case we say that the divisor is \emph{unmixed} (u), while in the second case we say that the divisor is \emph{mixed} (m). 
Hence, to specify one of these four divisors it is enough to specify whether it is even or odd, and unmixed or mixed.
Therefore, we denote these divisors by~$\divom{}$, $\divou{}$, $\divem{}$, and~$\diveu{}$.

Now we can go back to our usual notation for marked points, and discuss the situation for all pyramids in the octahedron.
Recall that in Figure~\ref{figure:even_odd_pyramids} we fixed the convention about even and odd edges of the six pyramids of the octahedron, 
which correspond to the six quadrilaterals in~$\GEdg$. 
This convention is summarized in Figure~\ref{figure:even_odd}; 
one can notice that it is equivariant with respect to cyclic permutations of the vertices~$(1,4,5,2,3,6)$ of the octahedron.
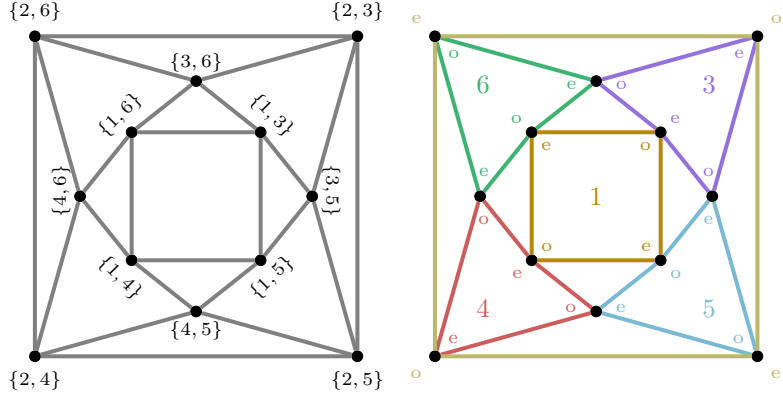
\begin{figure}[ht]
	\centering
	\begin{tikzpicture}[scale=1.2,baseline={(0,0)}]
		\node[vertex,label={[labelsty,rotate=-45,label distance=-1pt]above:$\{1,3\}$}] (1) at (0.707107, 0.707107) {};
		\node[vertex,label={[labelsty,rotate=-45,label distance=-1pt]below:$\{1,4\} $}] (2) at (-0.707107, -0.707107) {};
		\node[vertex,label={[labelsty,rotate=45,label distance=-1pt]below:$\{1,5\}$}] (3) at (0.707107, -0.707107) {};
		\node[vertex,label={[labelsty,rotate=45,label distance=-1pt]above:$\{1,6\}$}] (4) at (-0.707107, 0.707107) {};
		\node[vertex,label={[labelsty]above:$\{2,3\}$}] (5) at (1.76777, 1.76777) {};
		\node[vertex,label={[labelsty]below:$\{2,4\}$}] (6) at (-1.76777, -1.76777) {};
		\node[vertex,label={[labelsty]below:$\{2,5\}$}] (7) at (1.76777, -1.76777) {};
		\node[vertex,label={[labelsty]above:$\{2,6\}$}] (8) at (-1.76777, 1.76777) {};
		\node[vertex,label={[labelsty,rotate=-90]above:$\{3,5\}$}] (9) at (1.27279, 0.) {};
		\node[vertex,label={[labelsty,label distance=-2pt]above:$\{3,6\}$}] (10) at (0., 1.27279) {};
		\node[vertex,label={[labelsty,label distance=-2pt]below:$\{4,5\}$}] (11) at (0., -1.27279) {};
		\node[vertex,label={[labelsty,rotate=90]above:$\{4,6\}$}] (12) at (-1.27279, 0.) {};
		\draw[edge] (1)edge(3) (1)edge(4) (1)edge(9) (1)edge(10) (2)edge(3) (2)edge(4) (2)edge(11) (2)edge(12) (3)edge(9) (3)edge(11) (4)edge(10) (4)edge(12) (5)edge(7) (5)edge(8) (5)edge(9) (5)edge(10) (6)edge(7) (6)edge(8) (6)edge(11) (6)edge(12) (7)edge(9) (7)edge(11) (8)edge(10) (8)edge(12);
	\end{tikzpicture}
	\begin{tikzpicture}[scale=1.2,baseline={(0,0)}]
		\node[vertex,label={[labelsty,label distance=-2pt,col3]45:e},label={[labelsty,label distance=-2pt,col1]-135:o}] (1) at (0.707107, 0.707107) {};
		\node[vertex,label={[labelsty,label distance=-2pt,col4]-135:e},label={[labelsty,label distance=-2pt,col1]45:o}] (2) at (-0.707107, -0.707107) {};
		\node[vertex,label={[labelsty,label distance=-2pt,col5]-45:o},label={[labelsty,label distance=-2pt,col1]135:e}] (3) at (0.707107, -0.707107) {};
		\node[vertex,label={[labelsty,label distance=-2pt,col6]135:o},label={[labelsty,label distance=-2pt,col1]-45:e}] (4) at (-0.707107, 0.707107) {};
		\node[vertex,label={[labelsty,col3]-135:e},label={[labelsty,col2]45:o}] (5) at (1.76777, 1.76777) {};
		\node[vertex,label={[labelsty,col4]45:e},label={[labelsty,col2]-135:o}] (6) at (-1.76777, -1.76777) {};
		\node[vertex,label={[labelsty,col5]135:o},label={[labelsty,col2]-45:e}] (7) at (1.76777, -1.76777) {};
		\node[vertex,label={[labelsty,col6]-45:o},label={[labelsty,col2]135:e}] (8) at (-1.76777, 1.76777) {};
		\node[vertex,label={[labelsty,label distance=2pt,above,col3]110:o},label={[labelsty,label distance=2pt,below,col5]-110:e}] (9) at (1.27279, 0.) {};
		\node[vertex,label={[labelsty,label distance=2pt,right,col3]-20:o},label={[labelsty,label distance=2pt,left,col6]200:e}] (10) at (0., 1.27279) {};
		\node[vertex,label={[labelsty,label distance=2pt,left,col4]-200:o},label={[labelsty,label distance=2pt,right,col5]20:e}] (11) at (0., -1.27279) {};
		\node[vertex,label={[labelsty,label distance=2pt,below,col4]-70:o},label={[labelsty,label distance=2pt,above,col6]70:e}] (12) at (-1.27279, 0.) {};
		\draw[edge,col1] (1)edge(3) (1)edge(4)  (2)edge(3) (2)edge(4); 
		\draw[edge,col4]  (2)edge(11) (2)edge(12) (6)edge(11) (6)edge(12); 
		\draw[edge,col5] (3)edge(9) (3)edge(11)  (7)edge(9) (7)edge(11); 
		\draw[edge,col3] (1)edge(9) (1)edge(10) (5)edge(9) (5)edge(10); 
		\draw[edge,col6]  (4)edge(10) (4)edge(12) (8)edge(10) (8)edge(12); 
		\draw[edge,col2] (5)edge(7) (5)edge(8)  (6)edge(7) (6)edge(8); 
		\node[col1] at (0,0) {1};
		\node[col3] at ($1/2*(1)+1/2*(5)$) {3};
		\node[col4] at ($1/2*(2)+1/2*(6)$) {4};
		\node[col5] at ($1/2*(3)+1/2*(7)$) {5};
		\node[col6] at ($1/2*(4)+1/2*(8)$) {6};
	\end{tikzpicture}
 \caption{Definitions of even and odd vertices for each of the six quadrilaterals in~$\GEdg$.}
 \label{figure:even_odd}
\end{figure}
With these choices, we see for example that if we consider the pyramid~\quadrefA, 
then the odd mixed divisor~$\divom{1}$ is given by the following partition:
\[
 I = (P_{14}, P_{41}, P_{15}, P_{61}), 
 \qquad
 J = (P_{13}, P_{31}, P_{51}, P_{16}).
\]
The notion of pyramidal bond then codifies the information contained in the $I$-part of the partition 
determined by one of the four divisors associated with a pyramid in the following way. 
As we saw, the $I$-part of such a partition associated to a pyramid~\quadrefV\ is of the form:
\[
 I = 
 \Bigl(
 P_{va}, P_{av},
 \begin{array}{c}
  P_{vu} \\
  \text{\small or} \\
  P_{uv}
 \end{array} ,
 \begin{array}{c}
  P_{vw} \\
  \text{\small or} \\
  P_{wv}
 \end{array}
 \Bigr)
\]
If $b$ is the vertex such that $\{a, b\}$ is a non-edge of~$\GOct$, 
then the pyramid~\quadrefV\ is the one induced by the vertices $v, a, b, u, w$.
The two oriented edges of this subgraph, forming the pyramidal bond, are then $(v,u)$ (or $(u,v)$) and $(v,w)$ (or $(w,v)$).
For example, the pyramidal bond associated to the divisor~$\divom{1}$ is~$\Pyr_{\Pbonddl}$.

\subsection*{Justification of Rule~\ref{rule:table}}

This rule summarizes the content of~\cite[Section~4.1]{Gallet2019}.
In fact, flexible pyramids determine flexible quadrilaterals on the sphere, 
and the reference describes their behavior, 
concerning in particular the intersection of their motions with the divisors in~$\M_{0,8}$.

\subsection*{Justification of Rule~\ref{rule:lengths}}

We want to obtain necessary conditions for the edge lengths of a flexible octahedron.
Notice that Mikhal\"{e}v in \cite{Mikhalev2001} obtains the same conditions for any suspensions whose equator is a cycle\footnote{A \emph{suspension} is a polyhedron whose combinatorial structure is the one of a double pyramid.}. 
The first author who discussed these relations for octahedra was, to our knowledge, Lebesgue.

Let us suppose, for simplicity, that a motion of a flexible octahedron meets the divisor~$D_{I,J}$ in~$\M_{0,24}$ described in Remark~\ref{remark:vital_NAP}. 
Let $q_{\{i,j\}}$ be the point in the sphere~$S^2_{\C}$ determined by the edge~$\{i,j\}$ in~$\GOct$ as described in Section~\ref{reduction}. 
Let us first clarify the relation between $q_{\{i,j\}}$ and the two marked points~$P_{ij}$ and~$P_{ji}$ 
corresponding to it in the stable curves with marked points of~$C_{\GEdg}$. 
When the marked points belong to a stable curve that is not in the boundary of~$\M_{0,24}$, 
we can recover the coordinates of~$q_{\{i,j\}}$ from the ones of~$P_{ij}$ and of~$P_{ji}$. 
Let us suppose that $P_{ij} = (u_{ij}: v_{ij})$ and $P_{ji} = (u_{ji}: v_{ji})$ (here we think about them as points in~$\p^1_{\C}$). 
The point~$q_{\{i,j\}}$ is essentially the image of $(P_{ij}, P_{ji})$ under the Segre embedding $\p^1 \times \p^1 \longrightarrow \p^3$. 
We need to be a little cautious here, since we should not use the ``standard'' map
\[
 (u_{ij}: v_{ij}), (u_{ji}: v_{ji}) 
 \mapsto 
 (u_{ij} \, u_{ji} : u_{ij} \, v_{ji} : v_{ij} \, u_{ji} : v_{ij} \, v_{ji} )
\]
but rather
\[
 (u_{ij}: v_{ij}), (u_{ji}: v_{ji}) 
 \mapsto 
 (u_{ij} \, u_{ji} : v_{ij} \, v_{ji} : u_{ij} \, v_{ji} + v_{ij} \, u_{ji} : u_{ij} \, v_{ji} - v_{ij} \, u_{ji}).
\]
In fact, our choice of coordinates should be such that the points where $P_{ij} = P_{ji}$ correspond to the plane at infinity 
(this justifies the choice of the last coordinates), 
and moreover the origin should be the polar of the plane at infinity with respect to the polarity induced by the quadric that is the image of~$\p^1 \times \p^1$. 
Hence we get the following expression for the vector~$q_{\{i,j\}}$:
\begin{equation}
\label{equation:parametrization_q}
 q_{\{i,j\}} = 
 \left(
 \frac{u_{ij} \, u_{ji}}{u_{ij} \, v_{ji} - v_{ij} \, u_{ji}},
 \frac{v_{ij} \, v_{ji}}{u_{ij} \, v_{ji} - v_{ij} \, u_{ji}},
 \frac{u_{ij} \, v_{ji} + v_{ij} \, u_{ji}}{u_{ij} \, v_{ji} - v_{ij} \, u_{ji}}
 \right) \,.
\end{equation}
Since the quadrilateral $\{1,2,3,4\}$ in~$\GOct$ forms a closed loop, we get the following condition, 
where $\ell_{ij}$ is the (signed) length of the edge~$\{i,j\}$:
\begin{equation}
\label{equation:condition_quadrilateral}
 \ell_{13} \, q_{\{1,3\}} + \ell_{23} \, q_{\{2,3\}} + \ell_{24} \, q_{\{2,4\}} + \ell_{14} \, q_{\{1,4\}} = 0 \,.
\end{equation}
Our goal is to express the condition of~Equation~\eqref{equation:condition_quadrilateral} in local coordinates of the moduli space~$M_{0,8}$ 
obtained by forgetting all marked points of the form $P_{5,\ast}$, $P_{\ast, 5}$, $P_{6,\ast}$, and $P_{\ast,6}$,
where $\ast$ takes all the values in $\{1,2,3,4\}$. 
Once we have done that, we can restrict the equation to the (projection of the) divisor~$D_{I,J}$ 
and obtain a necessary condition on the numbers~$\ell_{ij}$. We make the following choice of local coordinates for~$\M_{0,8}$:
\begin{align*}
 P_{13} &= (1:0), & P_{31} &= (0:1), & P_{14} &= (x_1:1), & P_{41} &= (z: x_2), \\
 P_{23} &= (1:1), & P_{32} &= (z:1), & P_{24} &= (x_3:1), & P_{42} &= (z: x_4).
\end{align*}
Notice that, with this choice of coordinates, $\{z = 0\}$ is a local equation for the projection of the divisor~$D_{I,J}$ on $M_{0,8}$. 
By using this choice of coordinates in Equation~\eqref{equation:parametrization_q} 
and by substituting the expressions for the $\{ q_{\{i,j\}} \}$ in Equation~\eqref{equation:condition_quadrilateral}, 
we get three equations given by rational functions in~$z, x_1, \dotsc, x_4$ and the lengths~$\ell_{13}, \dotsc, \ell_{14}$. 
Cleaning the denominators and saturating by them the obtained polynomial equations yields equations 
that can be restricted to the projection of the divisor~$D_{I,J}$ by imposing $z = 0$. 
Once we eliminate the variables\footnote{This and the previous operations can be performed by a computer algebra system such as Maple, Mathematica, or Sage.} $z, x_1, \dotsc, x_4$, we are left with a single equation, namely
\[
 \ell_{13} + \ell_{23} + \ell_{24} + \ell_{14} = 0.
\]
This equation is precisely the one prescribed by Rule~\ref{rule:lengths}.

\subsection*{Justification of Rule~\ref{rule:conjugates}}

This rule follows from the fact that the configuration curve~$C_{\GEdg}$ is a real variety, 
and so the degree of the divisor cut out on~$C_{\GEdg}$ by a divisor~$D_{I,J}$ 
equals the degree of the divisor cut out on~$C_{\GEdg}$ by its conjugate~$\overline{D}_{I,J}$. 
One then notices that complex conjugation interchanges the marked points $P_{ij}$ and $P_{ji}$, 
and so $\overline{D}_{I,J}$ determines the same quadrilateral of~$D_{I,J}$ but with opposite orientation.

\subsection*{Justification of Rule~\ref{rule:equations}}

Fix a motion $\mathfrak{K} \subseteq C_{\GEdg}$ and a pyramid~\quadrefV.
The pyramid~\quadrefV\ defines a quadrilateral~$\mcal{Q}$ in~$\GEdg$.
By assumption, we know that the pyramid~\quadrefV\ is simple.
Let \mbox{$\pi_{\mcal{Q}} \colon \M_{0, 24} \longrightarrow \M_{0,8}$} be the projection 
that forgets all marked points except the ones related to~$\mcal{Q}$.
The fact that the pyramid~\quadrefV\ is simple implies that the restriction $\pi_{\mcal{Q}}|_{\mathfrak{K}}$ is birational.
As recalled in Section~\ref{reduction}, there are $4$ divisors (together with their complex conjugates)
that are relevant for us, namely $\{ \divom{v}, \divou{v}, \divem{v}, \diveu{v} \}$.
For each of them, we can use the following elementary fact from algebraic geometry: 
if $f \colon X \longrightarrow Y$ is a birational morphism between projective curves, 
and $E$ is a divisor on~$Y$, then the degree of~$E$ equals the degree of the pullback of~$E$ via~$f$.
By applying this fact to each of the four divisors, we get the following equations:
\[
 \sum_{\pi_{\mcal{Q}}(D_{I,J}) = D^{v}} \deg (D_{I,J}|_{\mathfrak{K}}) = \deg (D^{v}|_{\pi_{\mcal{Q}}(\mathfrak{K})}) ,
\]
for each $D^{v} \in \{ \divom{v}, \divou{v}, \divem{v}, \diveu{v} \}$.
Thus we obtain equations linking sums of $\mu$-numbers of octahedral bonds to $\mu$-numbers of pyramidal bonds.
Because of the choice in the notation we made so far, 
which is equivariant under cyclic permutations of the indices $(1, 4, 5, 2, 3, 6)$,
in order to compute all the $6 \times 4 = 24$ equations, 
it is enough to compute the equations for the pyramid~\quadrefA; 
the other equations are obtained by cyclic permutations of the numerical indices.
Therefore, for each divisor, say $\divom{1}$, in~$\M_{0,8}$ we need to compute the divisors~$D_{I,J}$ in~$\M_{0,24}$ 
that project to~$\divom{v}$ via~$\pi_{\mcal{Q}}$. 
Since $\divom{1}$ is given by the partition
\[
 I_{\mathrm{om}} = (P_{14}, P_{41}, P_{15}, P_{61}), 
 \qquad
 J_{\mathrm{om}} = (P_{13}, P_{31}, P_{51}, P_{16}),
\]
it is enough to compute all the partitions $(I,J)$ of the set 
$\{P_{ij}, P_{ji} \colon \{i,j\}\in E_\oct \}$
of vertices of~$\GEdg$ that extend the partition $(I_{\mathrm{om}}, J_{\mathrm{om}})$.
There is exactly one such partition:
\[
 I = (P_{4*},P_{*4}, P_{15}, P_{61}, P_{25}, P_{62}),
 \qquad
 J = (P_{3*},P_{*3}, P_{51}, P_{16}, P_{52}, P_{26}).
\]
The computations for the other divisors in~$\M_{0,8}$ are reported in Table~\ref{table:equations}. 
From this table we see that the rule yielding the equations is the same as the graphical procedure in Rule~\ref{rule:equations}.

\begin{table}
\caption{Derivation of graphical procedure in Rule~\ref{rule:equations}.}
\begin{center}
\begin{adjustbox}{angle=90}
  \begin{tabular}{>{\centering\arraybackslash$}l<{$} >{\centering\arraybackslash$}l<{$} >{\centering\arraybackslash}m{3.2cm} >{\centering\arraybackslash}m{3.1cm} >{\centering\arraybackslash}m{2.4cm}}
		\toprule
		\text{Partition in } \M_{0,8} & \text{Partition(s) in } \M_{0,24} & Pyramidal bond & Extension(s) to octahedral bond(s)  & $\mu$-relations \\\midrule
    (P_{14},P_{41},{\color{col1}P_{15}},{\color{col1}P_{61}}) & (P_{4*},P_{*4},{\color{col1}P_{15}},{\color{col1}P_{61}}, {\color{colpb}P_{25}},{\color{colpb}P_{62}}) &
    \begin{tikzpicture}[scale=0.6]
			\coordinate (o) at (0,0);
			\node[svertex,label={[labelsty]right:$1$}] (1) at (1,0) {};
			\node[svertex,label={[labelsty]right:$5$},rotate around={120:(o)}] (5) at (1) {};
			\node[svertex,label={[labelsty]left:$2$}] (2) at (-1,0) {};
			\node[svertex,label={[labelsty,label distance=-2pt]left:$6$},rotate around={120:(o)}] (6) at (2) {};
			\draw[sdedge,col1] (6) -- (1);
			\draw[sdedge,col1] (1) -- (5);
			\draw[edge] (6) -- (2);
			\draw[edge] (2) -- (5);
    \end{tikzpicture}
    &
    \begin{tikzpicture}[scale=0.6]
      \coordinate (o) at (0,0);
			\node[svertex,label={[labelsty]right:$1$}] (1) at (1,0) {};
			\node[svertex,label={[labelsty]right:$5$},rotate around={120:(o)}] (5) at (1) {};
			\node[svertex,label={[labelsty]left:$2$}] (2) at (-1,0) {};
			\node[svertex,label={[labelsty,label distance=-2pt]left:$6$},rotate around={120:(o)}] (6) at (2) {};
			\draw[sdedge,col1] (6) -- (1);
			\draw[sdedge,col1] (1) -- (5);
			\draw[sdedge,colpb] (6) -- (2);
			\draw[sdedge,colpb] (2) -- (5);
			\node at (0,0) {\scriptsize$\ObondC$};
    \end{tikzpicture}
    &
    $\mu_{\Pbonddl}^{1} = \mu_{\ObondC}^{43}$\\\hline
    (P_{14},P_{41},{\color{col1}P_{15}},{\color{col1}P_{16}}) & (P_{4*},P_{*4}, {\color{col1}P_{15}},{\color{col1}P_{16}}, {\color{colpb}P_{52}}, {\color{colpb}P_{62}}) &
    \begin{tikzpicture}[scale=0.6]
			\coordinate (o) at (0,0);
			\node[svertex,label={[labelsty]right:$1$}] (1) at (1,0) {};
			\node[svertex,label={[labelsty]right:$5$},rotate around={120:(o)}] (5) at (1) {};
			\node[svertex,label={[labelsty]left:$2$}] (2) at (-1,0) {};
			\node[svertex,label={[labelsty,label distance=-2pt]left:$6$},rotate around={120:(o)}] (6) at (2) {};
			\draw[sdedge,col1] (1) -- (6);
			\draw[sdedge,col1] (1) -- (5);
			\draw[edge] (6) -- (2);
			\draw[edge] (2) -- (5);
    \end{tikzpicture}
    &
    \begin{tikzpicture}[scale=0.6]
      \coordinate (o) at (0,0);
			\node[svertex,label={[labelsty]right:$1$}] (1) at (1,0) {};
			\node[svertex,label={[labelsty]right:$5$},rotate around={120:(o)}] (5) at (1) {};
			\node[svertex,label={[labelsty]left:$2$}] (2) at (-1,0) {};
			\node[svertex,label={[labelsty,label distance=-2pt]left:$6$},rotate around={120:(o)}] (6) at (2) {};
			\draw[sdedge,col1] (1) -- (6);
			\draw[sdedge,col1] (1) -- (5);
			\draw[sdedge,colpb] (6) -- (2);
			\draw[sdedge,colpb] (5) -- (2);
			\node at (0,0) {\scriptsize$\ObondB$};
    \end{tikzpicture}
    & $\mu_{\Pbonddo}^{1} = \mu_{\ObondB}^{43} + \mu_{\ObondAc}^{43}$
     \\[-5pt]
    & (P_{4*},P_{*4},{\color{col1}P_{15}},{\color{col1}P_{16}}, {\color{colpb}P_{25}}, {\color{colpb}P_{26}}) &
    &
    \begin{tikzpicture}[scale=0.6]
      \coordinate (o) at (0,0);
			\node[svertex,label={[labelsty]right:$1$}] (1) at (1,0) {};
			\node[svertex,label={[labelsty]right:$5$},rotate around={120:(o)}] (5) at (1) {};
			\node[svertex,label={[labelsty]left:$2$}] (2) at (-1,0) {};
			\node[svertex,label={[labelsty,label distance=-2pt]left:$6$},rotate around={120:(o)}] (6) at (2) {};
			\draw[sdedge,col1] (1) -- (6);
			\draw[sdedge,col1] (1) -- (5);
			\draw[sdedge,colpb] (2) -- (6);
			\draw[sdedge,colpb] (2) -- (5);
			\node at (0,0) {\scriptsize$\ObondAc$};
    \end{tikzpicture}
    &
    \\\hline    
    (P_{15},P_{51},{\color{col1}P_{13}},{\color{col1}P_{41}}) & (P_{5*},P_{*5}, {\color{col1}P_{13}},{\color{col1}P_{41}}, {\color{colpb}P_{23}}, {\color{colpb}P_{42}}) &
    \begin{tikzpicture}[scale=0.6]
			\coordinate (o) at (0,0);
			\node[svertex,label={[labelsty]right:$1$}] (1) at (1,0) {};
			\node[svertex,label={[labelsty,label distance=-2pt]right:$4$},rotate around={60:(o)}] (4) at (1) {};
			\node[svertex,label={[labelsty]left:$2$}] (2) at (-1,0) {};
			\node[svertex,label={[labelsty]left:$3$},rotate around={60:(o)}] (3) at (2) {};
			\draw[sdedge,col1] (1) -- (3);
			\draw[sdedge,col1] (4) -- (1);
			\draw[edge] (3) -- (2);
			\draw[edge] (2) -- (4);
    \end{tikzpicture}
    &
    \begin{tikzpicture}[scale=0.6]
      \coordinate (o) at (0,0);
			\node[svertex,label={[labelsty]right:$1$}] (1) at (1,0) {};
			\node[svertex,label={[labelsty,label distance=-2pt]right:$4$},rotate around={60:(o)}] (4) at (1) {};
			\node[svertex,label={[labelsty]left:$2$}] (2) at (-1,0) {};
			\node[svertex,label={[labelsty]left:$3$},rotate around={60:(o)}] (3) at (2) {};
			\draw[sdedge,col1] (1) -- (3);
			\draw[sdedge,col1] (4) -- (1);
			\draw[sdedge,colpb] (2) -- (3);
			\draw[sdedge,colpb] (4) -- (2);
			\node at (0,0) {\scriptsize$\ObondB$};
    \end{tikzpicture}
    &
    $\mu_{\Pbondal}^{1} = \mu_{\ObondB}^{56}$\\\hline
    (P_{15},P_{51},{\color{col1}P_{13}},{\color{col1}P_{14}}) & (P_{5*},P_{*5}, {\color{col1}P_{13}},{\color{col1}P_{14}}, {\color{colpb}P_{23}}, {\color{colpb}P_{24}}) &
    \begin{tikzpicture}[scale=0.6]
			\coordinate (o) at (0,0);
			\node[svertex,label={[labelsty]right:$1$}] (1) at (1,0) {};
			\node[svertex,label={[labelsty,label distance=-2pt]right:$4$},rotate around={60:(o)}] (4) at (1) {};
			\node[svertex,label={[labelsty]left:$2$}] (2) at (-1,0) {};
			\node[svertex,label={[labelsty]left:$3$},rotate around={60:(o)}] (3) at (2) {};
			\draw[sdedge,col1] (1) -- (3);
			\draw[sdedge,col1] (1) -- (4);
			\draw[edge] (3) -- (2);
			\draw[edge] (2) -- (4);
    \end{tikzpicture}
    &
    \begin{tikzpicture}[scale=0.6]
      \coordinate (o) at (0,0);
			\node[svertex,label={[labelsty]right:$1$}] (1) at (1,0) {};
			\node[svertex,label={[labelsty,label distance=-2pt]right:$4$},rotate around={60:(o)}] (4) at (1) {};
			\node[svertex,label={[labelsty]left:$2$}] (2) at (-1,0) {};
			\node[svertex,label={[labelsty]left:$3$},rotate around={60:(o)}] (3) at (2) {};
			\draw[sdedge,col1] (1) -- (3);
			\draw[sdedge,col1] (1) -- (4);
			\draw[sdedge,colpb] (2) -- (3);
			\draw[sdedge,colpb] (2) -- (4);
			\node at (0,0) {\scriptsize$\ObondA$};
    \end{tikzpicture}
    & $\mu_{\Pbondao}^{1} = \mu_{\ObondA}^{56} + \mu_{\ObondC}^{56}$
     \\[-5pt]
    & (P_{5*},P_{*5}, {\color{col1}P_{13}},{\color{col1}P_{14}}, {\color{colpb}P_{32}}, {\color{colpb}P_{42}}) &
    &
    \begin{tikzpicture}[scale=0.6]
      \coordinate (o) at (0,0);
			\node[svertex,label={[labelsty]right:$1$}] (1) at (1,0) {};
			\node[svertex,label={[labelsty,label distance=-2pt]right:$4$},rotate around={60:(o)}] (4) at (1) {};
			\node[svertex,label={[labelsty]left:$2$}] (2) at (-1,0) {};
			\node[svertex,label={[labelsty]left:$3$},rotate around={60:(o)}] (3) at (2) {};
			\draw[sdedge,col1] (1) -- (3);
			\draw[sdedge,col1] (1) -- (4);
			\draw[sdedge,colpb] (3) -- (2);
			\draw[sdedge,colpb] (4) -- (2);
			\node at (0,0) {\scriptsize$\ObondC$};
    \end{tikzpicture}
    &
    \\\bottomrule
  \end{tabular}
\end{adjustbox}
\end{center}
\label{table:equations}
\end{table}

\subsection*{Justification of Rule~\ref{rule:edge_multiplicity}}

This follows from \cite[Section~4.1]{Gallet2019}: 
the multiplicities of the edges in this paper correspond to the degrees of the maps~$r_{i}^{k\ell}$ in the reference.

\subsection*{Justification of Rule~\ref{rule:parallelogram}}

Assuming that vertices $3$, $4$, $5$, $6$ are coplanar and that the vertices $1$ and $2$ lie symmetric with respect to that plane,
we claim that the plane quadrilateral~$12$ is either an antiparallelogram or a parallelogram. 
To show this claim, let $C_O$ be the configuration space of the octahedron for the motion we are considering, 
let $C_1$ and $C_2$ be configuration spaces of the pyramids~$\quadrefA$\ and~$\quadrefB$, 
and let $C_{12}$ be the configuration space of the plane quadrilateral~$12$. 
There are natural projection maps $f_1 \colon  C_O \longrightarrow C_1$, $f_2 \colon C_O \longrightarrow C_2$, $g_1 \colon C_1 \longrightarrow C_{12}$, $g_2 \colon C_2 \longrightarrow C_{12}$,
and an isomorphism $h \colon C_1 \longrightarrow C_2$ defined by reflecting the vertex~$1$ at the plane of the quadrilateral~$12$. 
We claim also that $g_2$ is an isomorphism (and $g_1 = g_2 \circ h$ is also an isomorphism). 
Assume, for a contradiction, that $g_2$ is a $2:1$ map. 
Then there is an automorphism $s \colon C_2 \longrightarrow C_2$ flipping the two points of any fiber of~$g_2$. 
We can define two proper subsets $X, Y \subset C_O$ in the following way: $X$ is the set of all $x \in C_O$ such that $f_2(x)=h \bigl(f_1(x)\bigr)$, and $Y$ is the set of all $y \in C_O$ such that $f_2(y)=s \bigl( h(f_1(y)) \bigr)$. 
These are two closed subsets which cover $C_O$. 
This is a contradiction to the irreducibility of~$C_O$, so $g_1$ and $g_2$ are isomorphisms.

To conclude the proof, we want to show that all four vertices of the plane quadrilateral~$12$ are simple, 
where the definition of the multiplicity of the vertex of a plane quadrilateral parallels Definition~\ref{definition:multiplicity} for multiplicity of dihedral angles.
In fact, a plane quadrilateral in which all vertices are simple is a parallelogram or an antiparallelogram.
It suffices to show that vertex~$3$ is simple. 
Let $D$ be the configuration space of the pyramid~$\quadrefC$; 
let $D_1$ be configuration space of the subgraph with vertices $1,3,5,6$, 
which is determined by the dihedral angle at the edge~$13$; 
let $D_2$ be configuration space of the subgraph with vertices $2,3,5,6$, 
which is determined by the dihedral angle at the edge~$23$; 
let $D_{12}$ be configuration space of the subgraph with vertices $3,5,6$, 
which is determined by the angle at~$3$. 
Then by the same argument as before, the natural projections $D_2 \longrightarrow D_{12}$ and $D_1 \longrightarrow D_{12}$ are isomorphisms. 
By the simplicity of the edge~$13$, it follows that the projection $C_1 \longrightarrow D_1$ is an isomorphism. 
Because the projections commute, it follows that the projection $C_{12} \longrightarrow D_{12}$ is also
an isomorphism. 
Hence the angle at~$3$ is simple and so the quadrilateral~$12$ is a parallelogram or an antiparallelogram.

\begin{figure}[ht]
  \centering
  \includegraphics[height=5cm]{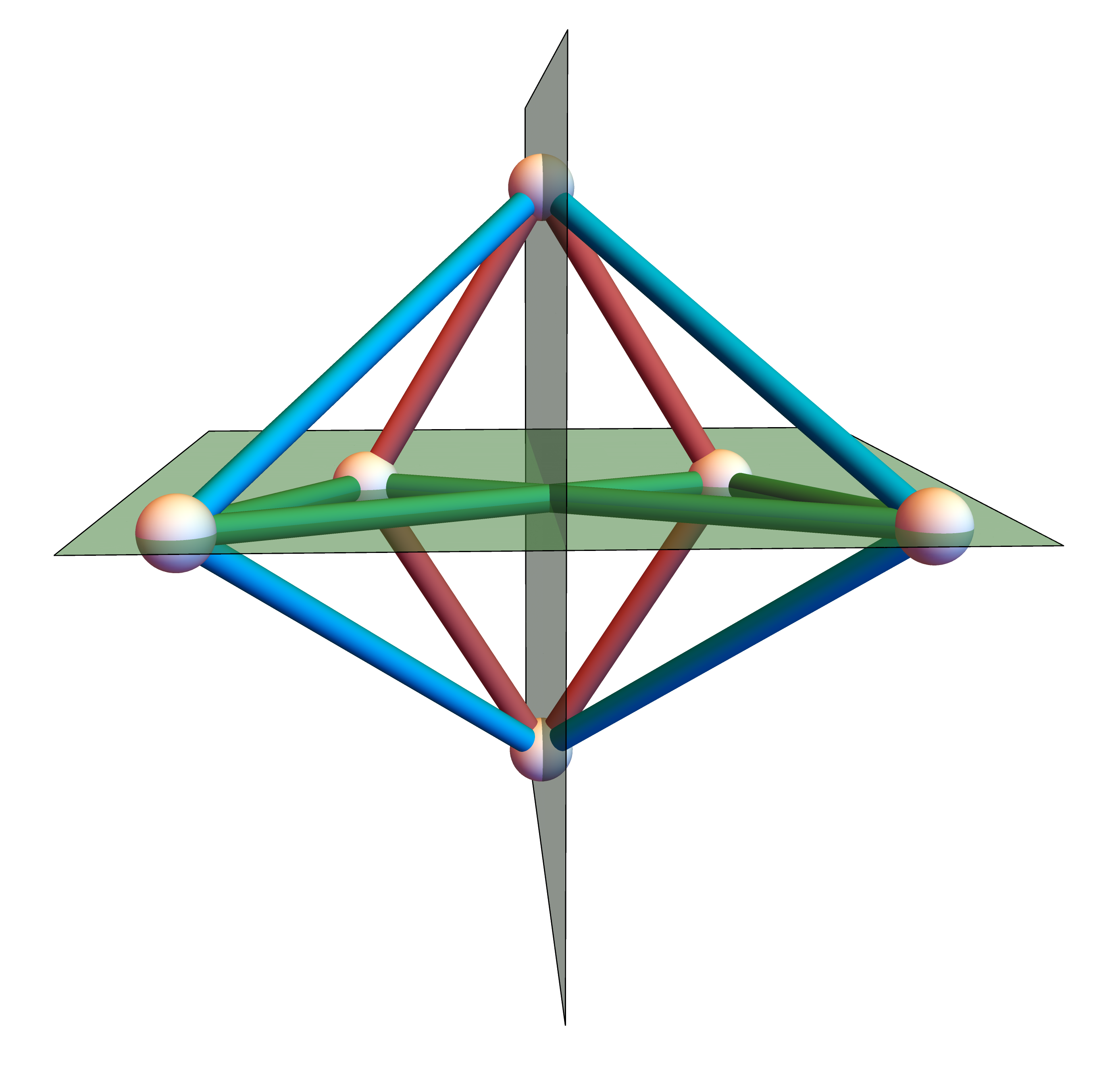}
  \caption{An example of a motion which is an instance of all three Bricard types.}
  \label{figure:alltypes}
\end{figure}

\begin{remark}
The case where the quadrilateral~$12$ is a parallelogram is probably purely hypothetical, but the case where the quadrilateral~$12$ is 
an antiparallelogram really does exist. An example is an octahedron with edge lengths 
\[ \ell_{13}=\ell_{14}=\ell_{23}=\ell_{24}=20, \ell_{15}=\ell_{16}=\ell_{25}=\ell_{26}=13, \ell_{35}=\ell_{46}=11, \ell_{36}=\ell_{45}=21 . \]
This flexible octahedron is an instance of all three Bricard types. It has two plane symmetries, one by the plane through $3$, $4$, $5$, $6$, and another
by the plane intersecting orthogonally in the symmetry line of the antiparallelogram, which makes it plane-symmetric. The line reflection
making it line-symmetric is the composition of the two plane reflections. See Figure~\ref{figure:alltypes} for an example.
\end{remark}

\bigskip
\bigskip
\textsc{(MG) International School for Advanced Studies/Scuola Internazionale Superiore di
Studi Avanzati (ISAS/SISSA), Via Bonomea 265, 34136 Trieste, Italy}\\
Email address: \texttt{mgallet@sissa.it}

\textsc{(GG) Johann Radon Institute for Computation and Applied Mathematics (RICAM), Austrian
Academy of Sciences}\\
Email address: \texttt{georg.grasegger@ricam.oeaw.ac.at}

\textsc{(JL, JS) Johannes Kepler University Linz, Research Institute for Symbolic Computation (RISC)}\\
Email address: \texttt{jan.legersky@risc.jku.at}, \texttt{jschicho@risc.jku.at}

\textsc{(JL) Department of Applied Mathematics, Faculty of Information Technology, Czech Technical University in Prague}

\end{document}